%% file: Lipschitz-paper.tex
\documentclass[11pt,a4paper]{amsart}

\usepackage{comment}

\usepackage{framed}
\usepackage{xargs}
\usepackage[dvipsnames]{xcolor}
\usepackage[colorinlistoftodos,prependcaption,textsize=tiny]{todonotes}
\newcommandx{\commentdaniele}[2][1=]{\todo[linecolor=Blue,backgroundcolor=Blue!25,bordercolor=Blue,#1]{D: #2}}
\newcommandx{\commentvalentina}[2][1=]{\todo[linecolor=Green,backgroundcolor=Green!25,bordercolor=Green,#1]{V: #2}}
\usepackage[colorlinks, linktocpage=true]{hyperref}
\usepackage{rotating}
\usepackage{psfrag}
\usepackage{epigraph}
\usepackage[english]{babel}
\usepackage{epsfig}
\usepackage{epstopdf}
\usepackage{faktor}
\usepackage{amsfonts} 
\usepackage{mathrsfs} 
\usepackage{graphicx, color, booktabs, xypic}
\usepackage{caption}
\usepackage{subcaption}
\usepackage{amssymb, array}  
\usepackage{stmaryrd}
\usepackage{MnSymbol}
\usepackage{amsmath} 
\usepackage{amscd, amsthm, dsfont} 
\usepackage{tabu}

\usepackage{thmtools}
\usepackage{thm-restate} 

\newcommand{\R}{\ensuremath{\mathbb{R}}}
\newcommand{\bbS}{\ensuremath{\mathbb{S}}}

\newcommand{\T}{\ensuremath{\mathrm{Teich}}}

\newcommand{\G}{\ensuremath{\mathcal{G}}}

\newcommand{\PP}{\ensuremath{\mathcal{P}}}
\newcommand{\Q}{\ensuremath{\mathcal{Q}}}
\newcommand{\TT}{\ensuremath{\mathcal{T}}}
\newcommand{\HH}{\ensuremath{\mathcal{H}}}
\newcommand{\FF}{\ensuremath{\mathcal{F}}}

\newcommand{\C}{\ensuremath{\mathcal{C}}}
\newcommand{\A}{\ensuremath{\mathcal{A}}}

\newcommand{\Lip}{\ensuremath{\mathrm{Lip}}}

\newcommand{\vleq}{\text{\rotatebox[origin=c]{-90}{$\leq$}}}
\newcommand{\veq}{\text{\rotatebox[origin=c]{-90}{$=$}}}

\newtheorem{theorem}{Theorem}[section]
\newtheorem{lemma}[theorem]{Lemma}
\newtheorem{proposition}[theorem]{Proposition}
\newtheorem{conjecture}[theorem]{Conjecture}

\theoremstyle{definition}

\newtheorem{definition}[theorem]{Definition}

\newtheorem{notation}[theorem]{Notation}

\theoremstyle{remark}

\newtheorem{remark}[theorem]{Remark}

\usepackage[hang,flushmargin]{footmisc}

\usepackage{enumitem}

\makeatother

\title{Generalized stretch lines for surfaces with boundary}

\author{Daniele Alessandrini} 
\address{Department of Mathematics, Columbia University, 2990 Broadway, New York, USA.}
\email{daniele.alessandrini@gmail.com}

\author{Valentina Disarlo} 
\address{Mathematisches Institut, Universit\"at Heidelberg, INF 205, Heidelberg, Germany.}
\email{vdisarlo@mathi.uni-heidelberg.de}

\begin{document}

\begin{abstract}
In 1986 William Thurston introduced the celebrated (asymmetric) Lipschitz distance on the Teichm\"uller space of closed or punctured surfaces. We extend his theory to the Teichm\"uller space of surfaces with boundary endowed with the arc distance. We construct a large family of geodesics for the Teichm\"uller space of a surface with boundary, generalizing Thurston's stretch lines. We prove that the Teichm\"uller space of a surface with boundary is a geodesic and Finsler metric space with respect to the arc distance. As a corollary, we find a new class of geodesics in the Teichm\"uller space of a closed surface that are not stretch lines in the sense of Thurston. 
\end{abstract}

\maketitle



\input{01-Introduction}

\input{02-Background}

\input{03-Laminations}

\input{05-Average}

\input{06-Geometric_Pieces}

\input{07-Boundary_Block}

\input{08-Triangulated_Surface}

\input{09-Generalized_Stretch_Lines}

\input{10-Geometry}

\bibliographystyle{amsplain}   
\bibliography{references-new}

\end{document}

%% file: 01-Introduction.tex
\section{Introduction}

In this paper we will study the geometry of the Teichm\"uller space of an oriented surface of finite type with non-empty boundary when it is endowed with the \emph{arc distance}. This is an asymmetric distance, which generalizes  the celebrated \emph{Thurston's asymmetric distance} on the Teichm\"uller space of a closed surface defined by William P. Thurston.

\subsection{Thurston's theory for closed or punctured surfaces} 
In \cite{Thurston} Thurston defines two asymmetric distances on the Teichm\"uller space of a closed or punctured surface, which naturally mimics the Teichm\"uller distance in a hyperbolic setting: the distance $d_{Th}$, which encodes the changes in the 
length spectrum of simple closed curves; the Lipschitz distance $d_{Lh}$, which measures the optimal Lipschitz constant of a homeomorphism isotopic to the identity (for precise definitions see Section \ref{subsec:5 functionals}). Following the analogy with the 
Teichm\"uller distance, Thurston constructs a family of paths, called \emph{stretch lines}, which are geodesics for both distances. Using these paths, Thurston proves that the two distances always coincide and turn the Teichm\"uller space into a geodesic 
Finsler metric space (see Section \ref{subsec:empty boundary}). Understanding the structure and behavior of \emph{all} the geodesics for Thurston's distance is still an open problem. The geometry of Thurston's distance of the Teichm\"uller space of closed or punctured surfaces was further studied by many authors, including Bonahon \cite{Bonahon_pleated, Bonahon_holder}, 
Papadopoulos \cite{Papadopoulos}, Th\'eret \cite{TheretThese}, Walsh \cite{Walsh}, Dumas-Lenzhen-Rafi-Tao \cite{Dumas-Lenzhen-Rafi-Tao}, Lenzhen-Rafi-Tao \cite{Lenzhen-Rafi-Tao1, Lenzhen-Rafi-Tao2}, Choi-Rafi 
\cite{Choi-Rafi}.

\subsection{The theory for surfaces with boundary}
In this paper we study similar asymmetric distances on the Teichm\"uller space $\T(S)$ of a surface $S$ with non-empty boundary. The case of surfaces with boundary is particularly interesting, see for example the beautiful applications by Gu\'eritaud-Kassel \cite{Kassel-Gueritaud} 
on proper affine actions of free groups and Margulis spacetime (notice that the presence of a non-empty boundary is crucial in their work). In the case of surfaces with boundary Thurston's original formula $d_{Th}$ does not give a distance 
anymore (see Parlier \cite{Parlier}, Papadopoulos-Th\'eret \cite{PapTher1} and Section \ref{subsec:non-empty boundary}). In spite of this, Liu-Papadopoulos-Su-Th\'eret \cite{LiuPapSuTh} defined a new distance, the \emph{arc distance} $d_A$,  which considers the length spectrum 
of simple closed curves and simple proper arcs orthogonal to the boundary (for details see Section \ref{subsec:5 functionals}). A few examples of geodesics for $d_A$ were given by Papadopoulos-Th\'eret \cite{PapTher3} and Papadopoulos-Yamada \cite{PapYam}. Alessandrini-Liu-Papadopoulos-Su \cite{ALPS} studied the close relationship between the arc distance and Thurston's compactification. 

The arc distance is the main object of study in this paper. Motivated by \cite{PapTher3, PapYam, ALPS},  we will study  the property of the metric space $(\T(S), d_A)$. We will prove that it is a geodesic Finsler metric space. We will construct some special paths, the \emph{generalized stretch lines}, which mimic the properties of Thurston's original stretch lines in this new setting. Our results will be useful in work in progress by Calderon-Farre to produce their shear-shape coordinates for Teichm\"uller space. In their recent preprint \cite{PapHuang} Huang-Papadopoulos study similar questions in the special case of the one-holed torus in the different setting of Teichm\"uller spaces with fixed boundary length.  All the results in this paper were announced at the Oberw\"olfach conference ``New Trends in Teichm\"uller Theory and Mapping Class Groups'' in 2018 (see the report \cite{Oberwolfach}). 
    
\subsection{Our results}
Let $S$ be an orientable surface of finite-type with non-empty boundary. In this paper we introduce a new asymmetric distance on $\T(S)$, denoted by $d_{L\partial}$, which measures the optimal Lipschitz constant of a continuous map preserving $\partial S$ isotopic to the identity. We now have three distances on $\T(S)$, which satisfy the following inequalities (for precise definitions and statements, see Section \ref{subsec:5 functionals}):
$$d_A \leq d_{L\partial} \leq d_{Lh}\,. $$ 

In analogy with Thurston's theory, we will construct a large family of geodesics for the two distances $d_A$ and $d_{L\partial}$, which we call generalized stretch lines. For any two points on the same generalized stretch line we will construct an optimal Lipschitz map, which we call a generalized stretch map. The constructions of generalized stretch lines and stretch maps are the most important contributions of this paper. This construction can be summarized in the following statement (the terminology will be introduced later in Sections \ref{sec:background}, \ref{sec:laminations} and \ref{sec:stretch_1}). 

\begin{restatable}{theorem}{thurstonstretch}
\label{theorem:S}
Let $S$ be a surface with non-empty boundary and fix $X \in \T(S)$. For every maximal lamination $\lambda$ on $X$ and for every $t \geq 0$ there exists $X_\lambda^t \in \T(S)$ and a Lipschitz map $\Phi^t:X \to X_\lambda^t$, called \emph{generalized stretch map}, with the following properties: 
\begin{enumerate} 
\item $X_\lambda^0 = X$;
\item $\mathrm{Lip }(\Phi^t) = e^t$;
\item $\Phi^t(\partial X) = \partial X_\lambda^t$; 
\item $\Phi^t$ stretches the arc length of the leaves of $\lambda$ by the factor $e^t$; 
\item for every geometric piece $\mathcal{G}$ in $X \setminus \lambda$, the map $\Phi^t$ restricts to a generalized stretch map $\phi^t: \mathcal{G} \to \mathcal{G}_t$ as described in Lemmas \ref{lemma:triangle}, \ref{lemma:quad};
\item if $\lambda$ contains a non-empty measurable sublamination, we have
$$\mathrm{Lip}(\Phi^t) = \mathrm{min }\{ \mathrm{Lip}(\psi) ~|~\psi\in \mathrm{Lip}_0(X, X_\lambda^t),  \psi(\partial X) \subset \partial X_\lambda^t\}, $$
where $\mathrm{Lip}_0(X, Y)$ is the set of all Lipschitz maps 
homotopic to the identity.
\end{enumerate}
\end{restatable}

\begin{restatable}{corollary}{thurstonstretchline}
\label{corollary:S}
For every $X \in \T(S)$ and every maximal lamination $\lambda$ on $X$, if $\lambda$ contains a non-empty measurable sublamination then the \emph{generalized stretch line}
\begin{align*} 
s_{X,\lambda}: \R_{\geq 0} &\longrightarrow \T(S) \\ 
t &\mapsto X_\lambda^t 
\end{align*} 
is a geodesic path parametrized by arc-length for both $d_A$ and $d_{L\partial}$. 
\end{restatable}

Our construction presents new challenges when compared with Thurston's work. In Thurston's case of closed or punctured surfaces, every maximal lamination decomposes the surface into ideal triangles. Thurston constructs the stretch map between two ideal triangles explicitly via the horocyclic foliation. 
In the case of surfaces with boundary a maximal lamination decomposes the surface in geometric pieces of four different types (see Figure \ref{geom_pieces}).  Unlike Thurston \cite{Thurston}, we do not construct explicit maps between the geometric 
pieces. Instead, we use a trick of its own interest, which allows to ``average'' two Lipschitz maps. Our \emph{average map} will be a Lipschitz map whose Lipschitz constant is bounded above by the average of the two Lipschitz constants. Our construction is obtained by adapting a result of Gu\'eritaud-Kassel \cite{Kassel-Gueritaud}. 
Using generalized stretch lines, we prove: 

\begin{restatable}{theorem}{geodesicspace} \label{thm:geodesic space}
The space $(\T(S), d_A)$ is a geodesic metric space. Every two points $X,Y \in \T(S)$ can be joined by a segment that is geodesic for both $d_A$ and $d_{L\partial}$ and is a finite concatenation of generalized stretch segments.
\end{restatable}

\begin{restatable}{corollary}{finslerspace} \label{cor:finsler space}
The arc distance $d_A$ is induced by  a Finsler metric on $\T(S)$.
\end{restatable}

We find that $d_A$ and $d_{L\partial}$ coincide, this completes our generalization of Thurston's \cite[Theorem 8.5]{Thurston}. 
\begin{restatable}{corollary}{distancescoincide} \label{cor:distances coincide}
Given $X,Y \in \T(S)$, there exists a continuous map $\phi\in \mathrm{Lip}_0(X,Y)$, with $\phi(\partial X) \subset \partial Y$ and with optimal Lipschitz constant such that $\log(\mathrm{Lip}(\phi)) = d_A(X,Y).$ 
In particular, we have 
$$d_A(X,Y) = d_{L\partial}(X,Y).$$
\end{restatable}
As a byproduct of our constructions, we also find new geodesics for the Teichm\"uller space of closed or punctured surfaces endowed with Thurston's distance. Indeed, Liu-Papadopoulos-Su-Th{\'e}ret \cite{LiuPapSuTh} proved that the doubling map 
$$\jmath: (\T(S), d_A) \ni X \hookrightarrow X^d \in (\T(S^d), d_{Th})$$
is an isometry. By doubling our generalized stretch lines, we can construct many new geodesics for $(\T(S^d),d_{Th})$ that lie completely in the submanifold of \emph{symmetric hyperbolic structures}.  

\begin{restatable}{corollary}{geodesicembedding} \label{cor:geodesic}
The map $(\T(S), d_{A}) \hookrightarrow (\T(S^d), d_{Th})$ is a geodesic embedding.
\end{restatable}

Notice that Thurston's construction of stretch lines in general breaks the symmetry of hyperbolic structures, see for instance the examples by Th\'eret \cite{TheretThese}. Our construction, instead, provides new geodesics which preserve symmetric hyperbolic structures. 

\begin{restatable}{corollary}{newgeodesics} \label{cor:new geodesics}
Let $X \in \T(S)$ and let $\lambda$ be a maximal lamination of $X$ containing a measurable sublamination with at least one leaf orthogonal to the boundary of $X$. Then, the line $t \mapsto 
(X_\lambda^t)^d \in \T(S^d)$ is a geodesic for $(\T(S^d), d_{Th})$ that is not a stretch line in the sense of Thurston \cite{Thurston}.
\end{restatable}

\input{04-Sketch}

\subsection{Open problems}
Our work leads us to conjecture that the three natural distances $d_A, d_{L\partial}, d_{Lh}$ on $\T(S)$ are all equal:   
\begin{conjecture} \label{conj:distances}
For every $X,Y \in \T(S)$ we have: 
$$d_A (X, Y) = d_{Lh} (X,Y)~.$$ 
\end{conjecture} 
Notice that our generalized stretch map $\Phi^t$ is a homeomorphism if and only if its restriction to each geometric piece $\Phi^t|_{\mathcal G}: \mathcal G \to \mathcal G^t$ is also a homeomorphism. Since our construction of the maps is  not explicit, we cannot  
tell whether they are injective. 

\subsection{Organization of the paper}
This paper is organized as follows. 
In Section \ref{sec:background}, we introduce the main definitions that we use throughout the paper and we give a more detailed account of the theory of asymmetric distances on Techm\"uller spaces. 
In Section \ref{sec:laminations}, we introduce the notion of geodesic laminations, maximal and measured laminations for surfaces with boundary.
In Section \ref{sec:sketch} we give a sketch of the proof of the main theorem (Theorem \ref{theorem:S}), this is a preview of the following sections.
In Section \ref{sec:Lip} we describe an averaging procedure for Lipschitz maps between convex hyperbolic surfaces.
In Section \ref{sec:stretch_1} we use it  to construct optimal Lipschitz maps between the geometric pieces. 
After these preliminary sections, we describe the construction of our generalized stretch lines. 
In Sections \ref{sec:bdry_block} and \ref{sec:triangulated_surface} we construct some auxiliary surfaces (the boundary block and the triangulated surface) and we describe how to stretch them. 
In Section \ref{sec:gsl} we glue the stretched boundary block and part of the stretched triangulated surface together in a suitable way, in order to construct the generalized stretch lines. This will prove Theorem \ref{theorem:S}. 
In Section \ref{sec:geometryT} we prove all the other results stated in the introduction. 

\subsection{Acknowledgements} 

The authors thank Maxime Fortier Bourque, Athanase Papadopoulos, Kasra Rafi, Weixu Su and Dylan Thurston for enlightening conversations.
The authors also thank Indiana University Bloomington for the nice working environment provided during Spring 2016. The authors acknowledge support from U.S. National Science Foundation grants DMS 1107452, 1107263, 1107367 ``RNMS: GEometric structures And Representation varieties'' (the GEAR Network). The second author acknowledges support by the European Research Council under Anna Wienhard's ERC-Consolidator grant 614733 (GEOMETRICSTRUCTURES) and by the Deutsche Forschungsgemeinschaft (DFG, German Research Foundation) under Germany's Excellence Strategy EXC-2181/1 - 390900948 (the Heidelberg STRUCTURES Cluster of Excellence) and under the Priority Program SPP 2026
``Geometry at Infinity'' (DI 2610/2-1).

%% file: 04-Sketch.tex
\subsection{Sketch of the proof of Theorem \ref{theorem:S}}   \label{sec:sketch}


Let $X \in \T(S)$ and $\lambda$ be a maximal lamination on $X$. We want to construct a generalized stretch line starting from $X$ and directed by $\lambda$, i.e. for every $t\geq 0$ we want to construct $X_\lambda^t \in \T(S)$ satisfying the properties of Theorem \ref{theorem:S}. 

\subsubsection{Geometric pieces}
As a first step we characterize the connected components of $X \setminus \lambda$, that is, the geometric pieces. We will see in Proposition \ref{prop:lamination} that there are only four types of pieces (see Figure \ref{geom_pieces}). 

Given a geometric piece $\mathring{\G} \subset X \setminus \lambda$ we define a suitable generalized stretch map $\phi^t : \G \to \G^t$ from the original piece $\G$ to its ``stretched'' analogue $\G^t$. The map has optimal Lipschitz constant $\Lip(\phi^t) = e^t$. When $\G$ is an ideal triangle we use the homeomorphism defined by Thurston. In the other cases we use an implicit construction, which generalizes an argument by Gu\'eritaud-Kassel \cite{Kassel-Gueritaud} (see Section \ref{sec:Lip} and \ref{sec:stretch_1}). 

\subsubsection{Decomposition of \texorpdfstring{$X$}{X}} 

If $\lambda$ has no leaves that hit the boundary of $X$, all the geometric pieces are triangles, and Theorem \ref{theorem:S} follows by Thurston \cite{Thurston}. We are interested in the case where at least one leaf of $\lambda$ is orthogonal to $\partial X$.   
We define the boundary  block of $\lambda$ in $X$ in Section \ref{subsec:def boundary block} as the (possibly disconnected) subsurface given by  the union of all the non-triangular geometric pieces (Figure \ref{BCdecomposition}):  
\[ B = \bigsqcup \{ \mathcal G_i ~| ~\mathcal G_i \mbox{ is a geometric piece that is not an ideal triangle }\} \subseteq X\,.\]
The boundary block is a complete hyperbolic surface of finite volume, whose boundary might be non compact. It is equipped with a finite maximal lamination $\lambda_B \subset \lambda$. The boundary $\partial B$ will contain a finite union of \emph{cycles} $c_j$, each determining a \emph{crown} $C_j$ as in Figures \ref{Thurston:cycles} and \ref{crowns} (respectively Definitions \ref{def:cycle} and \ref{def:crown}).
Denote the union of all such crowns by $C$ and the complement of $C$ in $B$ by $B_C:= \overline{B \setminus C} \subset B$. Denote by $X_C := \overline{X \setminus B_C} \subset X$ the complement of $B_C$ in $X$, defined in Section \ref{subsec:def boundary block}, see Figure \ref{BCdecomposition}. The surface $X_C$ is equipped with a lamination $\lambda_{X_C} \subset \lambda$ where no leaf is orthogonal to $\partial X_C$. We have the following commutative diagram, where $\lambda = \lambda_B \cup \lambda_{X_C}$.
$$\xymatrix{
   & (B, \lambda_B) \ar[dr] & \\ 
C \ar[ur] \ar[dr] & & (X, \lambda)  \\ 
& (X_C, \lambda_{X_C}) \ar[ur] &  
}$$
Every arrow is the canonical inclusion, which is also a Riemannian isometry.  

\subsubsection{Strategy to define \texorpdfstring{$X_\lambda^t$}{X lambda t}} 

For $t\geq 0$ we construct suitable complete hyperbolic surfaces $B^t$, $C^t$ and $(X_C)^t$ homeomorphic to $B$, $C$ and $X_C$ respectively (see Sections \ref{sec:stretchB}, \ref{sec:S_A} and \ref{sec:gsl}). The new surfaces come with preferred Riemannian isometries $\iota^t: C^t \hookrightarrow B^t$ and $h^t: C^t \hookrightarrow X_C^t$.  We then define $X_\lambda^t$ as follows:
$$X_\lambda^t:= B^t \bigsqcup (X_C)^t/\sim~,$$ 
where $\iota^t(z) \sim h^t(z)$ for every $z \in C^t$. The quotient projection $\pi: B^t \bigsqcup (X_C)^t \to X_\lambda^t$ restricts to Riemannian isometries on $(X_C)^t$ and $B^t$ (Proposition \ref{S^t}):
$$\xymatrix{
   & B^t \ar[dr]^{\pi_|} & \\ 
C^t \ar[ur]^{\iota^t} \ar[dr]_{h^t} & & X_\lambda^t  \\ 
& (X_C)^t \ar[ur]_{\pi_|} &  
}~.$$
The generalized stretch map $\Phi^t: X \to X_\lambda^t$ is defined by glueing together suitable stretch maps $\beta^t: B \to B^t$ and $\psi^t$ from an open dense subset of $X_C$ to $(X_C)^t$, with the required properties. 
The details about the construction of $(B^t, \beta^t)$ are given in Section \ref{sec:stretchB}. The details about the construction of $((X_C)^t, \psi^t)$ are given in Section \ref{subsec:gsl} and Section \ref{sec:S_A}. The details about how to glue the maps 
$\beta^t$ and $\psi^t$ are discussed in Section \ref{subsec:stretch maps}.

%% file: 02-Background.tex
\section{Background}    \label{sec:background}

In this section, we introduce the main definitions that we use throughout the paper and we give an account of the theory of asymmetric distances on Techm\"uller spaces.  
\subsection{Hyperbolic surfaces}

We start by recalling some basic definitions about hyperbolic surfaces. We denote the \emph{hyperbolic plane} by
$$\mathbb{H}^2 = \{z\in \mathbb{C} \mid \mathrm{Im}(z) >0 \}, $$
endowed with the Riemannian metric 
$$g_{\mathbb{H}^2} = \frac{dx^2 + dy^2}{y^2}, \mbox{ with } z = x + iy. $$\
The \emph{hyperbolic half-plane} is the subset
$$ \{z\in \mathbb{C} \mid \mathrm{Re}(z) \geq 0, \mathrm{Im}(z) >0 \} $$
where the positive $y$-axis is its geodesic boundary. 

\begin{definition}[Hyperbolic surface]
A \emph{hyperbolic surface} is a Riemannian manifold $X$ (possibly with boundary) where every point has a neighborhood isometric to an open subset of the hyperbolic half-plane. A \emph{complete hyperbolic surface} is a hyperbolic surface that is complete as a metric space.
\end{definition}

\begin{definition}[Convex hyperbolic surface] \label{def:convex surface}
A \emph{convex hyperbolic surface} is a connected hyperbolic surface whose universal covering is isometric to a convex subset of the hyperbolic plane. 
\end{definition}

Note that the boundary of a complete hyperbolic surface is a union of geodesics (circles or infinite geodesic lines). Moreover, the connected components of a complete hyperbolic surface are always convex. Examples of complete hyperbolic surfaces are ideal polygons in the hyperbolic plane. 


\begin{definition}[Finite hyperbolic surface]
A \emph{finite hyperbolic surface} is a complete hyperbolic surface with finite volume and compact boundary.
\end{definition}

For such finite hyperbolic surfaces, every boundary component is a closed geodesic, topologically a circle, and every puncture is isometric to a cusp.

\begin{definition}[Local isometry]
A \emph{local isometry} between two hyperbolic surfaces $X$ and $Y$ is a local diffeomorphism such that the pull-back of the metric on $Y$ is equal to the metric on $X$. 
\end{definition}

\begin{remark}
In this paper, we will need to consider many local isometries that are 1-1 but still are not isometric embeddings in the sense of metric spaces, i.e. they do not preserve the distances between pairs of points. This is because their image is not a convex subset of the target surface. 
\end{remark}

\subsection{Teichm\"uller space} In this paper we will denote by $S$ an orientable surface of finite type of genus $g$, with $b\geq 0$ compact boundary components and $p\geq 0$ punctures. The boundary of $S$ will be denoted by $\partial S$. We assume the Euler characteristic $\chi(S) = 2-2g-b-p $ to be negative. 
\begin{definition}[Hyperbolic structure]
A \emph{hyperbolic structure} on $S$ is a pair $(X,m)$, where $X$ is a finite hyperbolic surface, and $m:S \rightarrow X$ is a homeomorphism, called the \emph{marking}. 
\end{definition}
\begin{definition}[Teichm\"uller space]
The \emph{Teichm\"uller space} $\T(S)$ is the space of all hyperbolic structures on $S$ up to isometries that commute with the markings (up to homotopy). We will denote an element $[(X,m)] \in \T(S)$ by $X$ for short.
\end{definition}

The Teichm\"uller space $\T(S)$ is diffeomorphic to $\R^{6g-6+2p+3b}$. 
Let $S^d$ be the surface obtained doubling $S$ along its boundary, and let $\sigma: S^d \to S^d$ be the associated involution. A hyperbolic structure on $S$ can be equivalently defined as a hyperbolic structure on $S^d$ whose isometry group contains $\sigma$. 
\begin{definition}[Doubling embedding]
If $X$ is a hyperbolic structure on $S$, its \emph{double} $X^d$ is the hyperbolic structure on $S^d$ obtained by doubling $X$. This gives an embedding 
$$\T(S) \ni X \hookrightarrow X^d \in \T(S^d).$$
\end{definition}
 
\subsection{Curves and arcs} \label{length}

A simple closed curve in $S$ is \emph{trivial} if is either null-homotopic or homotopic to a puncture of $S$. We will denote by $\C$ the set of homotopy classes of non-trivial simple closed curves on $S$, and by $\mathcal B$ the boundary components of $\partial S$.  We recall that for every $X \in \T(S)$ and for every $\gamma \in \C$, there is a unique $X$-geodesic curve in the homotopy class $\gamma$, which is the shortest curve in $\gamma$. We will define the \emph{length} $\ell_X(\gamma)$ to be the length of this geodesic curve.

A \emph{proper arc} in $S$ is a continuous map $\alpha:[0,1]\rightarrow S$ with $\{\alpha(0), \alpha(1)\} \subset \partial S$. Our arcs are \emph{unoriented}, i.e. we will consider $\alpha(t)$ equivalent to $\alpha(1-t)$ 
A proper arc is \emph{simple} if it is injective.  Two proper arcs are \emph{properly homotopic} if they are connected by a homotopy where the extremes of the arcs never leave $\partial S$ at any time. 
\begin{definition}[Essential arc]
An \emph{essential arc} is a proper arc that is not properly homotopic to a proper arc contained in $\partial S$. We will denote by $\A$ the set of proper homotopy classes of essential simple arcs. Recall that the elements of $\A$ are unoriented. 
\end{definition}
It is well-known that for every $X \in \T(S)$ and for every $\alpha \in \A$, there is a unique $X$-geodesic arc in the proper homotopy class $\alpha$ which is orthogonal to $\partial S$. This arc is the shortest in its proper homotopy class, its length is denoted by $\ell_X(\alpha)$.

\subsection{Five functionals}    \label{subsec:5 functionals}
The length functions $\ell_X(\cdot)$ of curves and arcs can be used to compare two hyperbolic structures and, in some cases, define distances on Teichm\"uller spaces. We will be interested in the following functionals:
\begin{align}
d_{Th}(X,Y) &= \sup_{\gamma \in \mathcal B \cup \C} \log \frac{\ell_Y(\gamma)}{\ell_X(\gamma)}, \\
d_{A}(X,Y) &= \sup_{\delta \in \mathcal{A}\cup \mathcal B \cup \mathcal C} \log \frac{\ell_Y(\delta)}{\ell_X(\delta)}. \label{form:distanceA}
\end{align}

Another natural way to compare two elements $X,Y \in \T(S)$ is to consider Lipschitz maps between them. Let $\mathrm{Lip}_0(X,Y)$ be the set of Lipschitz maps between $X$ and $Y$ that commute with the markings up to homotopy. Denote by $\mathrm{Lip}(\phi)$ the Lipschitz constant of a map $\phi$. We will consider the functionals: 
\begin{align}
d_L(X,Y) &= \inf\{\log \mathrm{Lip}(\phi) \ |\  \phi\in \mathrm{Lip}_0(X,Y)   \}\,; \\
d_{L\partial}(X,Y) &= \inf\{\log \mathrm{Lip}(\phi) \ |\  \phi\in \mathrm{Lip}_0(X,Y), \phi(\partial X) \subset \partial Y  \} \,; \\ 
d_{Lh}(X,Y) &= \inf\{\log \mathrm{Lip}(\phi) \ |\  \phi\in \mathrm{Lip}_0(X,Y), \phi \text{ is a homeomorphism}   \} \,. 
\end{align}

It is immediate that the above five functionals satisfy the following inequalities:
\begin{equation*} 
\begin{tabu}{ccccc}
d_{Th}(X,Y) & \leq & d_L(X,Y)            & \leq & d_{Lh}(X,Y)\\
   \vleq    &      &   \vleq             &      &   \veq\\
d_A(X,Y)    & \leq & d_{L\partial}(X,Y)  & \leq & d_{Lh}(X,Y) ~.
\end{tabu}
\end{equation*}
It is not difficult to see that they all satisfy the triangular inequality
$$d_*(X,Z) \leq  d_*(X,Y) + d_*(Y,Z) ~,$$
and they are not symmetric 
$$\exists X, Y: d_\star (X, Y) \neq d_{\star}(Y, X)\,.$$
A method to produce such $X,Y$ is given in \cite[Section 2]{Thurston}, and works also for surfaces with boundary. 
In the following we will discuss when the axiom of positivity 
$$d_\star (X, Y) \geq 0 \mbox{ and } d_\star(X, Y) = 0 \Leftrightarrow X = Y ~$$ 
holds, that is, the functionals actually define asymmetric distances. We will see that the answer depends on whether $\partial S$ is empty or not. 

\subsection{Closed or punctured surfaces}   \label{subsec:empty boundary} The case of closed or punctured surfaces was the case originally studied by Thurston \cite{Thurston}. He introduced the functionals $d_{Th}$ and $d_{Lh}$, and proved that they satisfy the positivity axiom.
 
One of the main results of \cite{Thurston} is that given $X,Y\in \T(S)$, there exists a homeomorphism $\phi\in \mathrm{Lip}_0(X,Y)$ such that $\log(\mathrm{Lip}(\phi)) = d_{Th}(X,Y)$. This implies 
$$d_{Th} = d_L = d_{Lh} ~.$$
 This distance is usually called \emph{Thurston's asymmetric distance}, or the \emph{Lipschitz distance}. 
A crucial step in the proof is the construction of a special family of lines in $\T(S)$, the \emph{stretch lines}, which are geodesics for the three distances. Given two points on the same stretch line, there is an optimal Lipschitz homeomorphism between them, the so-called \emph{stretch map}. Using these techniques, given two points $X, Y \in \T(S)$ he constructed a geodesic segment between $X$ and $Y$ by concatenating a finite number of such stretch lines. This proves that the Teichm\"uller space with Thurston's asymmetric distance is a geodesic metric space. Thurston also  studied the infinitesimal behaviour of his distance. He proved that it agrees with the asymmetric distance associated with a certain Finsler metric on the Teichm\"uller space.

\subsection{Surfaces with boundary} \label{subsec:non-empty boundary}
In the case of surfaces with non-empty boundary, Thurston's functional $d_{Th}$ does not  
satisfy the axiom of positivity. Indeed, Parlier \cite{Parlier} found two points $X\neq Y \in \T(S)$ with $d_{Th}(X,Y)\leq 0$. Papadopoulos-Th\'eret \cite{PapTher1} found elements $X\neq Y \in \T(S)$ with $d_{Th}(X,Y) < 0$. 

The properties of the functional $d_{Th}$ for surfaces with boundary were studied in detail by Gu\'eritaud-Kassel \cite{Kassel-Gueritaud}, who also introduced the functional $d_L$. They proved that the two functionals are related as follows:
\begin{itemize}
\item if $d_{Th}(X,Y) \geq 0$, then $d_{Th}(X,Y) = d_L(X,Y)$; 
\item if $d_{Th}(X,Y) < 0$, then $d_{L}(X,Y) < 0$. 
\end{itemize} 
They give applications to the theory of affine actions on $\R^3$ and Margulis space-times. 

The functional $d_A$ was introduced by Liu-Papadopoulos-Th{\'e}ret-Su \cite{LiuPapSuTh}. They proved that $d_A$ satisfies the axiom of positivity, therefore it defines an asymmetric distance on $\T(S)$, which they called the \emph{arc distance}. They proved the following: 
\begin{proposition}[Liu-Papadopoulos-Th{\'e}ret-Su {\cite[Corollary 2.8]{LiuPapSuTh}}]  \label{prop:doubling}
The doubling map $(\T(S),d_A) \hookrightarrow (\T(S^d),d_{Th})$ is isometric.
\end{proposition}
This proposition will be useful for the present work: whenever possible, we use doubling arguments to reduce our questions to well understood questions about surfaces without boundary. In any case, we remark that it is not possible to construct many geodesics for the arc distance using just doubling arguments.  
Other properties of the distance $d_A$ were studied in \cite{ALPS,PapTher2,PapYam}. 

The functional $d_{L\partial}$ is introduced in this paper in order to interpolate between $d_A, d_{Lh}$: 
$$d_A(X,Y) \leq d_{L\partial}(X,Y) \leq d_{Lh}(X,Y)\,.$$
Here we prove that $d_A = d_{L\partial}$ (Corollary \ref{cor:distances coincide}), that the Teichm\"uller space with this distance is a geodesic metric space (Theorem \ref{thm:geodesic space}) and that this distance is induced by a Finsler metric (Corollary \ref{cor:finsler space}).

%% file: 03-Laminations.tex
\section{Geodesic laminations for surfaces with boundary}  \label{sec:laminations}

In this section, we will review the main definitions and results about geodesic laminations for surfaces with boundary. 

\subsection{Maximal laminations}

Let $X \in \T(S)$. A \emph{geodesic lamination} on $X$ is a (possibly empty) closed subset $\lambda \subset X$ that is a union of pairwise disjoint simple $X$-geodesics, called the \emph{leaves} of $\lambda$, satisfying the following additional condition: if a leaf of $\lambda$ intersects the boundary, then the leaf must be a boundary component, or it intersects the boundary orthogonally.  

It is well-known that geodesic laminations do not depend on the hyperbolic structure $X$, but only on the topological surface $S$. More precisely, given $X, Y \in \T(S)$, for every geodesic lamination $\lambda$ on $X$, there exists a unique geodesic lamination $\lambda'$ on $Y$ and a homeomorphism $f:X\rightarrow Y$ consistent with the markings that maps $\lambda$ to $\lambda'$. 
 In light of this, we will often consider geodesic laminations as topological objects on $S$, without specifying the underlying hyperbolic structure. 
 
A \emph{sublamination} $\lambda'$ of a geodesic lamination $\lambda$ is a closed subset $\lambda' \subset \lambda$ that is itself a geodesic lamination. A \emph{maximal lamination} is a geodesic lamination that is maximal with respect to inclusion, that is, it is not a sublamination of a strictly larger geodesic lamination. When $\partial X = \emptyset$ a maximal lamination decomposes $X$ into finitely many ideal triangles. When $\partial X \neq \varnothing$ a maximal lamination also decomposes $X$ into finitely many pieces, in general not always triangles. We will classify them in Proposition \ref{prop:lamination}. 

Let $\lambda$ be a geodesic lamination on $X$. The \emph{double} of $\lambda$ is the lamination $\lambda^d$ on $X^d$ obtained by doubling $\lambda$. Note that $\lambda^d$ is maximal if and only if $\lambda$ is maximal and does not contain leaves orthogonal to $\partial X$.  

\begin{definition}[Geometric piece]
A \emph{geometric piece} is a polygon in $\mathbb H^2$ as in Figure \ref{geom_pieces}:
\begin{itemize}
\item an ideal triangle, called \emph{triangular piece}; 
\item a right-angled quadrilateral with two consecutive ideal vertices, called \emph{quadrilateral piece}; 
\item a right-angled pentagon with one ideal vertex, called \emph{pentagonal piece}; 
\item a right-angled hexagon, called \emph{hexagonal piece}.  
\end{itemize}

\begin{figure}[htbp]
\begin{center}
\psfrag{s1}{$l_1$}
\psfrag{s2}{$l_2$}
\psfrag{s3}{$l_3$}
\psfrag{A1}{$a_1$}
\psfrag{A2}{$a_2$}
\psfrag{A3}{$a_3$}
\psfrag{A}{$a_1$}
\psfrag{a}[b]{$l_1$}
\psfrag{b}{$l_2$}
\psfrag{c}{$l_3$}
\includegraphics[width=8cm]{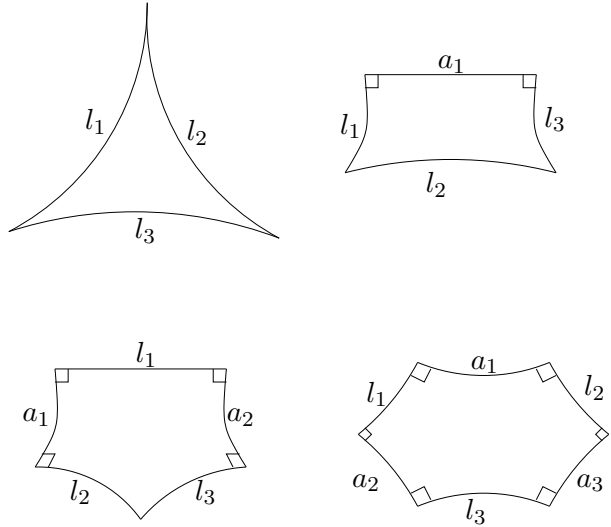}
\caption{The four geometric pieces: the edges $a_i$ correspond to segments in $\partial X$, the edges $l_i$ correspond to leaves of $\lambda$.}\label{geom_pieces}
\end{center}
\end{figure}
\end{definition}

\begin{proposition}\label{prop:lamination}
If $\lambda$ is a maximal lamination on $X$, then $X \setminus \lambda$ has $2|\chi(S)| = 4g-4+2p+2b$ connected components. Each connected component is locally isometric to the interior of a geometric piece, where the edges labeled $a_i$ correspond to segments in $\partial X$ and edges labeled $l_i$ correspond to leaves of $\lambda$ (see Figure \ref{geom_pieces}).
\end{proposition}
\begin{proof}
Let $C$ be a connected component of $X \setminus \lambda$. Note that $C$ contains no essential simple closed curve, otherwise we could extend $\lambda$ further by adding such a curve. Hence there are three possibilities for the topology of $C$: a pair of pants, a cylinder or a disk. If $C$ were a cylinder or a pair of pants, it would have $2$ or $3$ ends, but this is impossible: every end of $C$ contains a spike, or a segment of $\partial X$, or a simple closed curve in $X$ coming from a leaf of $\lambda$ or an entire boundary component in $\partial X$. If there is more than one end, we could add one more leaf to $\lambda$ joining two of them. 

We conclude that $C$ is topologically a disk, i.e. it is isometric to a hyperbolic polygon whose boundary contains segments in $\partial X$ or leaves of $\lambda$ and whose vertices are right-angled or ideal. Denote by $s$ the number of its ideal vertices and by $n$ the number of segments in $\partial X$. Given two spikes, a spike and a segment, or two segments, we can join them with a geodesic perpendicular to $\partial X$ that, by maximality, must lie in the boundary of $C$. Hence $s+n \leq 3$, and it must be $s+n=3$. Now we can see that the possibilities are $s=3, n=0$ (ideal triangle), $s=2, n=1$ (quadrilateral), $s=1, n=2$ (pentagon), $s=0, n=3$ (right-angled hexagon) as in Figure (\ref{geom_pieces}). 

Now we count the number of connected components in $X \setminus \lambda$. Since each one has area $\pi$ and the surface has area $2\pi|\chi(S)|$, we find  $2|\chi(S)|$ connected components. 
\end{proof}

\begin{proposition}
Every lamination $\lambda$ can be extended to a maximal lamination by adding finitely many leaves.
\end{proposition}
\begin{proof}
If $\lambda$ is not maximal then $X \setminus \lambda$ is a finite union of finite-area connected subsurfaces with boundary. Up to extending $\lambda$ with finitely many simple closed 
curves, we can assume 
that each connected component is either a disk, a cylinder or a pair of pants. Since the area of each piece is $\pi$, each of its boundary component is a finite polygonal with ideal or right-angled vertices. It can thus be further subdivided with at most finitely many simple essential arcs. 
\end{proof}

\subsection{Transverse measures} Let $\lambda$ be a geodesic lamination on $X$, and $k$ be an arc transverse to $\lambda$. A \emph{transverse isotopy} of $k$ is an isotopy that preserves the transversality of $k$ and such that the endpoints of $k$ either remain in the complement of $\lambda$ or remain in the same respective leaves during the entire isotopy. A \emph{transverse measure} on $\lambda$ is a function $\mu$ that associates to every arc in $X$ transverse to $\lambda$ a measure $\mu_k$ on $k$ satisfying the following conditions: 
\begin{itemize}
\item $\mu$ is invariant under transverse isotopies of arcs; 
\item if $k \subset k'$ then $\mu_{k}= {\mu_{k'}}|_{k}$;
\item $\mathrm{supp}(\mu_k)= k \cap \lambda$.
\end{itemize}

A geodesic lamination is \emph{compactly supported} if it is contained in a compact subset of $X$. A geodesic lamination is \emph{measurable} if it is compactly supported and it admits a transverse measure. If two sublaminations of a geodesic lamination are both measurable then their union is also measurable. The largest measurable sublamination of a geodesic lamination is called its \emph{stump}. The stump can possibly be empty. 

\subsection{Measured laminations} A \emph{measured lamination} is a pair given by a compactly supported geodesic lamination and a transverse measure on it. The space of measured laminations on $S$ is denoted by $\mathcal{ML}(S)$. It is a topological space homeomorphic to $\R^{6g-6+2p+3b}$ by \cite[Proposition 3.9]{ALPS}. Two measured laminations are \emph{projectively equivalent} if their underlying geodesic laminations coincide and their transverse measures differ by multiplication by a positive real number. A \emph{projectivized measured lamination} is a projective equivalence class of non-trivial measured laminations. The space of projectivized measured laminations on $S$ will be denoted by $\mathcal{PML}(S)$. It is homeomorphic to a sphere $\bbS^{6g-7+2p+3b}$ by \cite[Proposition 3.9]{ALPS}.
The simplest examples of measured laminations are given by the elements of  $\A \cup \mathcal{B} \cup \C$. Each becomes a measured lamination once it is endowed with the \emph{counting measure}, that is, the measure that counts the number of intersection points with a transverse arc. 


The length function $\ell$ from Section \ref{length} continuously extends to $\mathcal{ML}(S)$ (see \cite{ALPS}): 
$$ \ell : \T(S) \times \mathcal{ML}(S) \ni (X, \mu) \longrightarrow \ell_X(\mu) \in \mathbb R_{+}\,. $$

The arc distance $d_A$ in (\ref{form:distanceA})  
can be also computed as follows.
\begin{theorem}
[{\cite[Proposition 3.3]{ALPS}}]    \label{thm:sup on laminations}
The following holds: 
\begin{align}\label{form:distanceA2}
d_{A} (X,Y) = \max_{\mu \in \mathcal{ML}(S)\setminus \{\emptyset\}} \log\frac{\ell_Y(\mu)}{\ell_X(\mu)} = \max_{[\mu] \in \mathcal{PML}(S)} \log \frac{\ell_Y(\mu)}{\ell_X(\mu)}~.\end{align} 
\end{theorem}

Note that the supremum over $\A \cup \mathcal B \cup \C$ in (\ref{form:distanceA})  
is now replaced by a maximum over $\mathcal{PML}(S)$ in (\ref{form:distanceA2}).  
The measured laminations where the maximum is achieved are called \emph{ratio-maximizing measured laminations} (see Subsection \ref{subsec:geodesics}).

\subsection{Generalities about Bonahon-Thurston's cocycles and cataclysms}\label{Bonahon}

In this subsection, we will recall some basics facts about the shearing cocycles of a hyperbolic structure, following Bonahon \cite{Bonahon_pleated,Bonahon_holder}. This tool will be essential for us later, in Section \ref{sec:S_A}. 

In \cite{Bonahon_pleated}, Bonahon works for most of the paper under the hypothesis that the surface is closed. At the end, in Section 12.2 and 12.3, he explains how to extend his results to the case of a finite hyperbolic surface. Notice that in Bonahon's setting, laminations are not allowed to hit the boundary orthogonally, as we allow here.
Hence a maximal lamination decompose the surface into triangles. 

For the rest of this subsection, we fix a maximal lamination $\lambda$ on $S$ that does not hit the boundary of $S$ orthogonally. A \emph{transverse cocycle} for $\lambda$ can be thought of as a finitely additive signed measure for $\lambda$. 

\begin{definition}[Transverse cocycle] 
A \emph{transverse cocycle} for $\lambda$ is a map associating a number $\alpha(k) \in \mathbb{R}$ to each unoriented arc $k$ transverse to $\lambda$ such that $\alpha$ is \emph{additive}, and $\alpha$ is $\lambda$-\emph{invariant} (see \cite{Bonahon_holder} for more details). A transverse cocycle $\alpha$ satisfies the \emph{cusp condition} if every simple closed curve transverse to $\lambda$ and going once around a puncture of $S$ has zero total measure for $\alpha$. We denote by $H^0(\lambda; \mathbb R)$ the set of all transverse cocycles for $\lambda$ satisfying the cusp condition. 
\end{definition}

We will represent transverse cocycles using \emph{train tracks}, see Bonahon \cite[Section 3]{Bonahon_holder} for the definition.  
Following \cite{Bonahon_holder}, a train track $\tau$ \emph{snugly carries} a geodesic lamination $\lambda$ if $\tau$ carries $\lambda$, if $\lambda$ meets every tie of $\tau$, and if there is no curve carried by $\tau$ which is disjoint from $\lambda$ and which joins an endpoint of a spike of $S \setminus \tau$ to another one. Every train track which carries $\lambda$ can be transformed into one that snugly carries $\lambda$ after a finite sequence of splittings. 

\begin{definition}[Switch condition]
Let $\tau$ be a train track, with set of edges $E_\tau$. A function $\alpha:E_\tau \rightarrow \R$ is said to \emph{satisfy the switch relations} if the following condition holds: for every switch $v$ of $\tau$, if $e_1, \ldots , e_m$ are the edges arriving on one side of $v$ and $f_1, \ldots , f_n$ are the edges arriving on the other side, then 
$$\sum_{i=1}^m \alpha(e_i) = \sum_{j=1}^n \alpha(f_i).$$
A function $\alpha:E_\tau \rightarrow \R$ is said to \emph{satisfy the cusp condition} if for every puncture of $S$, the sum of the $\alpha$-values of the edges of $\tau$ going into the puncture is zero. 
\end{definition}

Combining Theorem 11 and Theorem 17 in Bonahon's paper \cite{Bonahon_holder}, we have: 

\begin{theorem}[Bonahon \cite{Bonahon_holder}]\label{Bonahon3}
Let $\tau$ be a train track that snugly carries $\lambda$. There is a one-to-one correspondence between the set of all transverse cocycles for $\lambda$ satisfying the cusp condition and the set of all the functions $\alpha:E_\tau \rightarrow \R$ which satisfy the switch relations and the cusp condition. In particular, the set $H^0(\lambda; \mathbb R)$ is a finite dimensional vector space.
\end{theorem}

Every $X \in \T(S)$ induces a special transverse cocycle for $\lambda$, the \emph{shearing cocycle}.

\begin{proposition}[Bonahon \cite{Bonahon_pleated}]
Every $X \in \T(S)$ determines a unique transverse cocycle $\rho_X \in H^0(\lambda; \mathbb R)$, called the \emph{shearing cocycle} of $X$.
\end{proposition}

The previous proposition gives a natural map  $\T(S)  \to H^0(\lambda; \R)$ defined as follows: 
\[\T(S) \ni X \mapsto \rho_X \in H^0(\lambda; \R)\,.\]
To understand this map, Bonahon uses \emph{Thurston's symplectic form} $\omega$ on the vector space $H^0(\lambda, \mathbb R)$. When $\lambda$ is carried by a \emph{generic} train track, that is, a train track where each switch is adjacent to exactly 3 edges, $\omega$ can be expressed by a simple formula. For generic train tracks, the three edges adjacent to the same switch play different roles: there is one edge on one side of the switch, called the \emph{incoming edge}, and two edges on the other side, called \emph{outgoing edges}. The two outgoing edges are called the \emph{left} and the \emph{right} outgoing edge, as seen from the incoming edge according to the orientation of $S$.

\begin{lemma}[{\cite[\S 3.2]{PennerHarer},\cite[Section 3]{Bonahon_pleated}}]  \label{lemma:sympl form}
If  the train track is generic, $\omega$ can be expressed as 
\[ \omega(\alpha, \beta) = \sum_{v } [\alpha(e_v^r) \beta(e_v^l) - \alpha(e_v^l) \beta(e^r_v)]\,, \]
were the sum is taken over all the switches $v$ of $\tau$, $e_v^l, e_v^r$ are the left and right outgoing edges from $v$, and $\alpha(e),\beta(e)$ are the weights associated 
to the edge $e$. 
\end{lemma}

\begin{theorem}[Bonahon \cite{Bonahon_pleated}]\label{Bonahon4} \label{shear_coordinates}
A transverse cocycle $\alpha \in H^0(\lambda; \mathbb R)$ is the shearing cocycle for a hyperbolic structure on $S$ if and only if $\omega(\alpha, \beta) >0$ for every compactly supported transverse measure $\beta$ for $\lambda$. In particular, the map $X \mapsto \rho_X$ defines a real analytic homeomorphism from $\T(\Sigma)$ to an open convex cone bounded by finitely many faces in $H^0(\lambda; \R)$. 
\end{theorem}

Thurston's stretch lines can be easily described using this theory:

\begin{proposition}[Bonahon \cite{Bonahon_pleated}]\label{thurston_cocycle} 
For an $X \in \T(S)$, denote by $X_{Th}^t \in \T(S)$ Thurston's stretch line starting from $X$ and directed by $\lambda$. Then 
$$ \rho_{X_{Th}^t} = e^t \cdot \rho_X. $$ 
\end{proposition}

%% file: 05-Average.tex
\section{Average of Lipschitz maps} \label{sec:Lip} 

In this section we will deal with Lipschitz maps between convex hyperbolic surfaces. We introduce here a new technique that combines any two such maps into a new map whose Lipschitz constant is at most the average of their constants, generalizing a result by Gu\'eritaud-Kassel \cite{Kassel-Gueritaud}. We will employ this construction in Section \ref{sec:stretch_1}. 

\subsection{The Gu\'eritaud-Kassel construction for convex subsets of $\mathbb H^2$}
Given two metric spaces $(\Omega, d)$ and $(\Omega', d')$, a map $\phi : \Omega \to \Omega'$ is called ($K$-)\emph{Lipschitz} if there exists a real number $K\geq 0$ such that
$$d'(\phi(x_{1}),\phi(x_{2}))\leq Kd(x_{1},x_{2})$$
 for all $x_1, x_2 \in \Omega$.
The \emph{Lipschitz constant} $\Lip(\phi)$ is the smallest of such $K$'s. 

The following criterion allows to recover the Lipschitz constant from local information. 
\begin{lemma}[Gu\'eritaud-Kassel {\cite[Lemma 2.9]{Kassel-Gueritaud}}]\label{LOC-2-GLOB}
Let $\Omega$ be a convex subset of $\mathbb H^2$ or a convex hyperbolic surface. Let $(\Omega', d')$ be a metric space. 
If $\phi: \Omega \to (\Omega', d')$ is a continuous function then we have:
$$\mathrm{Lip}(\phi) = \mathrm{sup}_{x \in \Omega} \mathrm{inf}_{r>0} \mathrm{Lip}(\phi_{|B(x;r)}). $$
\end{lemma}

We will be interested in a generalization of the following map. 
\begin{definition}[Gu\'eritaud-Kassel {\cite[Section 2.3]{Kassel-Gueritaud}}] \label{def:average}
Let $\Omega, \Omega' \subseteq \mathbb H^2$ be convex subsets. 
Let $\phi,\psi: \Omega \to \Omega'$ be continuous maps.
The \emph{average} of $\phi$ and $\psi$ is the map $\Upsilon: \Omega \to \Omega' $ such that for every $x \in \Omega$, $\Upsilon(x)$ is the midpoint  of the geodesic joining $\phi(x)$ and $\psi(x)$. 
\end{definition}

\begin{lemma}[Gu\'eritaud-Kassel {\cite[Lemma 2.13]{Kassel-Gueritaud}}]\label{lemma: Fanny}
Let $\Omega, \Omega'\subset \mathbb{H}^2$ be convex subsets, and $\phi,\psi: \Omega \to \Omega'$ be Lipschitz maps. Then their average $\Upsilon: \Omega \to \Omega' $ is a Lipschitz map and $$\mathrm{Lip}(\Upsilon) \leq \frac{\mathrm{Lip}(\phi)+\mathrm{Lip}(\psi)}{2}~.$$
\end{lemma}

\subsection{Generalizing Gu\'eritaud-Kassel's construction to convex hyperbolic surfaces}
Assume that $X,Y$ are two convex hyperbolic surfaces (not necessarily homeomorphic) and let $\phi, \psi: X \to Y$ be two continuous maps in the same homotopy class. We will now define a new continuous map $\Upsilon_{x_0,\gamma}(\phi,\psi): X \to Y$  called the \emph{average} of $\phi$ and $\psi$. This map generalizes the previous construction of Gu\'eritaud-Kassel to hyperbolic surfaces. Our construction will depend on the choice of a base point $x_0 \in X$ and a geodesic path $\gamma:[0,1] \to Y $ joining $\phi(x_0)$ and $\psi(x_0)$.

\subsubsection*{Step 1: Construct two suitable lifts of $\phi$ and $\psi$} 
Let $p: \widetilde{X} \to X$ and $q: \widetilde{Y} \to Y$ be the universal coverings of $X$ and $Y$. Recall that $\widetilde{X}, \widetilde{Y} \subset \mathbb{H}^2$ are convex. Choose points $\widetilde{x_0}\in  p^{-1}(x_0)$ and $\widetilde{y_0} \in q^{-1}(\phi(x_0))$. 
There exists a unique lift $\widetilde{\phi}: \widetilde{X} \to \widetilde{Y}$ of $\phi$ such that $\widetilde{\phi}(\widetilde{x_0}) = \widetilde{y_0}$.  
Similarly, there exists a unique lift $\widetilde{\gamma}:[0,1] \to \widetilde{Y}$ such $\widetilde{\gamma}(0) = \widetilde{y_0}$ (see the diagram below). 
\begin{equation*}
\xymatrix{
      \widetilde{X}\ar[d]_p  \ar[r]^{\widetilde{\phi}} & \widetilde{Y}  \ar[d]^{q}  \\ 
      X \ar[r]_{\phi}                            & Y 
}
\qquad
\xymatrix{
                           & \tilde{Y}  \ar[d]^{q}  \\ 
      [0,1] \ar[r]_{\gamma} \ar[ru]^{\widetilde{\gamma}} & Y 
}   
\qquad 
\xymatrix{
      \widetilde{X}\ar[d]_p  \ar[r]^{\widetilde{\psi}} & \widetilde{Y}  \ar[d]^{q}  \\ 
      X \ar[r]_{\psi}                            & Y 
}
\end{equation*}

Denote by $\widetilde{z_0} = \widetilde{\gamma}(1)$. 
By construction $\widetilde{z_0} \in q^{-1}(\psi(x_0))$. There exists a unique lift $\widetilde{\psi}: \widetilde{X} \to \widetilde{Y}$ of $\psi$ such that $\widetilde{\psi}(\widetilde{x_0}) = \widetilde{z_0}$. The construction above give us two maps $\widetilde{\phi}, \widetilde{\psi}: \widetilde{X} \to \widetilde{Y}$, between convex subsets of $\mathbb H^2$. 


\subsubsection*{Step 2: Define $\Upsilon: \Omega \to \Omega'$} 
We can now define the map $$\widetilde{\Upsilon}: \widetilde{X} \to \widetilde{Y}$$ that maps every $x\in \widetilde{X}$ to the midpoint of the geodesic segment joining $\widetilde{\phi}(x)$ and $\widetilde{\psi}(x)$ as in Definition \ref{def:average}. The lemma below follows easily from elementary arguments on coverings and their automorphisms.   
\begin{lemma}
The map $\widetilde{\Upsilon}: \widetilde{X} \to \widetilde{Y}$ commutes with $p$ and $q$. 
$$\xymatrix{ 
\widetilde{X} \ar[r]^{\widetilde{\Upsilon}} \ar[d]_p & \widetilde{Y} \ar[d]^{q} \\ 
X \ar[r]_{\Upsilon} & Y 
}$$
\end{lemma}

\begin{definition}
The $(x_0, \gamma)$-\emph{average} between $\phi$ and $\psi$  is the map $\Upsilon: X \to Y$ induced by $\widetilde{\Upsilon}: \tilde X \to \tilde Y$ and defined as follows:
$$\Upsilon(x) := q(\widetilde{\Upsilon}(\widetilde{x})),$$ 
where  $\widetilde{x}$ is any element in $p^{-1}(x)$. We will also use the notation $$ \Upsilon_{x_0,\gamma}(\phi,\psi):= \Upsilon. $$ Using the basic properties of coverings, it is easy to verify that the map $\Upsilon$ does not depend on $\widetilde{y_0} \in q^{-1}(\phi(x_0))$ and $\widetilde{x_0} \in q^{-1}(x_0)$.  
\end{definition}

\begin{theorem}\label{prop:average}
Let $X$, $Y$ be two (possibly non-homeomorphic) convex hyperbolic surfaces. Let $\phi, \psi: X \to Y$ be two homotopic Lipschitz maps. Then for every $x_0 \in X$ and for every $\gamma:[0,1] \to Y$ with $\gamma(0) = \phi(x_0)$ and $\gamma(1) = \psi(x_0)$, the map $\Upsilon_{x_0,\gamma}(\phi, \psi): X\to Y$  is Lipschitz with 
$$\mathrm{Lip}(\Upsilon_{x_0,\gamma}(\phi, \psi)) \leq \frac{\mathrm{Lip}(\phi)+\mathrm{Lip}(\psi)}{2}~.$$
\end{theorem}
\begin{proof}
We lift the maps $\phi$ and $\psi$ to the universal coverings as in Step 1 above. The Lipschitz constant of $\widetilde{\Upsilon}$ is bounded by Lemma \ref{lemma: Fanny}. This gives a bound on the Lipschitz constant of $\Upsilon$ by Lemma \ref{LOC-2-GLOB}.
\end{proof}

%% file: 06-Geometric_Pieces.tex
\section{Generalized stretch maps between geometric pieces}\label{sec:stretch_1} 
 
We will construct optimal Lipschitz maps between geometric pieces of the same type. 

\begin{definition}
Let $\mathcal{G}$ and $\mathcal{G'}$ be two geometric pieces of the same type. A continous map $\phi: \mathcal{G} \to \mathcal{G'}$ is \emph{label-preserving} if it  maps every edge of $\mathcal{G}$ to an edge of $\mathcal{G'}$ with the same label. Recall that the labels are assigned as in Figure \ref{geom_pieces}.  
\end{definition}

\subsection{Centers and shears} \label{subsec:shear coordinates}
Only triangular and quadrilateral pieces have bi-infinite edges. There is a one-parameter family of ways to glue two of them together along a bi-infinite edge. We will parametrize the glueing using the \emph{shear parameter}, that is, the (signed) distance between their \emph{centers}. We recall these key-definitions below.  
 
\begin{definition}[Center of $l_i$ with respect to $\TT$]
Let $l_i$ be a bi-infinite edge in a triangular piece $\TT$. The \emph{center} of $l_i$ with respect to $\TT$ is the intersection point $O_{\TT}^i$ between $l_i$ and the geodesic perpendicular to $l_i$ through the opposite vertex.
\end{definition}

Note that each triangular piece has three centers $O_\TT^1$, $O_{\TT}^2$, $O_{\TT}^3$. 

\begin{definition}[Center of $l$ with respect to $\Q$]
Let $l$ be the (unique) bi-infinite edge in a quadrilateral piece $\Q$. The \emph{center} of $l$ with respect to $\Q$ is the intersection point $O_\Q$ between $l$ and the unique (geodesic) perpendicular to $l$ and to the opposite edge.  
\end{definition}

We are now in the position to define the following.

\begin{definition}[Shear between $\G$  and $\G'$]   \label{def:shear}
Let $\G$ and $\G'$ be two geometric pieces glued along the bi-infinite edge $e$. We define the \emph{shear parameter} between $\G$ and $\G'$  as the signed distance $\mathrm{shear}_e(\G, \G') \in \mathbb R$ between the centers $O_\G$ of $e$ with respect to $\G$ and the center $O_{\G'}$ of $e$ with respect to $\G'$. The sign is given by the orientation of the surface, which we always assume to be counter-clockwise: the sign is positive if $O_{\G}$ comes before $O_{\G'}$ with respect to the orientation that $\partial \G$ induces on $e$. (Notice that flipping the roles of $\G$ and $\G'$ in this construction does not change the sign of $\mathrm{shear}_e(\G, \G')$.)
\end{definition}

\begin{figure}[htb]
\begin{center}
\psfrag{Q}{$\mathcal Q$}
\psfrag{T}{$\mathcal T$}
\psfrag{O'}{\tiny $O_{\mathcal Q}$}
\psfrag{O}[tl]{\tiny $O_{\mathcal T~~}~~$}
\includegraphics[width=4cm]{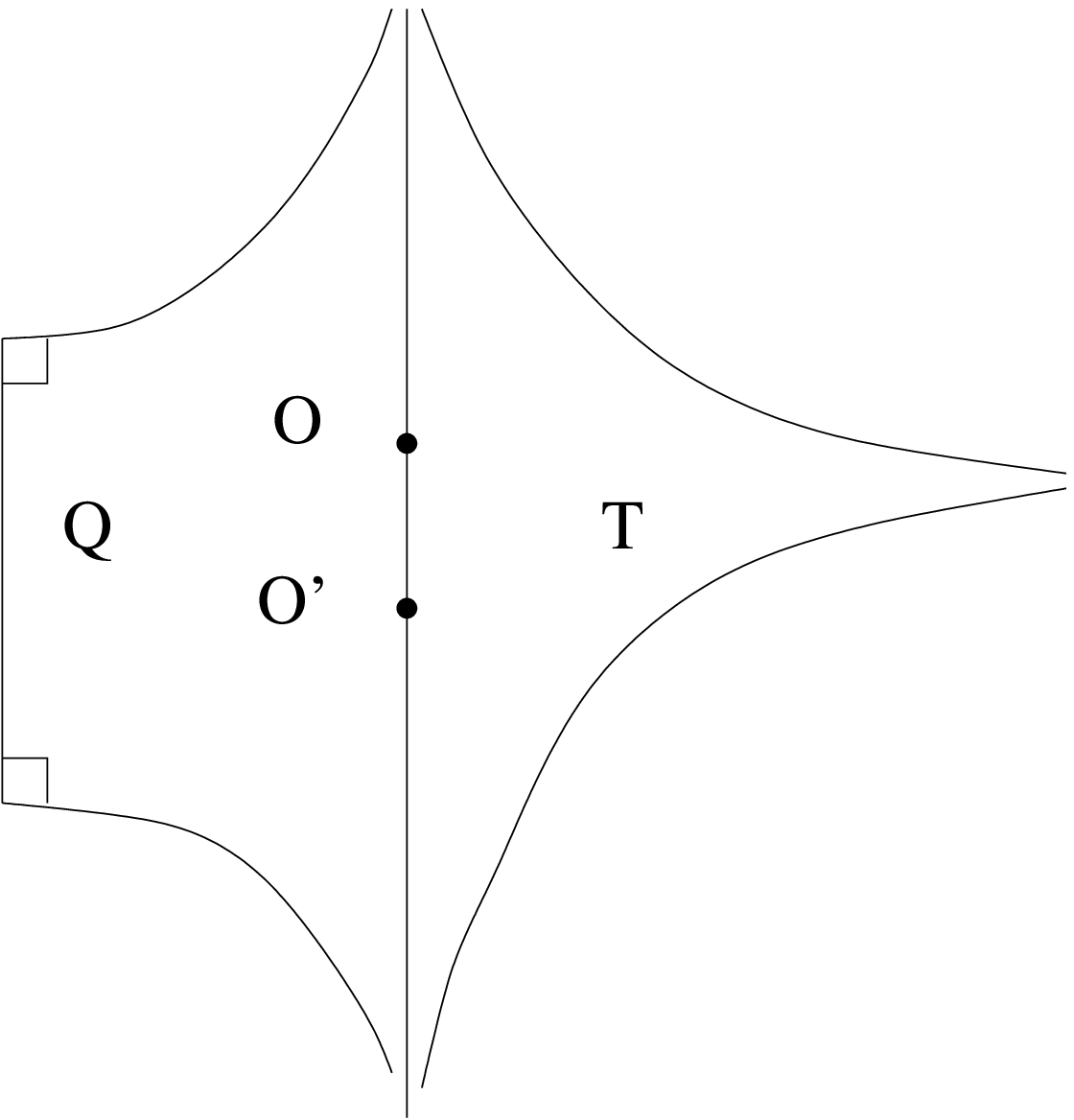}
\caption{$\mathrm{shear}(\mathcal Q, \mathcal T) > 0 $ }      \label{fig:shear}
\end{center}
\end{figure}

\subsection{Thurston's stretch homeomorphism between triangular pieces} 
Thurston \cite{Thurston} explicitly describes a family of Lipschitz homeomorphisms between ideal triangles. 

\begin{lemma}[Thurston {\cite[Proposition 2.2]{Thurston}}]\label{lemma:triangle}
Let $t\geq 0$ and $\TT, \TT^t$ be two triangular pieces. There exists a label-preserving map $\phi^t: \TT \to \TT^t$ with the following properties: 
\begin{enumerate}
\item $\phi$ is a homeomorphism;
\item $\phi^t(O_\TT) = O_{\TT^t}$ and its restriction $\phi^t_|: l_i \to l_i^t$ multiplies the arc length by $e^t$ for each $i=1, 2, 3$; 
\item  $\mathrm{Lip}(\phi^t) = e^t$.
\end{enumerate}
\end{lemma}
We denote the target of the map $\phi^t$ by $\TT^t$ for consistency with the next subsections ($\TT^t$ is actually isometric to $\TT$). In order to define his map $\phi^t$, Thurston defined the \emph{horocyclic foliation} a partial foliation $\mathcal{K}$ of $\TT$ defined as follows (see Figure \ref{fig:horocyclic foliation triangle}). Consider the vertex of $\TT$ adjacent to the edges $l_1$ and $l_2$. A horocycle $h$ centered at this vertex intersects the edges $l_1, l_2$ at the points $h_1, h_2$, so that
$d(h_1,O_\TT^1) = d(h_2, O_\TT^2).$
We consider a partial foliation $\mathcal{K}_{12}$ whose leaves are all the horocycles whose points $h_1,h_2$ are closer to the vertex than the corresponding center $O_\TT^1$ or $O_\TT^2$. 
We denote by $h_{12}^d$ the only horocycle $h$ in $\mathcal{K}_{12}$ such that $d(h_1,O_\TT^1)=d$. Similarly, we define partial foliations $\mathcal{K}_{23}, \mathcal{K}_{31}$ starting with the other two vertices. 
\begin{figure}[htbp]
\begin{center}
\psfrag{A}{$O_\TT^1$}
\psfrag{B}{$O_\TT^2$}
\psfrag{C}{$O_\TT^3$}
\includegraphics[width=4cm]{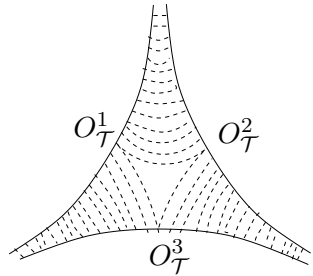}
\caption{Horocyclic foliation}  \label{fig:horocyclic foliation triangle}
\end{center}
\end{figure}
\begin{definition}[Horocyclic foliation]  \label{def:horfol3}
The \emph{horocyclic foliation} $\mathcal{K}$ is the union of the three partial foliations $\mathcal{K}_{12}$, $\mathcal{K}_{23}$ and $\mathcal{K}_{31}$. 
(The triangle bounded by the three horocycles $h_{12}^0$, $h_{23}^0$ and $h_{31}^0$ is unfoliated.) 
\end{definition}

\begin{lemma}[Thurston {\cite[Proposition 2.2]{Thurston}}]   \label{lem:horocyclic foliation triangle}
The map $\phi^t$ maps the leaf $h_{ij}^d$ of $\mathcal{K}$ in $\TT$ to the leaf $h_{ij}^{e^t d}$ of $\mathcal{K}$ in $\TT^t$ affinely. On the unfoliated region, $\phi^t$ is the identity.    
\end{lemma}

In the rest of this section, we will prove results analog to Lemmas \ref{lemma:triangle} and \ref{lem:horocyclic foliation triangle} for the geometric pieces of the other kinds.

\subsection{Parameters for the geometric pieces}    \label{subsec:parameters}

While there is just one ideal triangle up to isometry, the geometric pieces of the other kinds have parameters. A quadrilateral piece $\Q$ is uniquely determined by the length of the edge $a_1$, a pentagonal piece $\PP$ by the lengths of the edges $a_1$, $a_2$, an hexagonal piece $\HH$ by the lengths of three alternating edges, see Figure \ref{geom_pieces}. But these parametrizations are not very convenient for our needs. We will now introduce other parameters for the geometric pieces.

Doubling a quadrilateral piece $\Q$ along $a_1$, we get an ideal quadrilateral $\Q^d$, as in Figure \ref{Q1}. Consider on $a_1$ the orientation induced by $\mathbb H^2$, and triangulate $\Q^d$ adding a diagonal $e$ accordingly. Let $\TT, \TT'$ be the triangles obtained. Their shear $s$ parametrizes $\Q$ completely. Notice that $a_1$ takes values in $(0,\infty)$, while $s$ takes values in $(-\infty,\infty)$. As $s \to + \infty$ we have $a_1 \to + \infty$, but as $s \to -\infty$ we have $a_1 \to 0$. 
 
\begin{figure}[htb]
\begin{center}
\psfrag{C'}{\small $O_{\mathcal T'}$}
\psfrag{C}{\small $O_{\mathcal Q}$}
\psfrag{Q}[l]{ $\mathcal Q^d$}
\psfrag{s}{$s$}
\psfrag{T'}{$\mathcal T'$}
\psfrag{T}{$\mathcal T$}
\includegraphics[width=3cm]{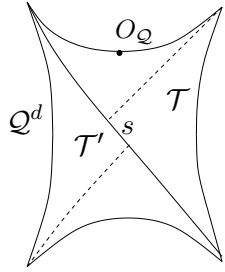}
\caption{The quadrilateral $\mathcal Q_s^d$} \label{Q1}
\end{center}
\end{figure}

Similarly, doubling a pentagonal piece $\PP$ along the edges $a_1$ and $a_2$, we get a hyperbolic cylinder $\PP^d$ with two spikes and a totally geodesic boundary. We will denote it by $l_1^d$. After choosing an orientation on $l_1^d$, we consider two geodesics $e_1, e_2$ coming from each of the two spikes and spiraling around the geodesic $l_1^d$ according to the chosen orientation. The geodesics $e_1, e_2$ decompose $\PP^d$ in two ideal triangles $\TT, \TT'$, as in Figure \ref{thurston:penta}. Their shears coordinates $s_1$ and $s_2$ parametrize $\PP^d$ and hence $\PP$: we will use them as parameters. The shear coordinates depend on the choice of orientation of the geodesic $l_1^d$. Indeed, by choosing the other orientation, we would have the mirror image of the same picture: the two shear coordinates would have the same absolute values but opposite signs. To fix the signs of the coordinates, we will always choose the orientation of $l_1^d$ so that $s_1 + s_2 > 0$. Note that this sum can never be zero as $|s_1+s_2|=\ell(l_1^d)$ by \cite[Proposition 3.4.21]{ThurstonBook}.

\begin{figure}[htb]
\begin{center}
\psfrag{a}{\small $a_1$}
\psfrag{b}[B]{\small$a_2$}
\psfrag{T}{\small $\TT$}
\psfrag{T'}{\small $\TT'$}
\includegraphics[width=11cm]{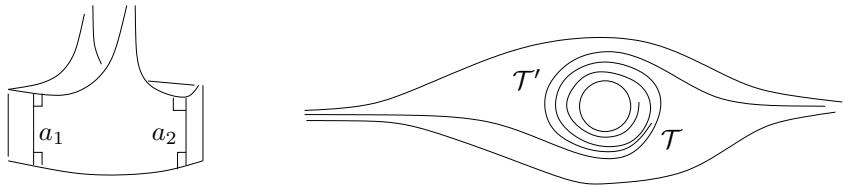}
\caption{Doubling the pentagon to get $\PP^d$}\label{thurston:penta}
\end{center}
\end{figure}

Again, doubling a hexagonal piece $\HH$ along the edges $a_1, a_2, a_3$, we get a pair of pants $\HH^d$ with $3$ geodesic boundary components that we will denote by $l_1^d, l_2^d, l_3^d$. We can find three disjoint infinite simple geodesics $e_1, e_2, e_3$ in $\HH^d$ such that $e_i$ spirals around $l_{i+1}^d$ and $l_{i+2}^d$, for $i\in \{1,2,3\}$ and sums of indices taken modulo $3$. The geodesics $e_1, e_2, e_3$ cut the pair of pants in two triangles $\TT, \TT'$, glued with shear coordinates $s_1, s_2, s_3$. Up to changing the directions of the spirals, we can always choose the geodesics $e_i$ such that $s_i + s_{i+1} = \ell(l_{i+2}^d)$ (see \cite[Proposition 3.4.21]{ThurstonBook}). Thus at least two of  the $s_i$'s are positive: assume $s_2, s_3 > 0$. We will parametrize the hexagons using the three shear coordinates $s = (s_1,s_2,s_3)$.

In the three cases, the lengths of the edges $a_i$ can be computed explicitly from the shears, but we will not need these formulae here.

\subsection{Stretching the geometric pieces}

In the rest of the paper, we will denote by $\Q:= \Q_s$,  $\PP := \PP_{s}$ with $s = (s_1,s_2)$, $\HH := \HH_{s}$ with $s = (s_1,s_2,s_3)$ a geometric piece having certain prescribed shear coordinates.
Moreover, we will use the notation $\Q^t := \Q_{e^t s}$, $\PP^t := \PP_{e^t s}$, $\HH^t := \HH_{e^t s}$.  

\begin{lemma}[Stretch of geometric pieces]\label{lemma:quad} 
Fix $t\geq 0$. Let $\FF_s = \Q_s, \PP_s$ or $\HH_s$ be a marked geometric piece. Then there exists a label-preserving map $\phi^t:\FF_s \to \FF_{e^t s}$ such that: 
\begin{enumerate}
\item $\phi^t$ is onto;
\item the map $\phi^t_|: l_i \to l_i^t$ is affine and multiplies the arc length by $e^t$ for every $i=1, 2, 3$;
\item $\mathrm{Lip}(\phi^t) = e^t$. 
\item If the piece is a quadrilateral, $\phi^t(O_\Q) = O_{\Q^t}$;
\end{enumerate}
\end{lemma}
\begin{proof}
We double $\FF_s$ and $\FF_{s e^t}$ obtaining $\FF_s^d$ and $\FF_{e^t s}^d$, as we did in Section \ref{subsec:parameters}. Let $\sigma : \FF_{s}^d \to \FF_{s}^d$ be the isometric involution that maps one copy of $\FF_s$ in $\FF_s^d$ to the other copy. Similarly, let $\sigma^t: \FF_{s e^t}^d  \to \FF_{s e^t}^d$ be the corresponding isometric involution on $\FF_{s e^t}^d$. Let $\TT$ and $\TT'$ be the two ideal triangles in $\FF_s^d$ constructed in Section \ref{subsec:parameters}, separated by edges $e_i$ (where $i=1$ for $\FF_s = \Q_s$, $i \in \{1,2\}$ for $\FF_s = \PP_s$, $i \in \{1,2,3\}$ for $\FF_s = \PP_s$.) Let $\TT_t$ and $\TT_t'$ be the corresponding triangles in $\FF_{s e^t}^d$. 

Let $\psi_t: \TT \to \TT_t$ and $\psi_t': \TT' \to \TT_t'$ be the two homeomorphisms as in Lemma \ref{lemma:triangle}. 
The maps $\psi_t$ and $\psi_t'$ agree on the edges $e_i$, since $ \mathrm{shear}_{e_i}(\TT_t, \TT_t') = e^t \cdot \mathrm{shear}_{e_i}(\TT, \TT')$. 
Hence, the maps $\psi_t$ and $\psi_t'$ glue to a homeomorphism $\Psi^t: \FF_{s}^d \to \FF_{s e^t}^d$. 
By construction, $\Psi^t$ maps every edge $l_i$ of $\FF_{s}^d$ to the corresponding edge of $\FF_{s e^t}^d$ multiplying its arc length by $e^t$. 

By Lemma \ref{LOC-2-GLOB} we have:
$$\mathrm{Lip}(\Psi^t) = \mathrm{max}\{ \mathrm{Lip}(\psi_t), \mathrm{Lip}(\psi_t') \} = e^t.$$ 

Similarly, $\sigma_t \circ \Psi^t \circ \sigma: \FF_{s}^d \to \FF_{s e^t}^d$ is a $e^t$-Lipschitz homeomorphism that maps every edge $l_i$  of $\FF_{s}^d$ to the corresponding edge of $\FF_{s e^t}^d$ multiplying its arc length by $e^t$. 

Choose $x_0 \in a_1 \subset \FF_{s}$. By construction we have: $\sigma(x_0) =x_0$ and $\Psi_t^\sigma(x_0) = \sigma_t \circ \Psi^t(x_0)$. Let $\gamma$ be a geodesic segment crossing $a_1$ once and joining $\Psi^t(x_0)$ and $\Psi^\sigma_t(x_0)$. Thus $\gamma$ is orthogonal to $a_1$, and $\sigma_t(\gamma) = \gamma$. We consider the map 
$\Upsilon^t:=\Upsilon_{x_0,\gamma}(\Psi^t, \Psi_t^\sigma): \FF_{s}^d \to \FF_{e^t s}^d$, that is, the average of $\Psi^t$ and $\Psi_t^\sigma$ with respect to $x_0$ and $\gamma$, in the sense of Theorem \ref{prop:average}. By construction and Theorem \ref{prop:average}, $\Upsilon^t$ enjoys the following properties: 
\begin{itemize}
\item $\Upsilon^t$ is onto; 
\item $\Upsilon^t$ maps every edge $l_i$ of $\FF_{s}^d$ to the corresponding one of $\FF_{s e^t}$ by multiplying its arc-length by $e^t$; 
\item $\mathrm{Lip}(\Upsilon^t) = e^t$.
\item If $\mathcal F_s = \Q_s$, $\Upsilon^t$ maps the center of $l_2$ in $\Q_{s}^d$ to the center of $l_2$ in $\Q_{s e^t}^d$;  
\end{itemize}
Moreover, by construction we have: 
$$\sigma_t \circ \Upsilon^t \circ \sigma = \Upsilon_{\sigma(x_0), \sigma(\gamma)}(\sigma_t \circ \Psi \circ \sigma, \sigma_t \circ \Psi^\sigma \circ \sigma) = \Upsilon_{x_0,\gamma}(\Psi^\sigma, \Psi) = \Upsilon_{x_0, \gamma}(\Psi, \Psi^\sigma)= \Upsilon^t.$$ 
Hence, the image by $\Upsilon^t$ of the edges $a_i$ of $\FF_{s}$ are the corresponding edges of $\FF_{e^t s}$, and $\Upsilon^t$ restricts to 
$\phi^t = \Upsilon_{|\FF_{s}} : \FF_{s} \to \FF_{e^t s}$ as in the statement. 
\end{proof}

It is interesting to compare our construction of the stretch map for the case of the hexagon with the one given by Papadopoulos-Yamada \cite{PapYam}, who construct optimal Lipschitz maps between special types of hexagons. Their work generalize an explicit example by Papadopoulos-Th\'eret \cite{PapTher3} for hexagons with $l_1 = l_2 = l_3$. The Lipschitz constant of the Papadopoulos-Yamada map is usually achieved only on one of the three alternating edges, but not on all of them. Because of this, their map is not suitable for our purposes.

\subsection{Understanding the shear parameters}

\begin{figure}[t!]  
\begin{center}
\begin{subfigure}[t]{0.3\textwidth}
\psfrag{l1}{$A$}
\psfrag{l2}{$B$}
\psfrag{a}{$C$}
\psfrag{l3}{$D$}
\psfrag{P1}{\tiny $P_{AD}$ ~~~ }
\psfrag{P2}{\tiny $P_{BC}$}
\psfrag{O}[b]{\small $O_{\Q}$}
\includegraphics[width=4cm]{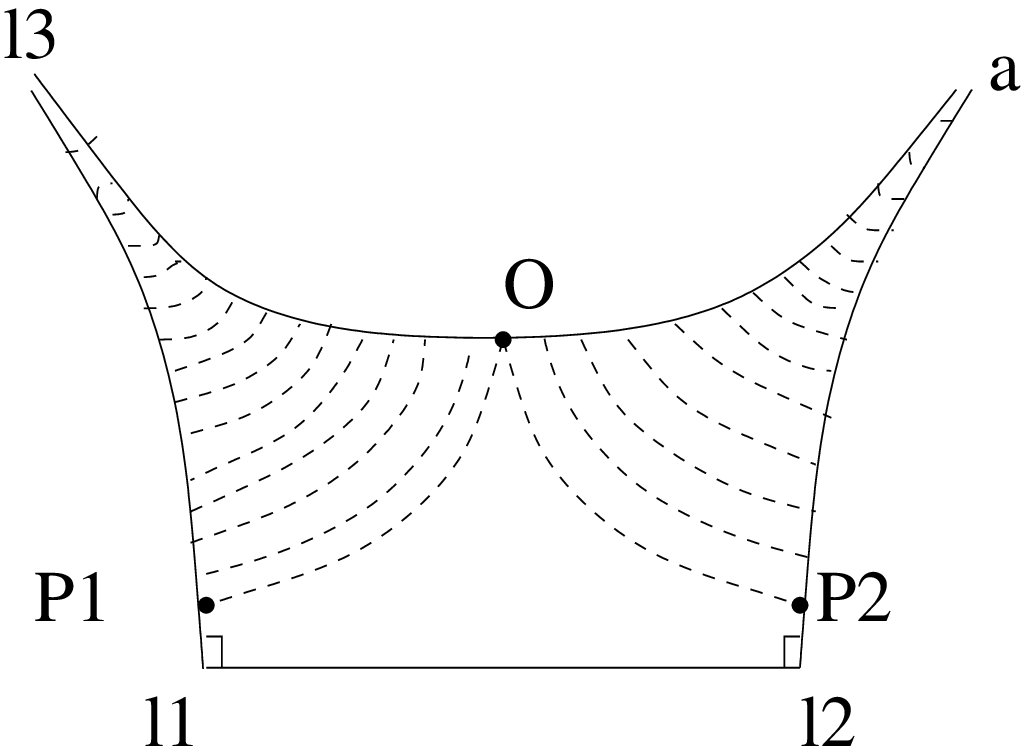}
\caption{$s>0$}
\end{subfigure}
\begin{subfigure}[t]{0.3\textwidth}
\psfrag{O}[b]{\small $O_{\Q}$}
\psfrag{l1}{$ $}
\psfrag{l2}{$ $}
\psfrag{l3}{$D$}
\psfrag{a}{$C$}
\psfrag{P1}{\tiny $P_{AD}=A$}
\psfrag{P2}{\tiny $P_{BC}=B$}
\includegraphics[width=4cm]{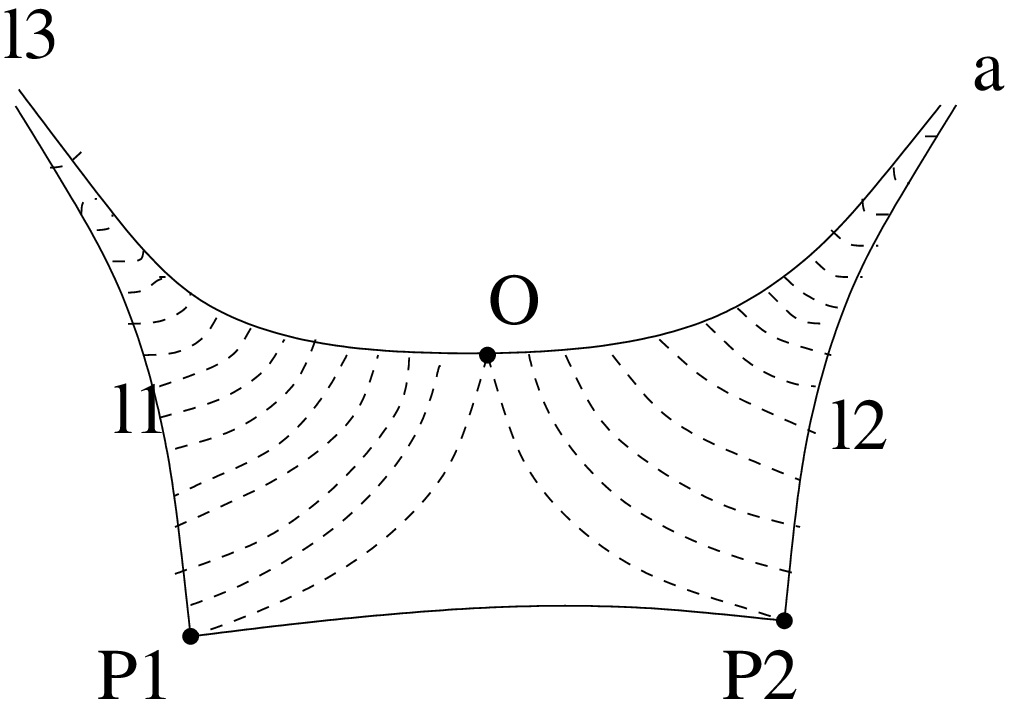}
\caption{$s=0$}
\end{subfigure}
\begin{subfigure}[t]{0.3\textwidth}
\psfrag{O}[b]{\small $O_{\Q}$}
\psfrag{l1}{$ $}
\psfrag{l2}{$ $}
\psfrag{l3}{$D$}
\psfrag{a}{$C$}
\psfrag{P1}{\tiny $P_{AD}$}
\psfrag{P2}{\tiny $P_{BC}$}
\includegraphics[width=4cm]{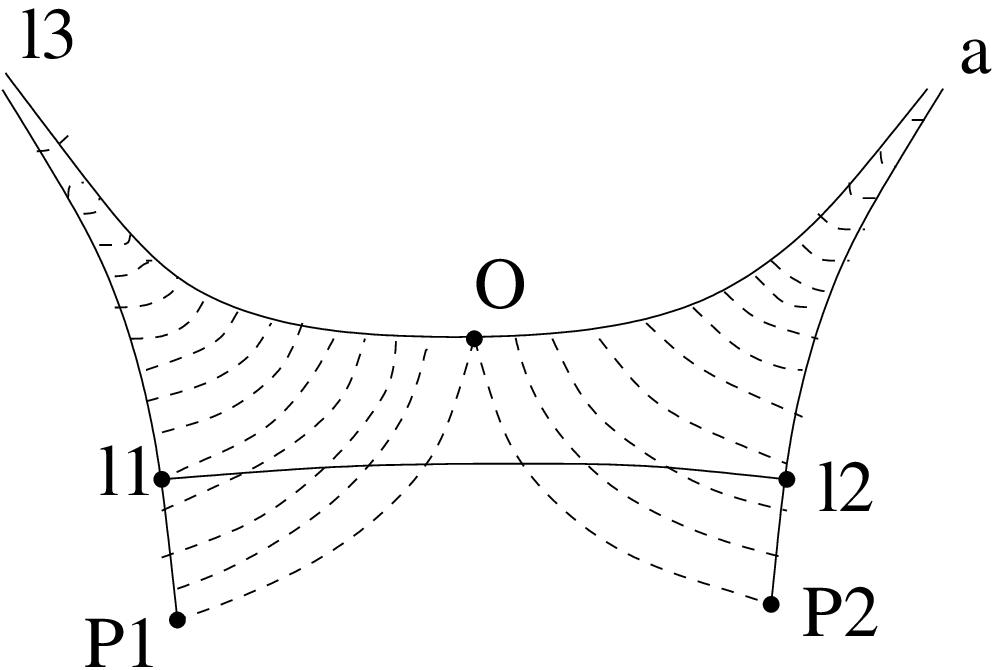}
\caption{$s<0$}
\end{subfigure}
\caption{Lemma \ref{shear_quads}} \label{foliation quad}
\end{center}
\end{figure}

We will now give a geometric interpretation to the shear parameters $s$ of the geometric pieces. Our aim is to prove Propositions \ref{shear_quads2}, \ref{shear_penta2}, \ref{shear_hexa2}.

Consider $\Q_s:=\overline{ABCD} \in \mathbb H^2$ where $D, C$ are the two ideal vertices, as in Figure \ref{foliation quad}. Let $\mathscr{F}$ be the horocyclic foliation based in $C$, with $f \in \mathscr F$ its (unique) leaf through $O_\Q$. Denote by $P_{AD}$ the intersection point between $f$ and the bi-infinite geodesic of $\mathbb H^2$ obtained extending $\overline{BC}$. Similarly, let $\mathscr{F}'$ be the horocyclic foliation based in $D$, with $f' \in \mathscr F'$ its (unique) leaf passing through $O_\Q$. Denote by $P_{BC}$ the intersection point between $f'$ and the bi-infinite geodesic of $\mathbb H^2$ obtained extending $\overline{AD}$. We will compute the ``signed'' distances between $P_{BC}, P_{AB}$ and $B, A$ respectively. We define: 
\begin{enumerate}
\item $d_{\pm}(P_{BC},B):= \epsilon \cdot d(P_{BC}, B)$, where $\epsilon = 1$ when $P_{BC} \in \overline{BC}$ and $\epsilon = -1$ when $P_{BC} \not \in \overline{BC}$;
\item $d_{\pm}(P_{AD},A):= \epsilon \cdot d(P_{AD}, A)$, where $\epsilon = 1$ when $P_{AD} \in \overline{AD}$ and $\epsilon = -1$ when $P_{AD} \not \in \overline{AD}$.
\end{enumerate}
By construction, it is clear that $d_{\pm}(P_{BC},B) = d_{\pm}(P_{AD},A)$.
 
\begin{lemma}\label{shear_quads}
In the notation above, we have: $$d_{\pm}(P_{BC},B) = d_{\pm}(P_{AD},A) = \frac{s}{2}.$$
\end{lemma}
\begin{proof}
We will compute these lengths explicitly. We will denote by $C^u,D^u$ the vertices of $\Q_s^d$ which are the reflection of $C,D$. The ideal quadrilateral $\Q_s^d$ can be drawn in the upper half plane model of $\mathbb{H}^2$, with vertices $D^u=-1, C^u =0, C= e^s, D= \infty$, see Figure \ref{quadrilateral in the plane}. With this choice, the two ideal triangles $\TT=\overline{C^uCD}$ and $\TT'=\overline{D^uC^uD}$ are glued with shear coordinate equal to $s$. 

We will first compute the coordinates of the center $O_\Q\in \overline{CD}$. We denote by $O_\Q^u \in \overline{C^uD^u}$ the reflection of $O_\Q$. The geodesic segment $\overline{O_\Q O_\Q^u}$ is the common perpendicular of the geodesics $\overline{CD}$ and $\overline{C^uD^u}$. In the language of Euclidean geometry, $\overline{O_\Q O_\Q^u}$ is an arc of a Euclidean circle centered at $C$ and perpendicular to $\overline{C^uD^u}$. By a computation, the  Euclidean radius of this circle is  $\sqrt{e^s(1+e^s)}$. This number is also the $y$-coordinate of the points $O_\Q$ and $P_{AD}$ (see Figure \ref{quadrilateral in the plane}).

We will now compute the coordinates of the point $A$ in a similar way. The geodesic segment $\overline{AB}$ is the common perpendicular of the geodesics $\overline{CC^u}$ and $\overline{DD^u}$. In the language of Euclidean geometry, $\overline{AB}$ is an arc of a Euclidean circle centered at $D^u$ and perpendicular to $\overline{CC^u}$. By a computation, the  Euclidean radius of this circle is  $\sqrt{1+e^s}$. This is also the $y$-coordinate of $A$.

The number  $d_{\pm}(P_{AD},A)$ is the log of the ratio of the $y$-coordinates of $P_{AD}$ and $A$: 
$$ d_{\pm}(P_{AD}, A) = \log \frac{\sqrt{e^s (1+e^s)}}{\sqrt{1+e^s}} = \frac{s}{2} \,.$$
\end{proof}

\begin{figure}[ht]
\begin{center}
\psfrag{A}{$A$}
\psfrag{B}{$B$}
\psfrag{f}{$f$}
\psfrag{D'}{$D=\infty$}
\psfrag{D}{$D^u=-1$}
\psfrag{C'}{$C^u=0$}
\psfrag{C}{$C=e^s$}
\psfrag{P}[B]{$P_{AD}~$}
\psfrag{O}{$O_Q$}
\psfrag{O'}{$O_Q^u$}
\psfrag{t'}[B]{$\mathcal T$}
\psfrag{t}{$\mathcal T'$}
\includegraphics[width=10cm]{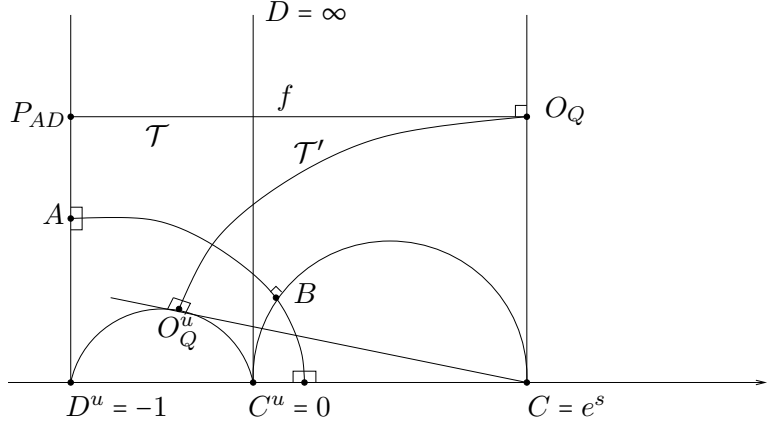}
\caption{The quadrilateral $\Q^d$ in the upper half plane model} \label{quadrilateral in the plane}
\end{center}
\end{figure}

The following result is an immediate consequence of Lemmas \ref{lemma:quad} and \ref{shear_quads}

\begin{proposition}\label{shear_quads2}
Let $t\geq 0$, and let  $\phi^t: \Q \to \Q^t$ be the generalized stretch map as in Lemma \ref{lemma:quad}. Then, if $s\geq 0$, the map $\phi^t$ sends the points $P_{BC}$ and $P_{AD}$ of  $\Q$ to the points $P_{BC}$ and $P_{AD}$, respectively, of $\Q^t$.
\end{proposition}

Let $s=(s_1, s_2)$ with $s_1+s_2>0$. Up to changing the order of $s_1$ and $s_2$, we can assume that $s_2 > s_1$, which in particular gives $s_2 > 0$. Consider the pentagon $\PP_s:=\overline{ABCDE} \in \mathbb H^2$ where $D$ is the ideal vertex, as in Figure \ref{figure_5}. The axes of the segments $\overline{AE}$ and $\overline{BC}$ intersect in a point $H$, which can be inside $\PP$, outside $\PP$ or on the side $\overline{EA}$, see Figure \ref{figure_5}. We will see in Lemma \ref{shear_penta} that this depends on the sign of $s_1$. Notice that the point $H$ lies on the bisector of the ideal angle at the vertex $D$.  
Let $M_{AE}, M_{BC}$ be the midpoints
of $\overline{AE}, \overline{BC}$ respectively, and $H_{AB}, H_{DC}, H_{DE}$ the projections of $H$ on the geodesics containing $\overline{AB}, \overline{DC}, \overline{DE}$ respectively. Denote $\mathscr{F}$ be the horocyclic foliation based in $D$. By construction there is one unique leaf $f \in \mathscr F$ passing through $H_{DE}$ and $H_{DC}$.

\begin{figure}[t!]  
\begin{center}
\psfrag{A}{$A$}
\psfrag{B}{$B$}
\psfrag{C}{$C$}
\psfrag{D}{$D$}
\psfrag{E}{$E$}
\psfrag{H3}{\tiny $H_{DE}$}
\psfrag{H2}{\tiny $H_{DC}$}
\psfrag{H1}{\tiny $H_{AB}$}
\psfrag{H}{\tiny $H$}
\begin{subfigure}[t]{0.3\textwidth}
\psfrag{M1}{\tiny $M_{AE}$}
\psfrag{M2}{\tiny $~M_{BC}$}
\includegraphics[width=4cm]{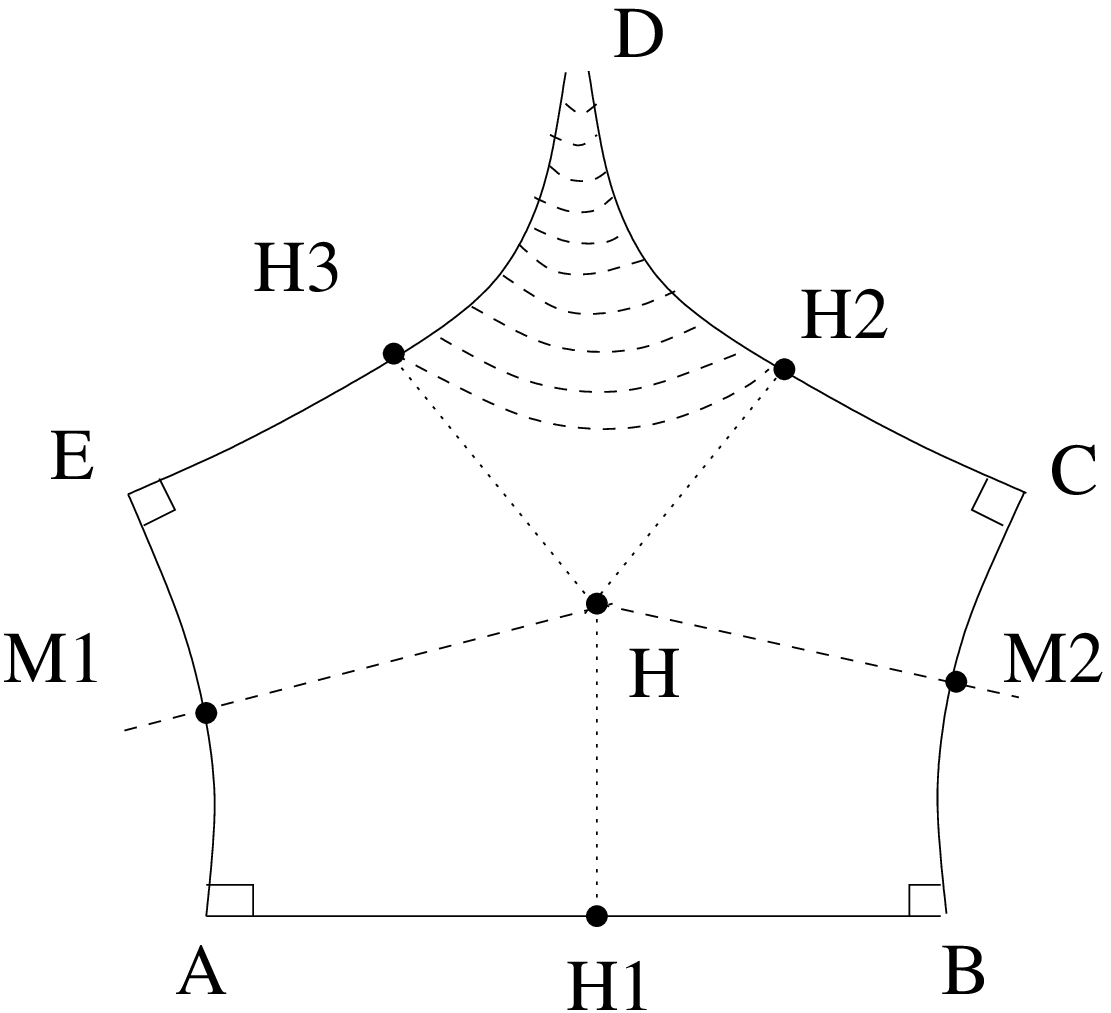}
\caption{$s_1, s_2 >0$}
\end{subfigure}
\hspace{0.6cm}
\begin{subfigure}[t]{0.28\textwidth}
\psfrag{M1}{\tiny $M_{BC}$}
\includegraphics[width=3.2cm]{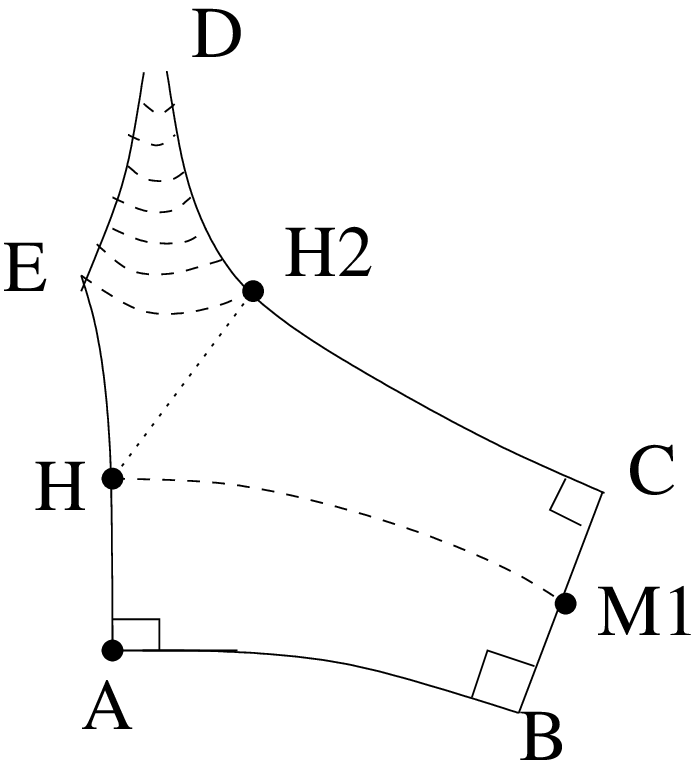}
\caption{$s_1=0, s_2>0$}
\end{subfigure}
\begin{subfigure}[t]{0.35\textwidth}
\psfrag{M1}{\tiny $M_{BC}$}
\psfrag{M2}[l]{\tiny ~$~M_{AE}$}
\includegraphics[width=4cm]{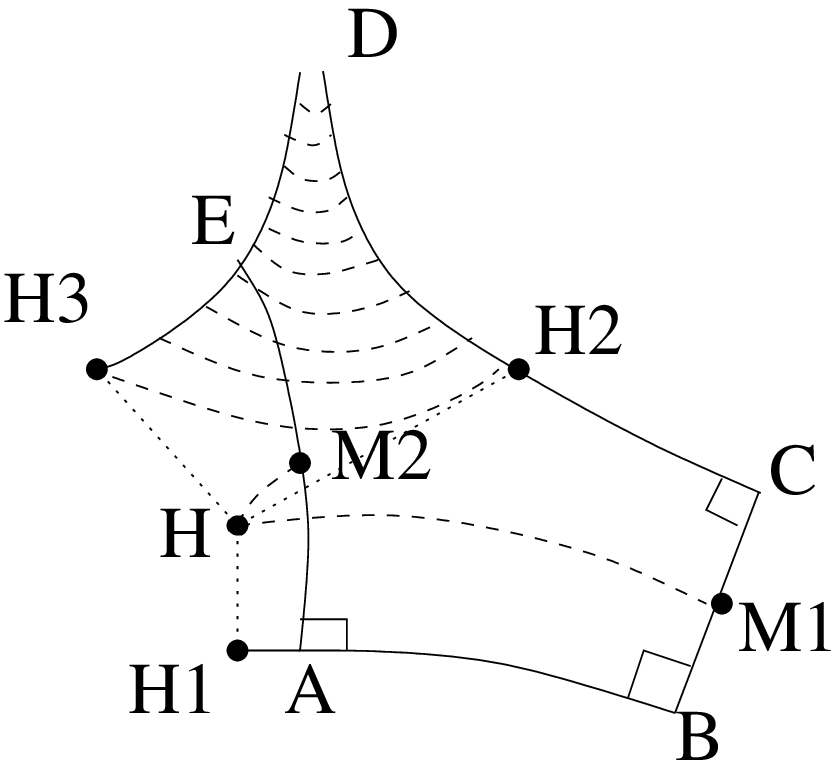}
\caption{$s_1<0, s_2>0$}
\end{subfigure}
\end{center}
\caption{Lemma \ref{shear_penta}}\label{figure_5}
\end{figure}

We compute the signed distances between $E$ and $H_{DE}$, and $B$ and $H_{AB}$. We define: 
\begin{enumerate}
\item $d_{\pm}(H_{DE},E):= \epsilon \cdot d(H_{DE}, E)$, where $\epsilon = 1$ when $H_{DE} \in \overline{ED}$ and $\epsilon = -1$ when $H_{DE} \not \in \overline{ED}$;
\item $d_{\pm}(H_{DC},C):= \epsilon \cdot d(H_{DC}, C)$, where $\epsilon = 1$ when $H_{DC} \in \overline{DC}$ and $\epsilon = -1$ when $H_{DC} \not \in \overline{DC}$; 
\item $d_{\pm}(H_{AB},A):= \epsilon \cdot d(H_{AB}, A)$, where $\epsilon = 1$ when $H_{AB} \in \overline{AB}$ and $\epsilon = -1$ when $H_{AB} \not \in \overline{AB}$;
\item $d_{\pm}(H_{AB},B):= \epsilon \cdot d(H_{AB}, B)$, where $\epsilon = 1$ when $H_{AB} \in \overline{AB}$ and $\epsilon = -1$ when $H_{AB} \not \in \overline{AB}$.
\end{enumerate}
By construction, it follows $d_{\pm}(H_{DE},E) = d_{\pm}(H_{AB}, A)$ and $d_{\pm}(H_{DC}, C) = d_{\pm}(H_{AB}, B)$. 
\begin{lemma}\label{shear_penta}
In the notation above, we have: 
\begin{align*}
d_{\pm}(H_{DE},E) &= d_{\pm}(H_{AB}, A) = \frac{s_1}{2} \\ 
d_{\pm}(H_{DC},C) &=  d_{\pm}(H_{AB}, B) = \frac{s_2}{2}.
\end{align*}
\end{lemma}
\begin{proof}
We will compute these lengths explicitly. Denote by $D^u$ the spike of $\PP_{s}^d$ which is the reflection of $D$. The universal covering of $\PP_{s}^d$ can be drawn in the upper half plane model of $\mathbb{H}^2$. We will denote by $\widetilde{D}$ a lift of $D$, by $\widetilde{D^u_+}$ the lift of $D^u$ at its left and by $\widetilde{D^u_-}$ the lift of $D^u$ at its right, see Figure \ref{pentagon in the plane}. We remark that the left part of Figure \ref{pentagon in the plane} is drawn in the disc model for an easier visualization, but all the computations are performed in the upper half plane model. Denote by $\widetilde{A}, \widetilde{B}, \widetilde{C}, \widetilde{E}$ the lifts of $A,B,C,E$, which form a copy of the pentagon with the vertex $\widetilde{D}$. Denote by $Z$ and $W$ the endpoints of the lift of the geodesic $l_1^d$.

We can assume $Z=0$, $\widetilde{D^u_+}= -1, \widetilde{D}=\infty$. Using the two triangles $\TT$ and $\TT'$, glued with shears $s_1, s_2$, we find $\widetilde{D^u_-} = e^{s_1}$. Similarly, using the ideal triangulation whose triangles spiral around $l_1^d$ in the opposite direction, glued with shears $-s_1, -s_2$, we find: $$W = \frac{e^{s_1+s_2}-1}{e^{s_2}+1}~.$$

We will now compute the coordinates of the points $\widetilde{A}$ and $\widetilde{E}$. The geodesic containing them is perpendicular to the lift of $l_1^d$, hence it lies on an Euclidean circle centered at $\widetilde{D^u_+}$. By an elementary computation its Euclidean radius is $r = \sqrt{\frac{e^{s_2}(e^{s_1}+1)}{e^{s_2}+1}}$. Hence the point $\widetilde{E}$ is the complex number $-1 + ri$. The point $\widetilde{A}$ can be found as the intersection of two circles:
$$\widetilde{A} = \frac{e^{s_1+s_2} - 1}{e^{s_1+s_2} + 2 e^{s_2} + 1} (1+ir).$$  

We now compute the intersection between the axis of the geodesic containing  $\widetilde{A}$ and $\widetilde{E}$ and the bisector of the ideal angle at $\widetilde{D}$. The bisector is the vertical line with real part equal to $\frac{e^{s_1}-1}{2}$:
$$\left\{\frac{e^{s_1}-1}{2} + it \mid t > 0 \right\}. $$
To compute its intersection with the axis, we will apply the  M\"obius transformation:
$$M: z \rightarrow \frac{-z-1-r}{z+1-r} ~.$$
The transformation acts in the following way:
$$ M\left(\widetilde{E}\right) = i, \ \ \ M\left(\widetilde{A}\right) = i \frac{1+2r +e^{s_2}(2 + e^{s_1} + 2r)}{e^{s_1+s_2} - 1} ~. $$ 

The axis of the segment between $M\left(\widetilde{E}\right)$ and $M\left(\widetilde{A}\right)$ is given by the equation
$$ \left\{ z \in \mathbb{C} \mid \mathrm{Im}(z)>0 \text{ and } z\bar{z} =  \frac{1+2r +e^{s_2}(2 + e^{s_1} + 2r)}{e^{s_1+s_2} - 1} \right\}~.$$

The imaginary part of the intersection between the bisector and the axis is given by the equation:
$$M\left(\frac{e^{s_1}-1}{2} + it\right) \overline{M\left(\frac{e^{s_1}-1}{2} + it\right)} = \frac{1+2r +e^{s_2}(2 + e^{s_1} + 2r)}{e^{s_1+s_2} - 1}~, $$
with solution 
$$t = \frac{1}{2} \sqrt{\frac{(e^{s_1}+1)(3 e^{s_1+s_2}-e^{s_1}-e^{s_2}-1)}{(e^{s_2}+1)}} ~. $$
This number is the imaginary part of the point $\widetilde{H}$, whose real part is $\frac{e^{s_1}-1}{2}$. 

The point $H_{DE}$ lies on the perpendicular line from $\widetilde{H}$ to the segment from $\widetilde{A}$ to $\widetilde{E}$. This line lies on a  Euclidean circle centered at $\widetilde{D^u_+}$, hence the imaginary part of $H_{DE}$ is equal to the radius of this circle, which is the absolute value of $\widetilde{H}+1$, namely
$$\sqrt{\frac{e^{s_1+s_2}(e^{s_1}+1)}{e^{s_2}+1}}.$$ 
The number $d_{\pm}(H_{DE},E)$ is the log of the ratio of the imaginary parts of  $\widetilde{H_{DE}}$ and $\widetilde{E}$: 

$$d_{\pm}(H_{DE}, E) = \log \frac{\sqrt{\frac{e^{s_1+s_2}(e^{s_1}+1)}{e^{s_2}+1}} }{\sqrt{\frac{e^{s_2}(e^{s_1}+1)}{e^{s_2}+1}}} = \frac{s_1}{2}~.$$  
For $d_{\pm}(H_{DC},C)$, notice that $d_{\pm}(H_{DE},E) + d_{\pm}(H_{DC},C) = l_1 = \frac{1}{2}(s_1 + s_2)$.
\end{proof}

\begin{figure}[ht]
\begin{center}
\begin{subfigure}[t]{0.40\textwidth}
\psfrag{A'}{$\widetilde A$}
\psfrag{B'}{$\widetilde B$}
\psfrag{Z}{$Z $}
\psfrag{W}{$W $}
\psfrag{B}{$\widetilde D $}
\psfrag{E}[t][l]{\tiny $\widetilde E~$}
\psfrag{C'}{\tiny $\widetilde C$}
\psfrag{C}{$\widetilde{D_{-}^u} $}
\psfrag{A}[B]{$\widetilde{ D_{+}^u} $}
\psfrag{s}[Br]{\tiny $s_1$}
\psfrag{s'}[Br][r]{\tiny$s_2$}
\psfrag{s''}{\tiny $-s_1$}
\psfrag{s'''}{\tiny $-s_2$}
\includegraphics[width=6cm]{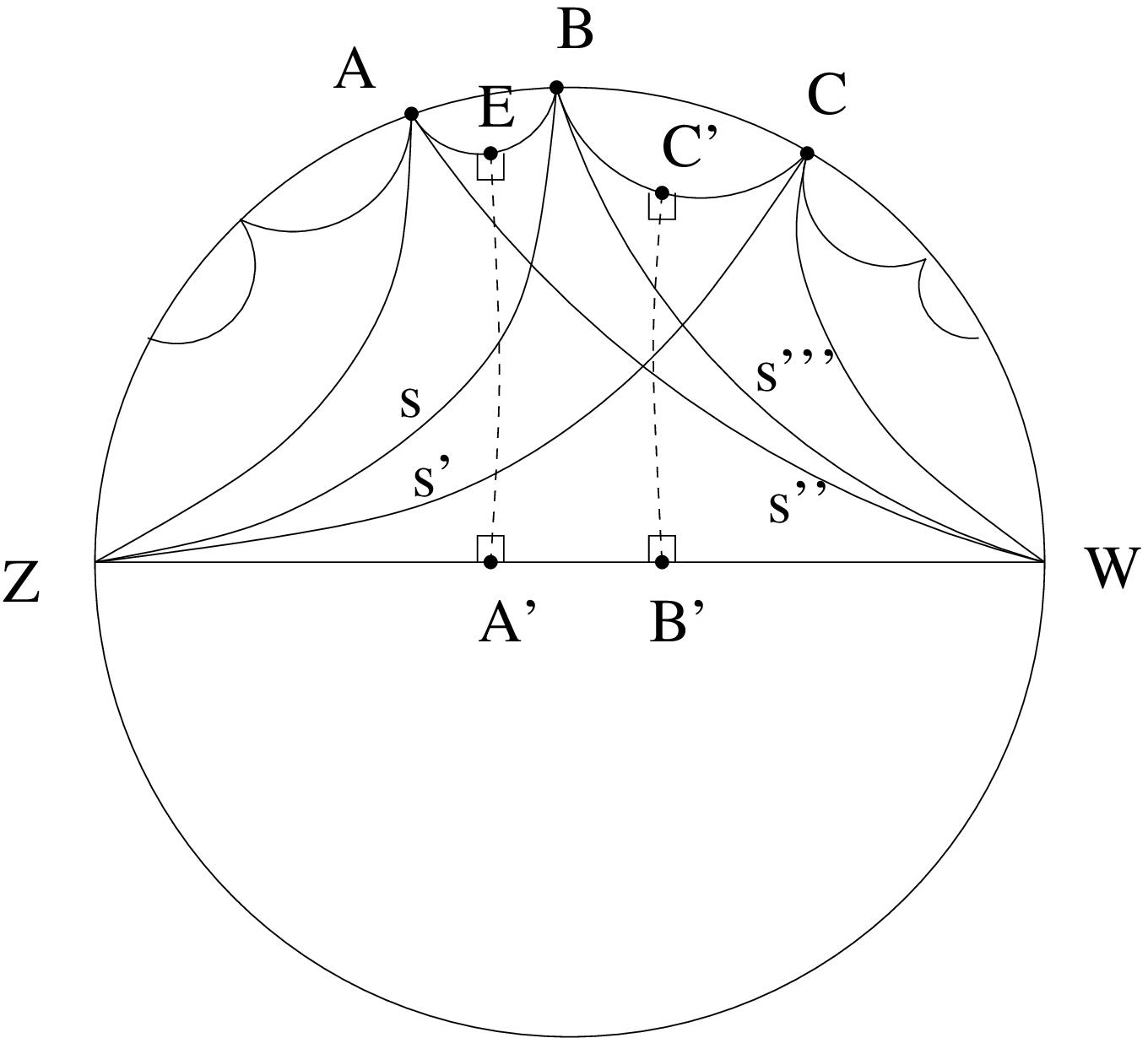} 
\end{subfigure}
\begin{subfigure}[t]{0.45\textwidth}
\psfrag{D}{$\widetilde D = \infty$}
\psfrag{E}{\tiny $\widetilde E$ }
\psfrag{C}{\tiny $\widetilde C$}
\psfrag{D'}[c][b]{$\widetilde{D_{+}^u}$}
\psfrag{Z}[c][b]{$Z$}
\psfrag{W}[c][b]{$W$}
\psfrag{D''}[c][b]{$\widetilde{D_{-}^u}$}
\includegraphics[width=5.5cm]{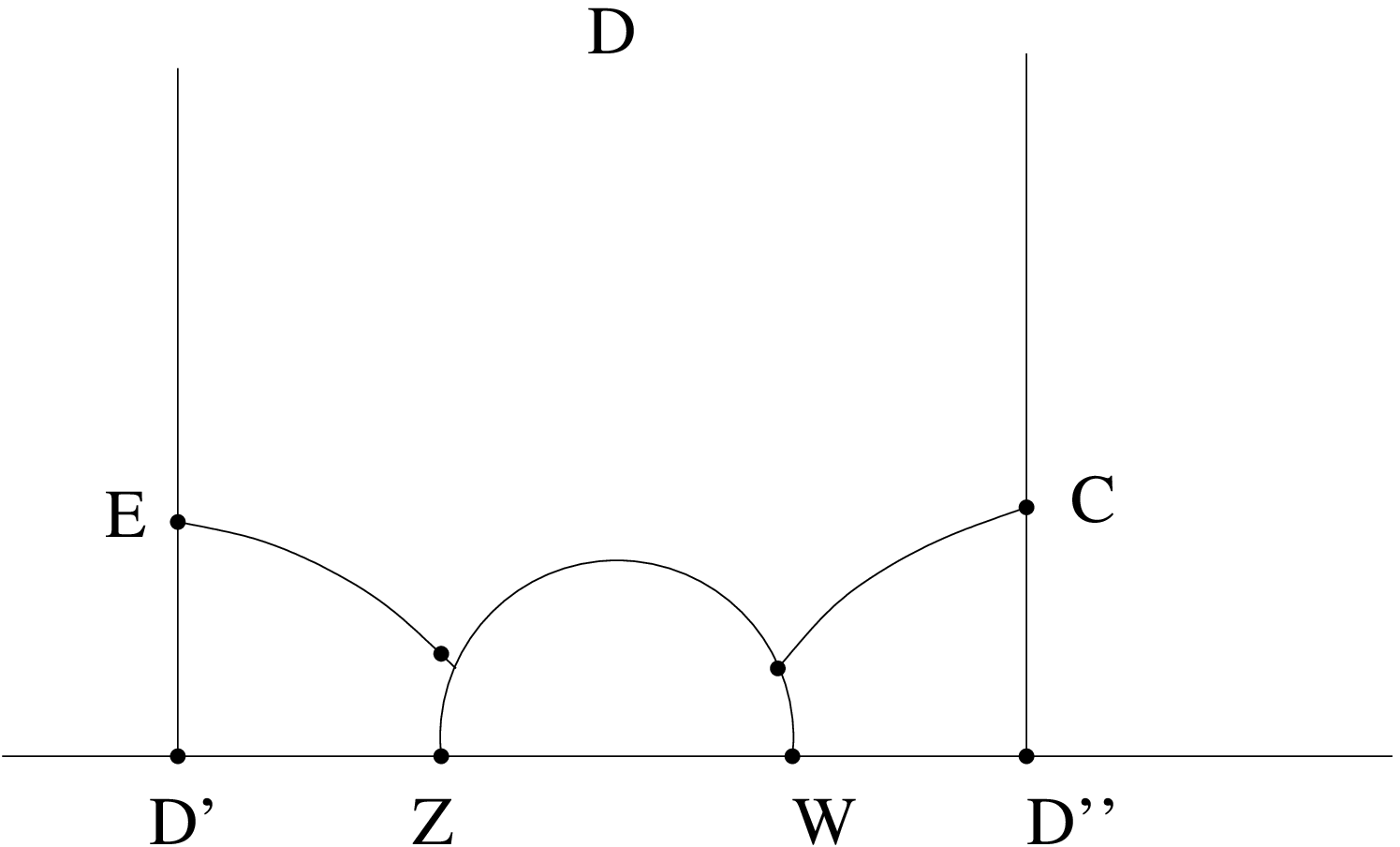}
\end{subfigure}
\caption{The universal covering of $\PP^d$} \label{pentagon in the plane}
\end{center}
\end{figure}

The following result is an immediate consequence of Lemmas \ref{lemma:quad} and \ref{shear_penta}.

\begin{proposition}\label{shear_penta2}
Let $t\geq 0$, and let  $\phi^t: \PP \to \PP^t$ be the generalized stretch map as in Lemma \ref{lemma:quad}. Then, the map $\phi^t$ sends the point $H_{DC}$ of  $\PP$ to the point $H_{DC}$ of $\PP^t$. Moreover, if $s_1 \geq 0$, $\phi^t$ sends the points $H_{DE}$ and $H_{AB}$ of  $\PP$ to the points $H_{DE}$ and $H_{AB}$, respectively, of $\PP^t$.
\end{proposition}

Let $s=(s_1, s_2, s_3)$. Consider the hexagon $\HH_s:=\overline{ABCDEF} \subset \mathbb H^2$ as in Figure \ref{fig:shear_hexa}: $l_1$ is the edge $\overline{CD}$, $l_2$ is the edge $\overline{AB}$, and $l_3$ is the edge $\overline{EF}$.
Consider the axes of the segments $\overline{BC}$, $\overline{DE}$, $\overline{FA}$. The three axes all meet in a common point $H$. Let $H_{AB}, H_{DC}, H_{EF}$ be the orthogonal projections of $H$ on the geodesics containing the segments $\overline{AB}$, $\overline{DC}$ and $\overline{EF}$ (see also \cite{PapYam}).

Consider two consecutive vertices $V, W$ of the hexagon, the orthogonal projection $H_{VW}$ of $H$ on the geodesic $\overline{VW}$. We define the signed distance of $H_{VW}$ from $V$:
$$d_{\pm}(H_{VW},V):= \epsilon \cdot d(H_{VW}, V),$$
where $\epsilon = 1$ if $H_{VW}$ lies on the geodesic ray starting from $V$ that contains $\overline{VW}$, and $\epsilon = -1$ if  $H_{VW}$ lies on the geodesic ray starting from $V$ that does not contain $\overline{VW}$.

\begin{lemma}\label{lem:shear_hexa}
We have: 
\begin{align*}
d_{\pm}(H_{EF}, F) &= d_{\pm}(H_{AB}, A) = \frac{s_1}{2}\,, \\ 
d_{\pm}(H_{DC}, C) &= d_{\pm}(H_{AB}, B) = \frac{s_2}{2}\,, \\ 
d_{\pm}(H_{EF}, E) &= d_{\pm}(H_{DC}, D) = \frac{s_3}{2}\,.
\end{align*}
\end{lemma}
\begin{proof}
The equalities between signed distances come from the construction of $H$ (see Pa\-pa\-do\-pou\-los-Yamada \cite{PapYam}). Using these equalities, we find a linear system, which admits a unique solution, see Figure \ref{fig:shear_hexa}:
\begin{align*}
d_{\pm}(H_{DC}, D) + d_{\pm}(H_{AB}, B) &= \ell(l_1)\,,\\
d_{\pm}(H_{AB}, B) + d_{\pm}(H_{EF}, F) &= \ell(l_2)\,,\\
d_{\pm}(H_{EF}, F) + d_{\pm}(H_{DC}, D) &= \ell(l_3)\,.
\end{align*}
Since by our initial assumptions we have $\frac{s_i}{2} + \frac{s_{i+1}}{2} = \ell(l_{i+2})$ for $i = 1, 2, 3$, we conclude.
\end{proof}

\begin{figure}[t!]   
\begin{center}
\psfrag{A}[b]{$A$}
\psfrag{B}[l]{$B$}
\psfrag{C}[l]{$C$~}
\psfrag{D}{$D$}
\psfrag{E}[Bl][T]{$E$}
\psfrag{H3'}[l]{ $F~$}
\psfrag{H2}{\tiny $H_{DC}$}
\psfrag{H1}{\tiny $H_{AB}$}
\psfrag{H1'}[b]{$A$}
\begin{subfigure}[t]{0.3\textwidth}
\psfrag{F}[b]{$F$}
\psfrag{H3}[tl][b]{\tiny $H_{EF}$}
\includegraphics[width=3.6cm]{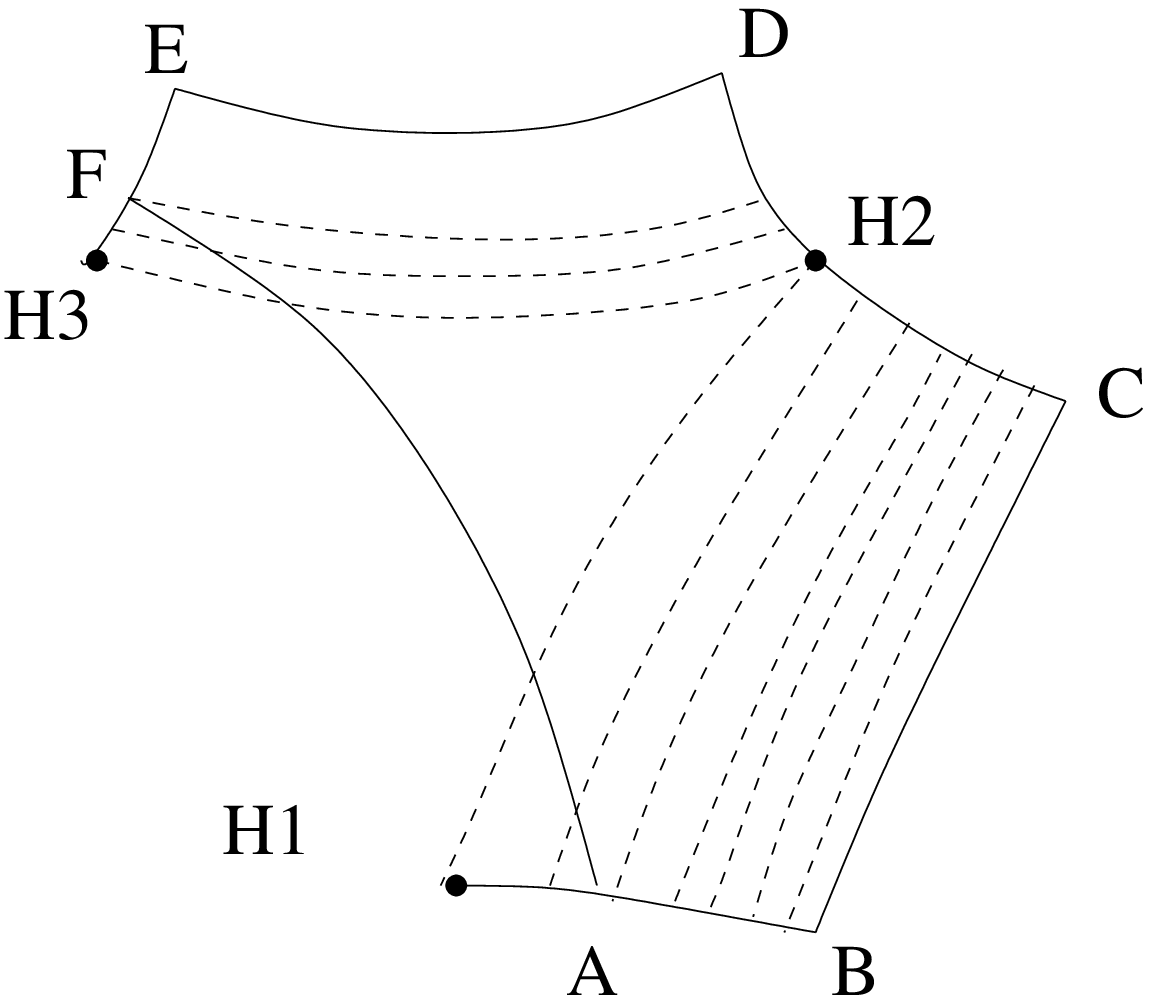}
\caption{$s_1 < 0, s_2, s_3 > 0$}
\end{subfigure}
\begin{subfigure}[t]{0.3\textwidth}
\psfrag{F}[l]{$F$}
\psfrag{H3}{\tiny $H_{EF}$}
\includegraphics[width=3.8cm]{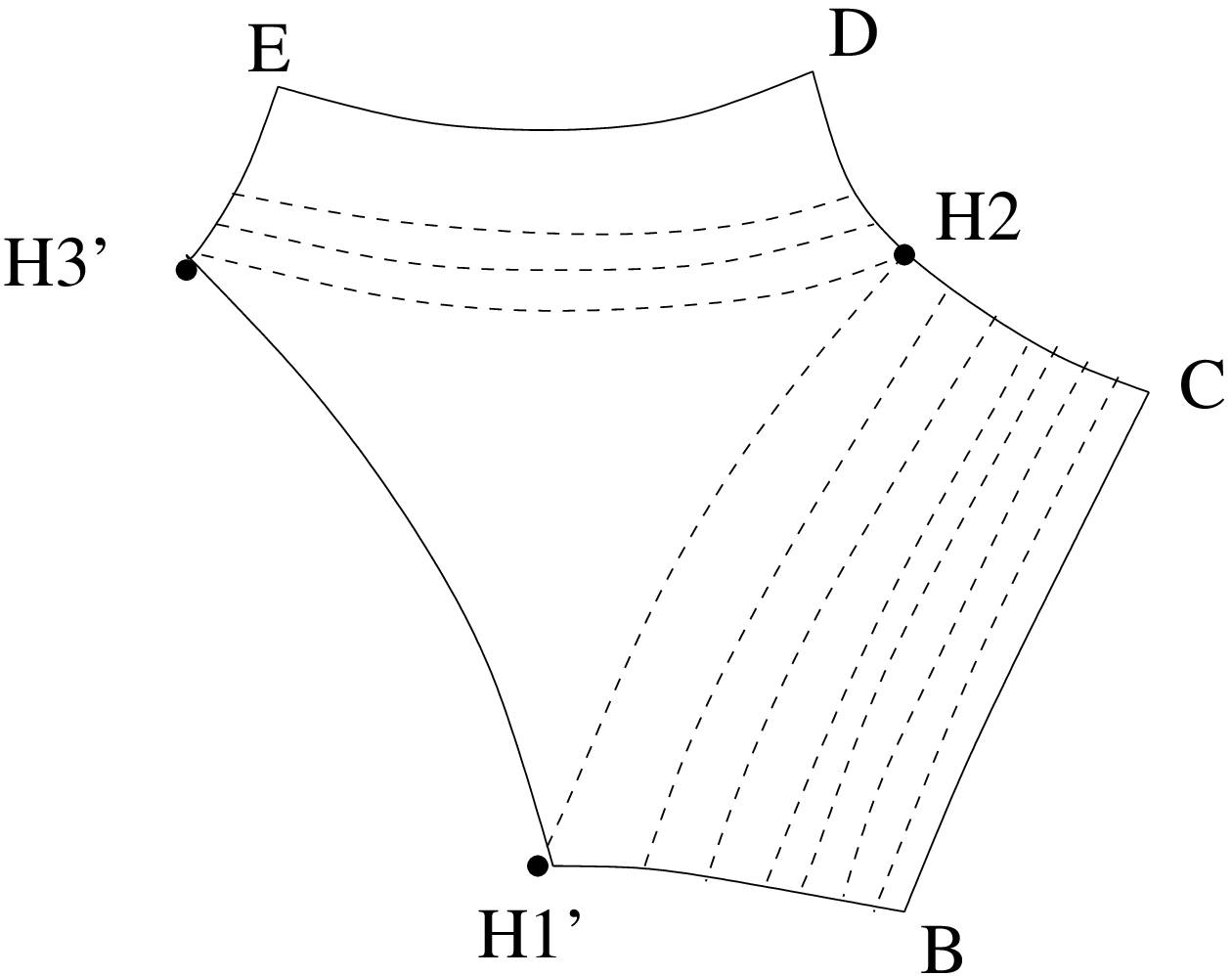}
\caption{$s_1 = 0, s_2, s_3 > 0$ }
\end{subfigure}
\begin{subfigure}[t]{0.3\textwidth}
\psfrag{F}[b]{$F$}
\psfrag{H3}{\tiny $H_{EF}$}
\includegraphics[width=3.7cm]{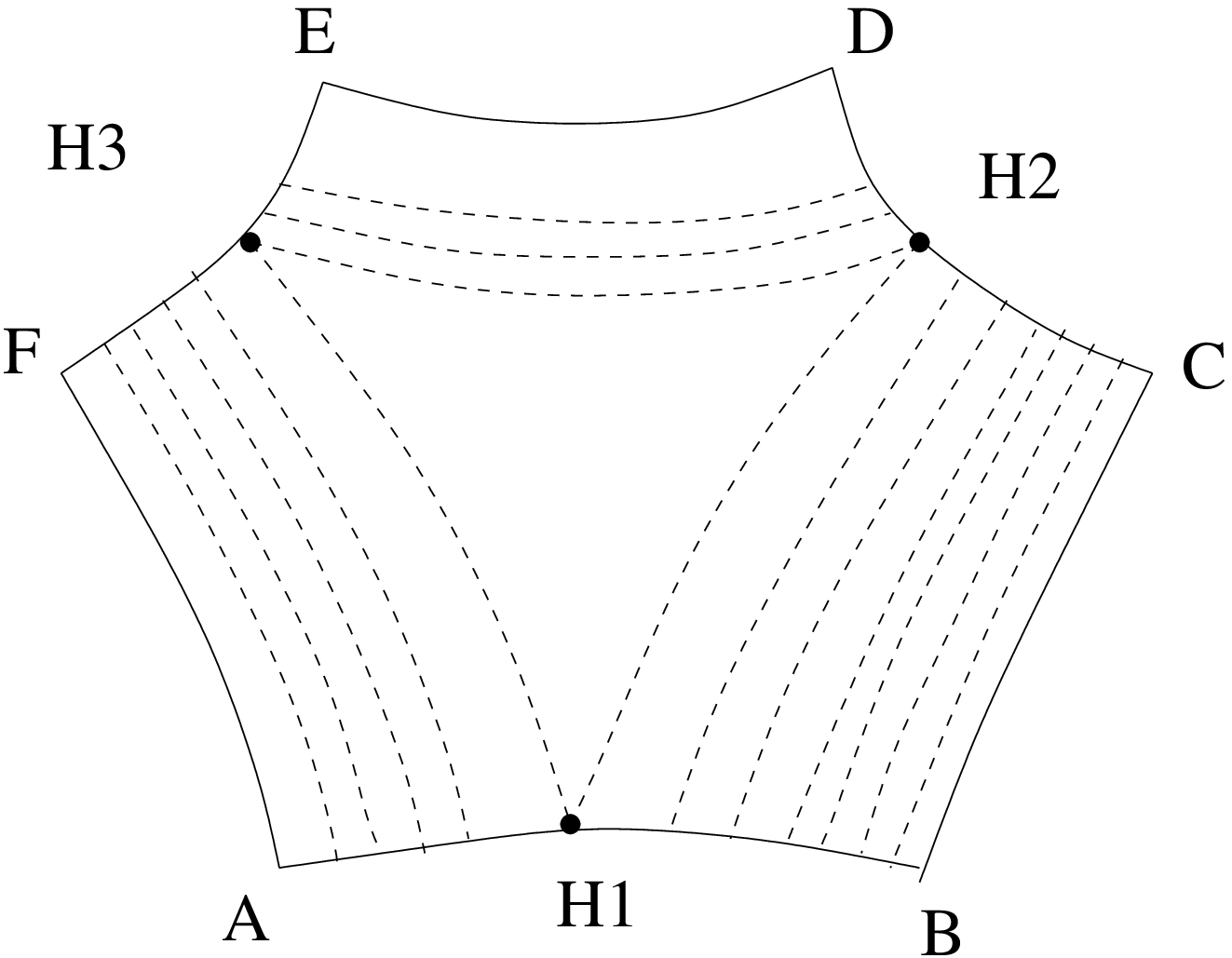}
\caption{$s_1, s_2, s_3 > 0$}
\end{subfigure}
\end{center}
\caption{Lemma \ref{lem:shear_hexa}}\label{fig:shear_hexa}
\end{figure}

The following result is an immediate consequence of Lemmas \ref{lemma:quad} and \ref{lem:shear_hexa}.

\begin{proposition}\label{shear_hexa2}
Let $t\geq 0$, and let  $\phi^t: \HH \to \HH^t$ be the generalized stretch map as in Lemma \ref{lemma:quad}. Then, the map $\phi^t$ sends the point $H_{DC}$ of  $\HH$ to the point $H_{DC}$ of $\HH^t$. Moreover, if $s_1 \geq 0$, $\phi^t$ sends the points $H_{DE}$ and $H_{AB}$ of  $\HH$ to the points $H_{DE}$ and $H_{AB}$, respectively, of $\HH^t$.
\end{proposition}

\subsection{The horocyclic foliation}

We will now construct a partial foliation $\mathcal{K}$, called the \emph{horocyclic foliation}, in every quadrilateral and pentagonal piece, see Figure \ref{quadrilateral_1}. Our aim is to prove Propositions \ref{lem:horocyclic foliation square} and \ref{lem:horocyclic foliation penta}. 

For a quadrilateral piece $\Q$, see Figure \ref{quadrilateral_1}. Recall that $C^u,D^u$ denote the vertices of $\Q_s^d$ which are the reflection of $C,D$. Denote by $O_C, O_D$ the points of $\overline{CD}$ that are the nearest point projections of $C^u, D^u$ respectively. We consider a partial foliation $\mathcal{K}_C$ whose leaves are all the horocycles centered at $C$ which intersect the side $\overline{CD}$ between $C$ and $O_C$. Similar definition for a partial foliation $\mathcal{K}_D$.
\begin{figure}[htbp]
\begin{center}
\psfrag{D}{$D$}
\psfrag{C}{$C$}
\psfrag{D'}{$D^u$}
\psfrag{C'}{$C^u$}
\psfrag{X}{\tiny $O_D$}
\psfrag{Y}{\tiny $O_C$}
\psfrag{Z}{\tiny $O_Q$}
\includegraphics[width=5cm]{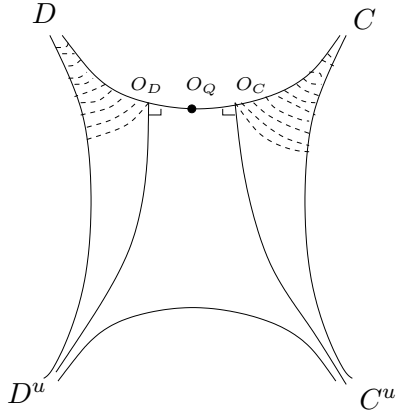}
\caption{Horocyclic foliation}\label{quadrilateral_1}
\end{center}
\end{figure}
\begin{definition}[Horocyclic foliation] \label{def:horfol4}
The \emph{horocyclic foliation} $\mathcal{K}$ is the union of the two partial foliations $\mathcal{K}_C$ and $\mathcal{K}_D$.
\end{definition}

From the computations in the proof of Lemma \ref{shear_quads} we have:
$$d(O_\Q, O_C) = d(O_\Q, O_D) = \tfrac{1}{2}\log(1+e^{-s}). $$ For $d \geq \tfrac{1}{2}\log(1+e^{-s})$, we denote by $h_C^d$ the only horocycle in $\mathcal{K}_C$ where the distance between $O_\Q$ and its intersection with the edge $\overline{CD}$ equals $d$.

\begin{lemma}   \label{lem:horocyclic foliation square}
The map $\phi^t$ in Lemma \ref{lemma:quad} maps the leaf $h_C^d$ of $\mathcal{K}$ in $\Q$ to the leaf $h_{C}^{e^t d}$ of $\mathcal{K}$ in $\Q^t$ affinely. Similarly, $\phi^t$ maps the leaf $h_D^d$ of $\mathcal{K}$ in $\Q$ to the leaf $h_D^{e^t d}$ of $\mathcal{K}$ in $\Q^t$ affinely.    
\end{lemma}
\begin{proof}
The map $\phi^t$ is the average of the two maps $\Psi^t$ and $\sigma^t \circ \Psi^t \circ \sigma$ (see the proof of Lemma \ref{lemma:quad}). The map $\Psi^t$ is represented on the left hand side of Figure \ref{fig:horocyclic foliation square}, the map $\sigma^t \circ \Psi^t \circ \sigma$ on the right hand side. Each map sends the horocycle $h_C^{d}$, with $d \geq \tfrac{1}{2}\log(1+e^{-s})$, to a horocycle, hence their average will also send this horocycle to a horocycle, which must be $h_{C}^{e^t d}$ by part (2) and (3) of Lemma \ref{lemma:quad}. This horocycle is still in $\mathcal{K}$ for $\Q^t$,  because a simple computation shows that  if $d \geq \tfrac{1}{2}\log(1+e^{-s})$ then  $e^t d \geq \tfrac{1}{2}\log(1+e^{-e^t s})$. 
\end{proof}

\begin{figure}[htb]
\begin{center}
\psfrag{D}{$D$}
\psfrag{C}{$C$}
\psfrag{O}{\small $O_C$}
\psfrag{O'}{\small $~O_D$}
\includegraphics[width=10cm]{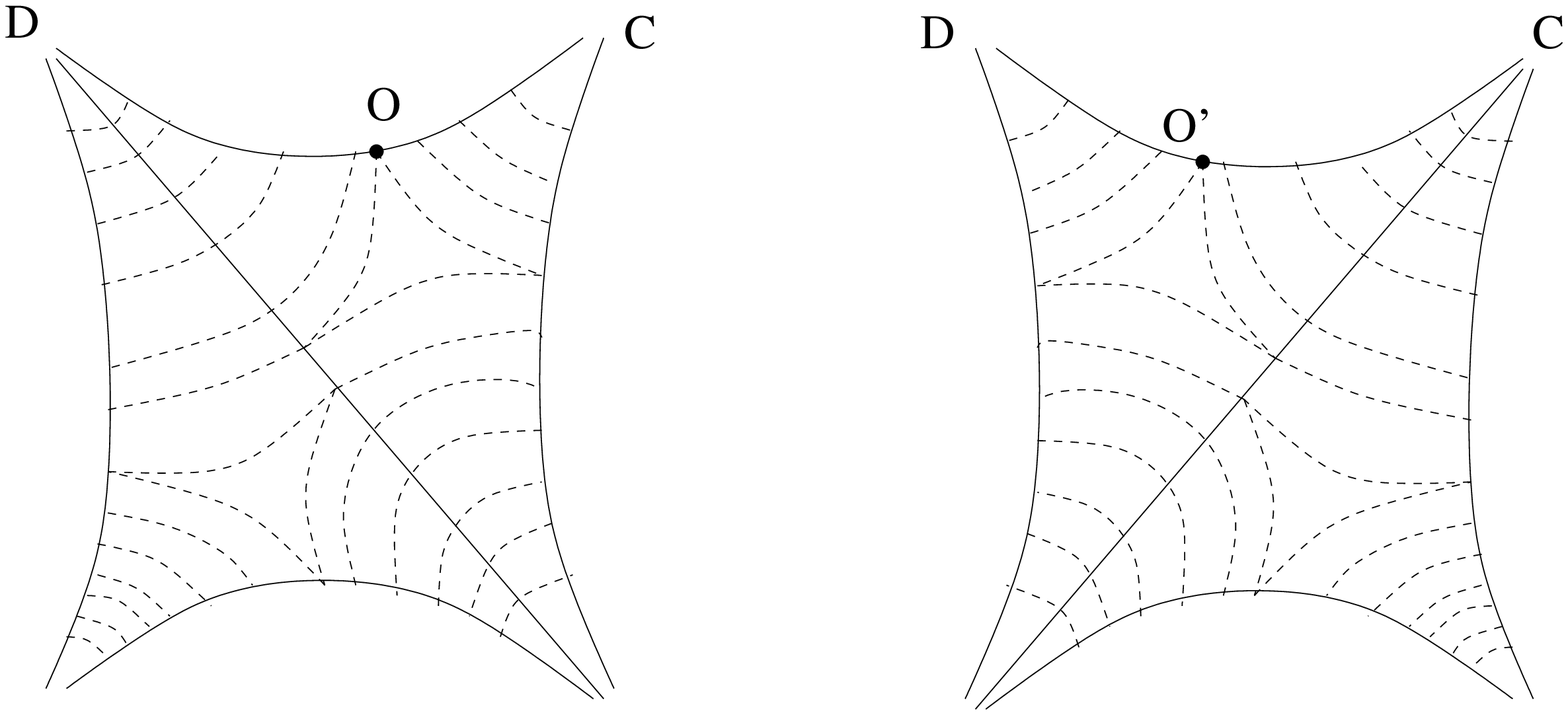}
\caption{Lemma \ref{lem:horocyclic foliation square}}   \label{fig:horocyclic foliation square}
\end{center}
\end{figure}

Now, let $\PP$ be a pentagonal piece. Using the notation as in the proof of Lemma \ref{shear_penta}, denote by $O_W$ the point of $\overline{\widetilde{D}\widetilde{E}}$ that is the nearest point projection of $W$. 

\begin{definition}[Horocyclic foliation] \label{def:horfol5}
We define the \emph{horocyclic foliation} $\mathcal{K}$ as the partial foliation whose leaves are all the horocycles centered at $\widetilde{D}$ which intersect $\overline{\widetilde{D}\widetilde{E}}$ between $\widetilde{D}$ and $O_W$. 
\end{definition}

\begin{lemma} The point $O_W$ lies on the half-line $\overline{\widetilde{D}\widetilde{E}}$, and we have:
$$d(O_W, \widetilde{E}) = \frac{s_2}{2}+\tfrac{1}{2}\log\left(\dfrac{e^{s_1} + 1}{e^{s_2} + 1}\right)\,. $$
\end{lemma}
\begin{proof}
From the computation in the proof of Lemma \ref{shear_penta}, we have: 
$$\mathrm{Im}(O_W) = W+1 = r^2\,.$$
Since $r>1$, we see that $O_W$ is above $\widetilde{E}$ and their distance is
\begin{equation*}
d(O_W, \widetilde{E}) = \log(r) = \tfrac{1}{2}\log\left(\dfrac{e^{s_2} (e^{s_1} + 1)}{e^{s_2} + 1}\right)\,.\qedhere 
\end{equation*}  
\end{proof}

For $d \geq \frac{s_2}{2}+\tfrac{1}{2}\log\left(\dfrac{e^{s_1} + 1}{e^{s_2} + 1}\right)$, we denote by $h^d$ the only horocycle in $\mathcal{K}$ where the distance between $O_W$ and its intersection with the edge $\overline{\widetilde{D}\widetilde{E}}$ equals $d$. 
 
\begin{lemma}  \label{lem:horocyclic foliation penta}
The map $\phi^t$ from Lemma \ref{lemma:quad} maps the leaf $h^d$ of $\mathcal{K}$ in $\PP$ to the leaf $h^{e^t d}$ of $\mathcal{K}$ in $\PP^t$ affinely.     
\end{lemma}
\begin{proof}
The map $\phi^t$ is the average of the two maps $\Psi^t$ and $\sigma^t \circ \Psi^t \circ \sigma$ (see the proof of Lemma \ref{lemma:quad}). Each of the two maps sends the horocycle $h^{d}$, with $d \geq \frac{s_2}{2}+\tfrac{1}{2}\log\left(\dfrac{e^{s_1} + 1}{e^{s_2} + 1}\right)$, to a horocycle, hence their average will also send this horocycle to a horocycle, which must be $h^{e^t d}$ by part (2) and (3) of Lemma \ref{lemma:quad}. This horocycle is still in $\mathcal{K}$ for $\PP^t$,  because a computation shows that $d \geq \frac{s_2}{2}+\tfrac{1}{2}\log\left(\dfrac{e^{s_1} + 1}{e^{s_2} + 1}\right)$ implies 
\begin{equation*}
e^t d \geq \frac{e^t s_2}{2}+\tfrac{1}{2}\log\left(\dfrac{e^{e^t s_1} + 1}{e^{e^t s_2} + 1}\right)\,. \qedhere
\end{equation*}
\end{proof}

\begin{remark}  
It might seem more natural to extend the horocyclic foliation $\mathcal{K}$ on a quadrilateral piece $\Q$ until the point $O_\Q$ as in Figure \ref{foliation quad}. Similarly, for a pentagonal piece $\PP$, it might seem more natural to extend it until the point $H_{DE}$ as in Figure \ref{figure_5}. Unfortunately $\phi^t$ does not map all the horocycles of these extended foliations to horocycles. This property only holds for the leaves of $\mathcal{K}$. 
\end{remark}

%% file: 07-Boundary_Block.tex
\section{The boundary block}  \label{sec:bdry_block}

Let $X\in \T(S)$ and $\lambda$ a maximal lamination on $X$. In this section and the following, we construct some auxiliary surfaces that we will use to define our generalized stretch lines. 
Here we define the boundary block of $\lambda$ in $X$, that is the subset of $X$ obtained as the union of all the geometric pieces that are not ideal triangles. It comes equipped with a finite maximal lamination $\lambda_B$, consisting of the boundary leaves of these pieces. The boundary block is non-empty if and only if at least one of the leaves of $\lambda$ is orthogonal to the boundary of $X$. 
After the definition of the boundary block, we will describe how to ``stretch'' it using the results of Section \ref{sec:stretch_1}.

\subsection{Definition of the boundary block} \label{subsec:def boundary block}

We define the \emph{boundary block} of $\lambda$ as the subset $B \subset X$ obtained as a union of all the geometric pieces of $X \setminus \lambda$ that have at least one edge on $\partial X$, that is, quadrilaterals, pentagons and hexagons:  
 $$ B := \bigcup \{ \mathcal G_i \mid \mathcal G_i \mbox{ is a geometric piece of } X \setminus \lambda \mbox{ of type (2), (3) or (4) } \} \subset X\,.$$  

By construction $B$ is a (possibly disconnected) 2-manifold with boundary. Notice that its boundary in general might not be compact. The inclusion map $B \hookrightarrow X$ induces via pull-back a Riemannian metric on $B$, which turns $B$ into a (possibly disconnected) complete hyperbolic surface of finite volume. Hence the connected components of $B$ are convex hyperbolic surfaces.
Notice that the inclusion map $B \hookrightarrow X$ is a 1-1 local isometry, but not necessarily an isometric embedding in the sense of metric spaces. Indeed, the infinite geodesics in the boundary of the quadrilateral and pentagonal pieces might spiral in a bounded region of $X$, but they are not contained in a bounded region for the hyperbolic metric on $B$. 

The boundary  of $B$ contains compact and non-compact components. We will denote by $\partial^c B$ the union of the compact components of $\partial B$, and by $\partial^{nc} B$ the union of the non-compact components of $\partial B$:
$$\partial B = \partial^c B \cup \partial^{nc} B.$$ 
The compact boundary components are also boundary components of $X$. The non-compact boundary components are bi-infinite geodesics that are boundary of quadrilaterals. Every non-compact boundary component of $B$ has two ideal vertices corresponding to two spikes.

\begin{definition}[Cycle in $\partial^{nc} B$]   \label{def:cycle}
A \emph{cycle} in $\partial^{nc} B$ is a cyclically ordered set $c:=\{b_1, \dots, b_s\}$ of components of $\partial^{nc} B$ such that for every $i$ the geodesic $b_i$ shares a spike with $b_{i-1}$ and with $b_{i+1}$, and $b_s$ shares a spike with $b_1$. We will denote by $Q_i$ the quadrilateral piece containing $b_i$ in its boundary. We will also denote by $a_i$ the spike shared by $b_i$ and $b_{i+1}$ and by $a_s$ the spike shared by $b_s$ and $b_1$ (see Figure \ref{Thurston:cycles}). 
\end{definition}
The boundary block $B$ has finitely many cycles $c_1, \dots, c_{m}$ in $\partial^{nc} B$. Every cycle $c_i$ in $\partial^{nc} B$ determines a (unique) simple closed geodesic $\gamma_i$ in its homotopy class. 
\begin{definition}[Crown spanned by a cycle] \label{def:crown}
The \emph{crown} spanned by $c_i$ is the subsurface $C_i :=\mathrm{ConvHull}(c_i, \gamma_i) \subset B$ which is the convex hull of $c_i$ and $\gamma_i$. By construction $C_i$ is a complete hyperbolic surface whose interior is topologically a cylinder, and $\partial C_i = \gamma_i \cup c_i$ (see Figure \ref{crowns}). 
\end{definition}

\begin{figure}[htb]
\begin{center}
\begin{subfigure}[t]{0.4\textwidth}
\psfrag{b_i}[B][T]{$b_i$}
\psfrag{B}[T]{$\partial X$}
\includegraphics[width=3.6cm]{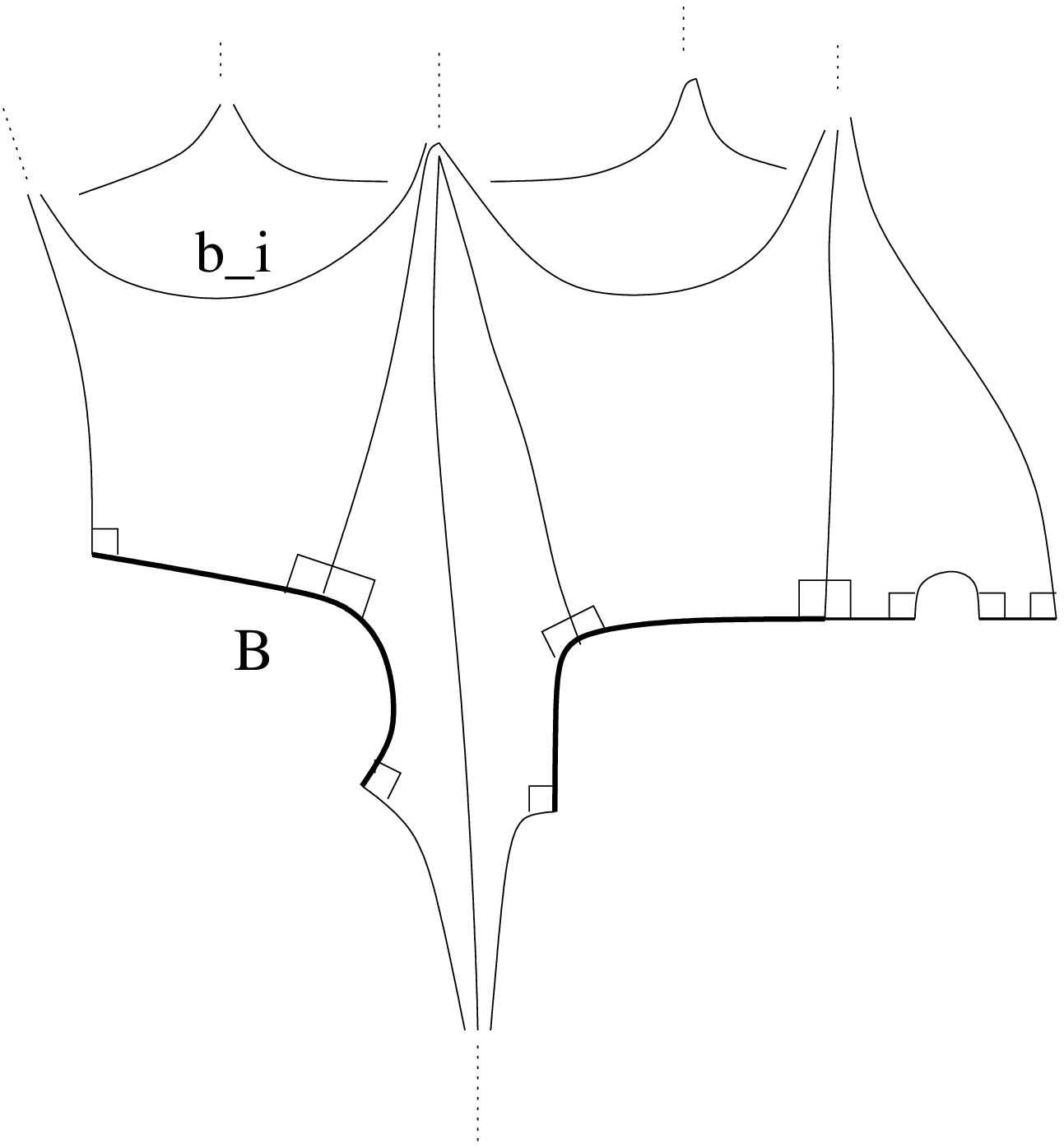}
\caption{A cycle $c \subset \partial^{nc} B$}\label{Thurston:cycles}
\end{subfigure}
\begin{subfigure}[t]{0.4\textwidth}
\psfrag{Gi}{$\zeta_i$}
\psfrag{Gk}[Tl][T]{$\zeta_{j}$}
\psfrag{Gj}[Tr]{\small ~$~\zeta_{j}$}
\includegraphics[width=3.8cm]{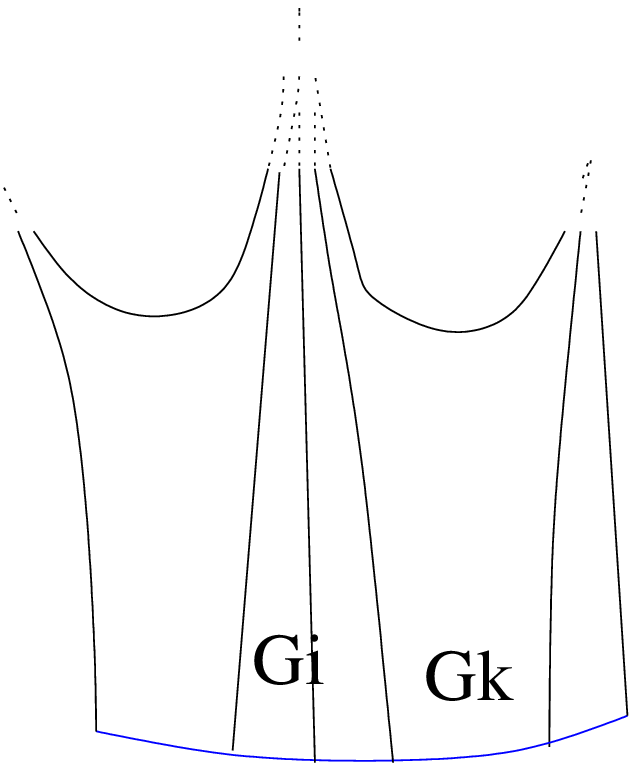}
\caption{The crown $C \subset B$ of $c$}\label{crowns}
\end{subfigure}
\end{center}
\caption{A cycle $c \subset \partial^{nc} B$ and its associated crown $C \subset B$}
\end{figure}
\begin{lemma}
If $c_j, c_k \in \partial^{nc} B$ are distinct cycles, then $C_j$ and $C_k$ are disjoint.
\end{lemma} 
\begin{notation}
We will use the following notation: 
\begin{itemize}
\item $C:= \bigcup_i C_i$, where $C_i$ is the crown associated to the cycle $c_i \subset \partial^{nc} B$; 
\item $B_C:= \overline{B \setminus C} \subset B \subset X$ is the closure of the complement of $C$ in $B$; 
\item $\Gamma = \bigcup \gamma_i$ is the finite union of all the pairwise disjoint simple closed geodesics $\gamma_i$; 
\item $X_C := \overline{X \setminus B_C} \subset X$ is the complement of $B_C$ in $X$. 
\end{itemize}
 \end{notation}
The following proposition is an immediate consequence of our constructions:
\begin{proposition}
The maximal lamination $\lambda$ induces a decomposition of $X$ as follows:
$$X = X_C \cup B \mbox{ with } X_C \cap B = C\,,$$ 
where $X_C$ is a (possibly disconnected) finite hyperbolic surface with the metric induced by $X$ and $$\partial X_C = \Gamma \cup (\partial X \setminus \partial B)~.$$ Furthermore, all of the following holds: 
\begin{itemize}
\item $\lambda_B := \{ l \in \lambda ~|~ l \mbox{ is a leaf entirely contained in } B \}$ is a maximal lamination for $B$; 
\item $\lambda_{X_C} := \{ l \in \lambda ~|~ l \mbox{ is a leaf entirely contained in } X_C \}$ is a lamination for $X_C$;  
\item $\lambda_B \cap \lambda_{X_C}$ is the union of the non-compact boundary components of $B$; 
\item $\lambda = \lambda_B \cup \lambda_{X_C}$.
\end{itemize}
\end{proposition}

Notice that by definition the following holds: 
\begin{itemize}
\item $\lambda_{X_C}$ does not contain any leaf that hits the boundary of $X$ perpendicularly;
\item $X_C \setminus \lambda_{X_C} = \mathring{C} \cup \bigcup\{ \mathring{\G} ~|~ \G \mbox{ is a triangular geometric piece of } X\setminus \lambda \}\,.$
\end{itemize}
\begin{figure}[htb]
\begin{center}
\psfrag{l}[Br]{$\lambda$}
\psfrag{BC}[b]{$B_C$}
\psfrag{C}{$C$}
\psfrag{SB}{$X_C$}
\psfrag{B}{$B$}
\psfrag{g}{$\Gamma$}
\includegraphics[width=9cm]{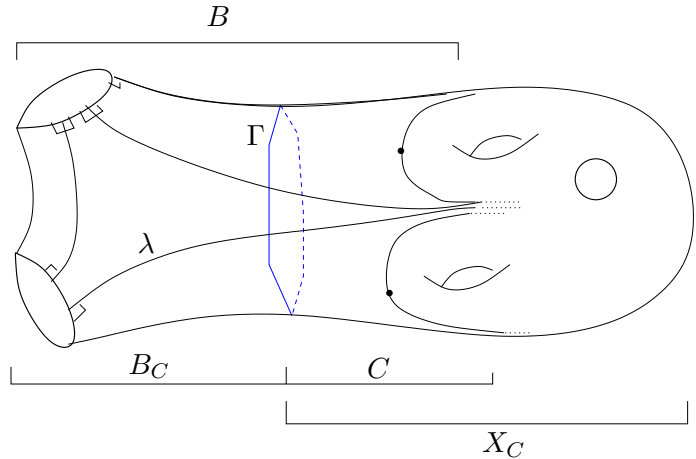}
\caption{The $\lambda$-decomposition of $X$ when $C$ is connected and $\Gamma = \{ \gamma \}$}\label{BCdecomposition}
\end{center}
\end{figure}

\subsection{Stretching the boundary block}   \label{sec:stretchB} 

For every $t\geq 0$ we will now stretch the boundary block along $\lambda_B$. We will denote the new hyperbolic surface by $B_{\lambda_B}^t$ or simply $B^t$. 

As a first step, for every geometric piece $\mathcal{G}_i$ involved in the definition of $B$ consider its stretched analog $\mathcal{G}_i^t$ defined in Section \ref{sec:stretch_1}. We define $B^t$ to be the (possibly disconnected) surface obtained by glueing together these $\mathcal{G}_i^t$'s following the glueing pattern of the corresponding $\mathcal{G}_i$'s in $B$: 
$$ B^t := \bigcup \{ \mathcal{G}_i^t ~| ~ \mathcal{G}_i \mbox{ is a geometric piece of } X \setminus \lambda  \mbox{ of type (2), (3) or (4) }  \}\,.$$
The edges of the pieces are glued pairwise via isometries according to the following rules: 
\begin{itemize}
\item When two pieces are glued along edges of finite length, our definition of the $\mathcal{G}_i^t$ guarantees that the corresponding edges have the same length, hence there is only one way to glue them by an isometry;
\item When two pieces are glued along half-infinite edges, again there is only one way to glue them by an isometry and making the common vertices coincide (for the quadrilateral pieces, the half-infinite edges are $AD$ or $BC$ in Figure \ref{foliation quad}, for the pentagonal pieces the half-infinite edges are $AD$ or $CD$ in Figure \ref{figure_5}); 
\item Two pieces are glued together along a bi-infinite edge if and only if both are quadrilaterals. If $s_0$ is the shear between $\mathcal G_1$ and $\mathcal G_2$ for the surface $B$ according to Definition \ref{def:shear}, we glue $\mathcal G_1^t$ and $\mathcal G_2^t$ by an isometry with shear $e^t \cdot s_0$. 
\end{itemize}
Notice that all the glueings are well-defined. As a consequence, we have the following.
\begin{lemma}
$B^t$ is a (possibly disconnected) complete hyperbolic surface of finite volume diffeomorphic to $B$. 
\end{lemma}
By construction every cycle of bi-infinite leaves $c_i \subset \partial^{nc} B$ corresponds to a cycle $c_i^t \subset \partial^{nc} B^t$. Every $c_i^t$ determines a simple closed geodesic $\gamma_i^t$ and a crown $$C_i^t:= \mathrm{ConvHull}(c_i^t, \gamma_i^t) \subset B^t~.$$ The $C_i^t$'s are all disjoint, we denote their (disjoint) union by $$C^t := \bigcup C_i^t \subset B^t~.$$

\begin{proposition}[Existence of a stretch map for $B$]\label{prop:B}
For every $t\geq0$ there exists a continuous map $\beta^t: B \to B^t$ homotopic to the identity with the following properties: 
\begin{enumerate}
\item $\beta^t(\partial B) = \partial B^t$;
\item $\beta^t$ stretches the arc-length of the leaves of $\lambda_B$ by $e^t$; 
\item on every geometric piece $\mathcal G$ in $B \setminus \lambda_B$ the map $\beta^t$ restricts to $\beta^t|_{\mathcal{G}} = \phi^t: \mathcal{G} \to \mathcal{G}^t$ as in Lemma \ref{lemma:quad}; 
\item $\Lip(\beta^t) = e^t$.  
\end{enumerate}
\end{proposition}

\begin{proof}
We define $\beta^t$ by glueing together the maps $\phi^t: \mathcal{G}_i \to \mathcal{G}_i^t$ as in Lemma \ref{lemma:quad} following the glueings of the $\G_i^t$'s in $B^t$. We have that $\Lip(\beta^t) \leq e^t$ by Lemma \ref{LOC-2-GLOB}, and $\Lip(\beta^t) \geq e^t$ because $\beta^t$ stretches the arc-length of the leaves of $\lambda_B$ by $e^t$.  
\end{proof}

%% file: 08-Triangulated_Surface.tex
\section{The triangulated surface}    \label{sec:triangulated_surface}
Here we will construct our second auxiliary surface. Again, let $X \in \T(S)$ and $\lambda$ be a maximal lamination on $X$. In Section \ref{subsec:def boundary block}, we defined the subsurface $X_C$ endowed with the lamination $\lambda_{X_C}$. Notice that $\lambda_{X_C}$ is not maximal for $X_C$: it divides $X_C$ into triangles and the crowns $C_j$ (see Figure \ref{BCdecomposition}). We will define a new complete hyperbolic surface: the \emph{triangulated surface} $X_A$. This surface will extend $X_C$, that is, it will come equipped with an isometry $g:X_C \hookrightarrow X_A\,.$
 
The triangulated surface is not embedded in $X$, it will be constructed by suitably gluing new triangles to the crowns $C_j$ of $X_C$. The surface $X_A$ will also be equipped with a maximal lamination $\lambda_A \supset \lambda_{X_C}$ such that $g(\lambda \cap X_C) \subset \lambda_A \cap X_C$ and $X_A \setminus \lambda_A$ is a union of ideal triangles (Figure \ref{S_A}).

\begin{figure}[ht!]
\begin{center}
\psfrag{g}{$\Gamma \sim_f \hat \Gamma$}
\psfrag{R}{$\hat R$}
\psfrag{C'}{$\hat{C}$}
\psfrag{C}[b]{$C \sim_f \hat{C}$}
\psfrag{SB}{$X_C$}
\psfrag{l}{}
\includegraphics[width=10cm]{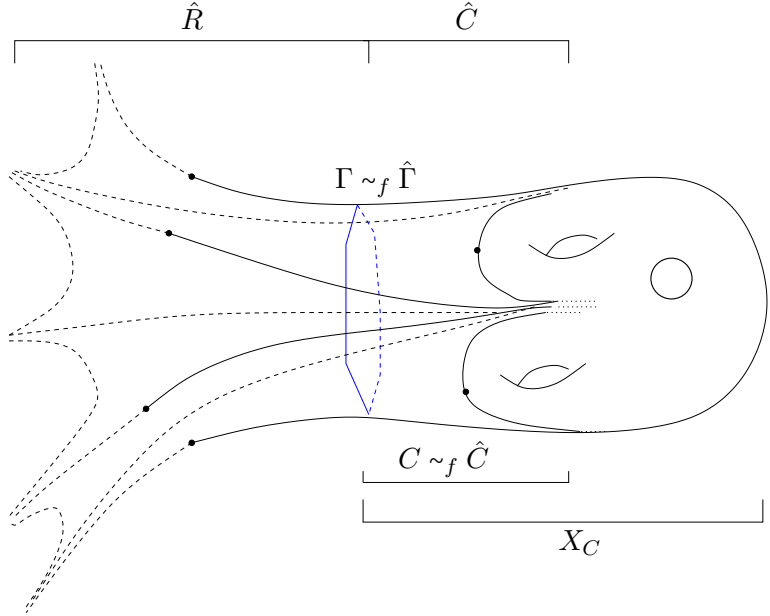}
\caption{The surface $X_{A}$}\label{S_A}
\end{center}
\end{figure}

We will then stretch the triangulated surface $X_A$ by using Bonahon's theory of cataclysms (see Section \ref{Bonahon}), which uses a cocycle to deform a triangulated surface. 
We will first stretch the auxiliary cylinders in Section \ref{stretch_A^t}, and finally we will define the stretched triangulated surface, by constructing a suitable cocycle in Section \ref{sec:S_A}. The stretched boundary block and the stretched triangulated surface will be key ingredients in the construction of the generalized stretch lines discussed in Section \ref{sec:gsl}.

\subsection{Extension of a crown}\label{comfort_cyl_sub}  As a first step in the construction of $X_A$, we will first work with crowns. Consider a crown $C_j\subset X_C$. We will extend it to a complete hyperbolic surface $A_j$, homeomorphic to an annulus, equipped with an isometry 
$$f_j:C_j \hookrightarrow A_j $$ 
and a finite maximal lamination $\delta_{A_j}$ with $f_i(\lambda \cap C_j) \subset \delta_{A_j} \cap C_j\,.$

For every crown $C_j$ and every geometric piece $\G_k$, we define the region $\zeta_k^j := \G_k \cap C_j$, see Figure \ref{crowns}. Now each crown is decomposed as $C_j= \bigcup_k \zeta_k^j$. For every region $\zeta_k^j$, we will now define a triangulated ideal polygon $\widehat{\G_k}$ with an isometry $\zeta_k^j \hookrightarrow \widehat{\G_k}$. After this, we will define $A_j$ by replacing each $\zeta_k^j$ with $\widehat{\G_k}$. 

We will consider three cases. Every region $\zeta_k^j$ has one or two spikes: if it has two spikes, then $\G_k$ is a quadrilateral, see Case (1) below; if it has one spike there are two possibilities, either $\G_k$ is a pentagon, as in Case (2) below, or $\G_k$ is again a quadrilateral and one of the adjacent pieces is again a quadrilateral glued to $\G_k$ along a bi-infinite edge, see Case (3) below.

\paragraph{Defining $\widehat {\mathcal G_k}$: Case (1)}
We assume that $\zeta_k^j$ has two spikes and $\zeta_k^j \subset \G_k = \Q$ is a quadrilateral, denoted by $\Q:=\overline{ABCD} \subset \mathbb H^2$, see Figure \ref{QD}. We orient the finite edge $\overline{CD}$ from $C$ to $D$, according to the orientation of $\Q$. 
We define $\widehat  \G_k$ as the double of $\Q$ along $\overline{CD}$, as in Figure \ref{QD}:
\[\widehat  \G_k = \Q^d:= \overline{ABC'D'}\,.\]

This ideal quadrilateral will be triangulated by adding the diagonal $\overline{BD'}$. The diagonal can be chosen in a consistent way using the orientation of $\overline{CD}$. We will call the edges $\overline{BC'}$ and $\overline{AD'}$ \emph{special edges}, and the points $C$ and $D$ \emph{special points}.

\begin{figure}[ht!]
\begin{center}
\begin{subfigure}[t]{0.3\textwidth}
\psfrag{D'}{$D'$}
\psfrag{C'}[l]{$C'$}
\psfrag{D}[tl]{$D$}
\psfrag{C}[t]{$C$}
\psfrag{B}[b]{$B$}
\psfrag{A}{$A$}
\includegraphics[width=3.5cm]{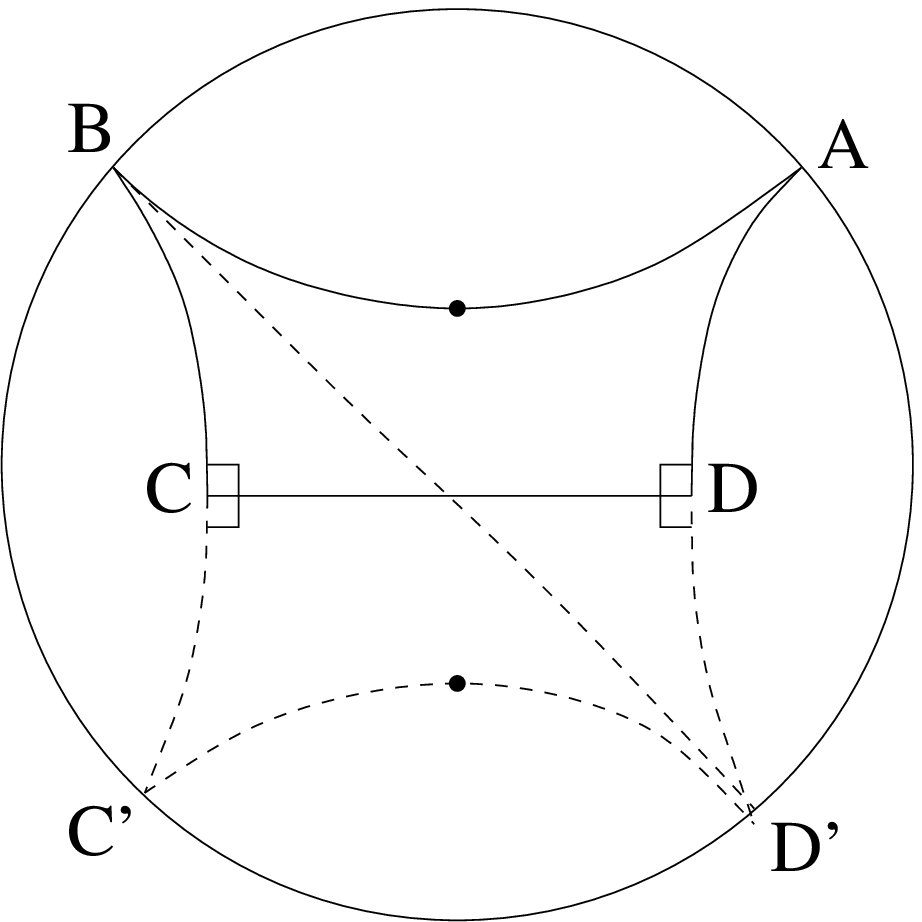}
\caption{$\widehat \Q:=\overline{ABC'D'}$}\label{QD}
\end{subfigure}
\begin{subfigure}[t]{0.3\textwidth}
\psfrag{B'}[b]{$B'$}
\psfrag{B}[l]{$B$}
\psfrag{E'}{$E'$}
\psfrag{E}{$E$}
\psfrag{D'}{$D'$}
\psfrag{C'}{$C'$}
\psfrag{D}[l]{$D$}
\psfrag{C}[l]{$C$}
\psfrag{B}{$B$}
\psfrag{A}{$A$}
\includegraphics[width=4cm]{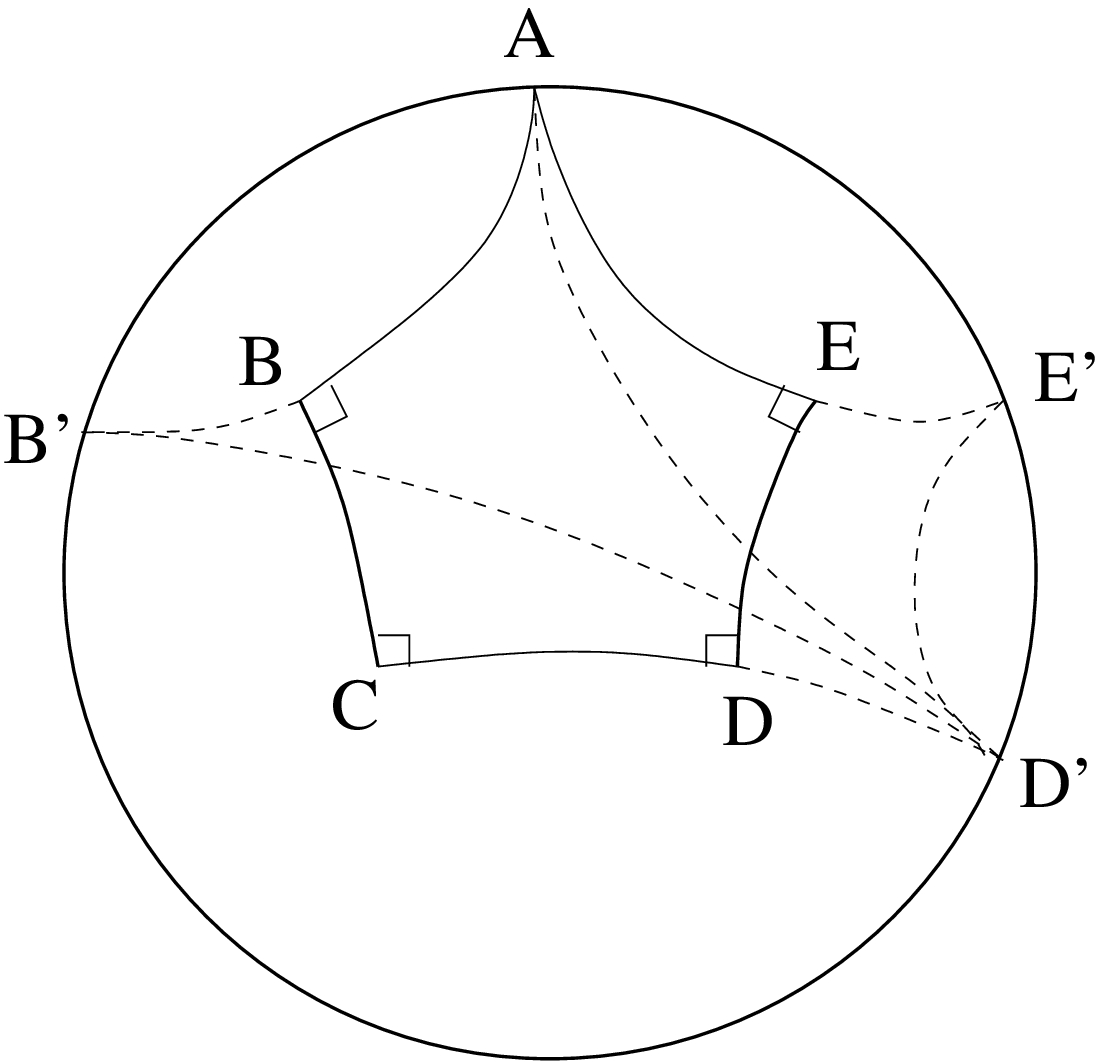}
\caption{$\widehat {\mathcal P}:= \overline{AB'D'E'} $}\label{pentagon}
\end{subfigure}
\begin{subfigure}[t]{0.3\textwidth}
\psfrag{E}{$E$}
\psfrag{D'}{$D'$}
\psfrag{C'}{$C'$}
\psfrag{D}{$D$}
\psfrag{E}{$E$}
\psfrag{E'}[c][r]{$E'$}
\psfrag{F'}[l]{$F'$}
\psfrag{F}[c][l]{~$F$}
\psfrag{C}{$C$}
\psfrag{B'}{$B'~$ }
\psfrag{B}[c][l]{~$B$ }
\psfrag{A}[b][l]{$A$ }
\includegraphics[width=4cm]{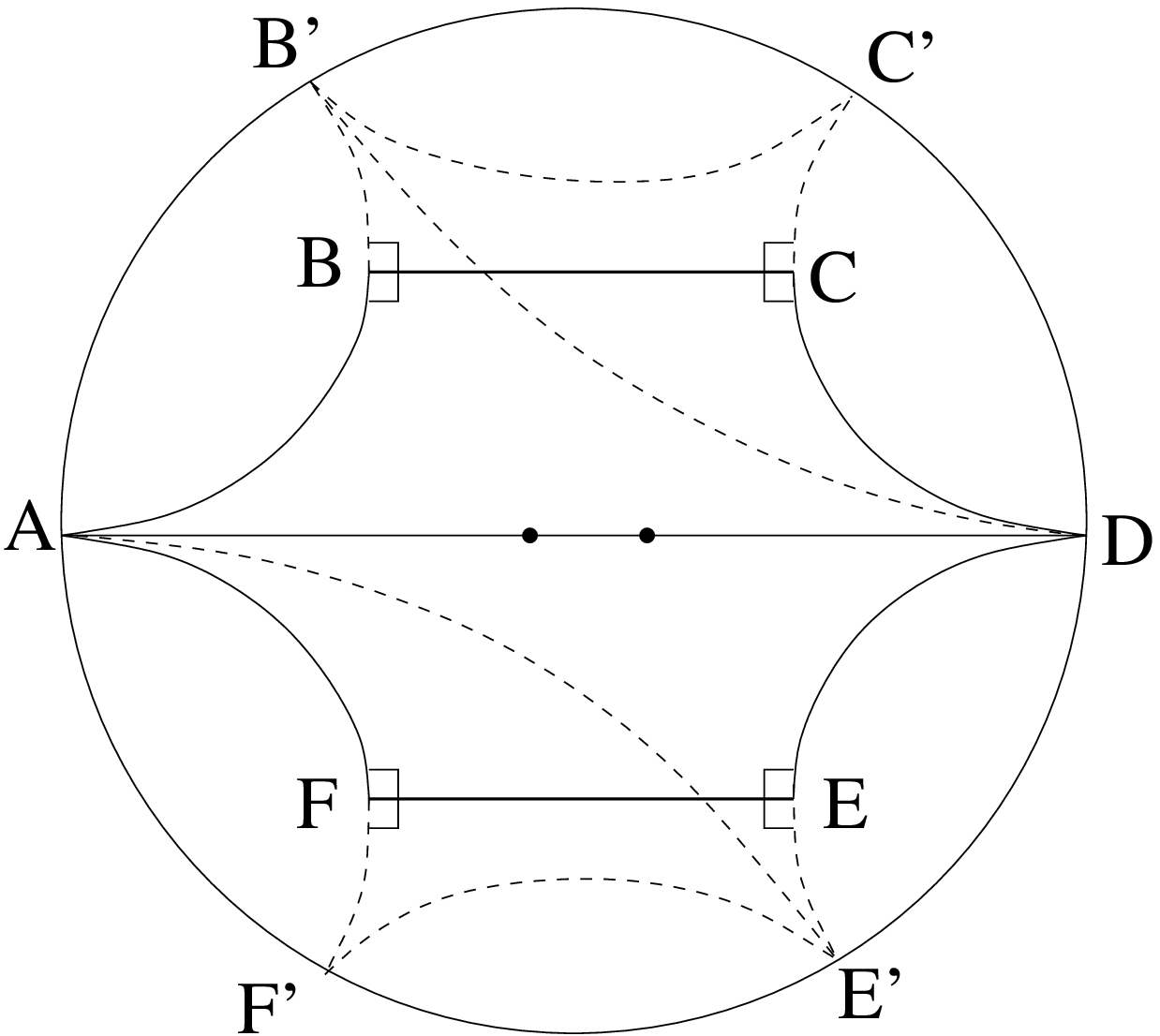}
\caption{$\widehat{\Q} \cup \widehat{\Q'}$}\label{Q6}
\end{subfigure}
\end{center}
\caption{Construction of $\widehat{\G_k}$.}
\end{figure}

\paragraph{Defining $\widehat {\G_k}$: Case (2)} We assume that $\zeta_k^j$ has one spike and $\zeta_k^j \subset \G_k = \PP$ is a pentagon, denoted by $\PP:= \overline{ABCDE} \subset \mathbb H^2$ as in Figure \ref{pentagon}.  
We orient the edge $\overline{CD}$ of $\PP$ opposite to the spike from $C$ to $D$, following the orientation of $\mathcal P$. We define $D', B', E' \in \partial \mathbb H^2$ as the extremes of the geodesics $\overline{CD}, \overline{AB}, \overline{AE}$ on the side of $D, B, E$ respectively.  
We define 
\[\widehat {\G_k}:= \overline{AB'D'E'}\,.\]
This ideal quadrilateral will be triangulated by adding the diagonal $\overline{AD'}$. We will call the edges $\overline{AB'}$ and $\overline{AE'}$ \emph{special edges}, and the points $B$ and $E$ \emph{special points}.

\paragraph{Defining $\widehat {\G_k}$: Case (3)} We assume that $\zeta_k^j, \zeta_{k+1}^j$ both have one spike and $\zeta_k^j \subset \G_k = \Q$,  $\zeta_{k+1}^j \subset \G_{k+1} = \Q'$, where both $\Q,\Q'$ are quadrilaterals, denoted by $\Q:= \overline{ABCD}, \Q':= \overline{AFED}$, where $\overline{AD}$ is the bi-infinite edge they share, see Figure \ref{Q6}. We define $B', C', E', F' \in \partial \mathbb H^2$ as the endpoints of the geodesics $\overline{AB}, \overline{DC}, \overline{DE}, \overline{AF}$ on the side of $B, C, E, F$ respectively.

If the common spike of $\zeta_k^j$ and $\zeta_{k+1}^j$ is $D$, we define
\[\widehat{\G_k}:= \overline{ADC'B'}~ \mbox{ and } \ \ \  \widehat{\G_{k+1}}:= \overline{ADE'}\,.\]
The ideal quadrilateral $\widehat{\G_k}$ will be triangulated by adding the diagonal $\overline{B'D}$. We will call the edges $\overline{AD}$, $\overline{C'D}$ its \emph{special edges}, and the points $O_\Q$ and $C$ its \emph{special points}. For the ideal triangle $\widehat{\G_{k+1}}$, we will call the edges $\overline{AD}$, $\overline{DE'}$ its \emph{special edges} and the points $O_{\Q'}$ and $E$ its \emph{special points}.

If the common spike of $\zeta_k^j$ and $\zeta_{k+1}^j$ is $A$, we define
\[\widehat{\G_k}:= \overline{AB'D}~ \mbox{ and } \ \ \ \widehat{\G_{k+1}}:= \overline{AF'E'D}\,.\]
For the ideal triangle $\widehat{\G_k}$ we will call the edges $\overline{AD}$, $\overline{AB'}$ its \emph{special edges} and the points $O_{\Q}$ and $B$ its \emph{special points}. The ideal quadrilateral $\widehat{\G_{k+1}}$ will be triangulated by adding the diagonal $\overline{AE'}$. We will call the edges $\overline{AD}$, $\overline{AF'}$ its \emph{special edges}, and the points $O_\Q'$ and $F$ its \emph{special points}.  

\paragraph{Definition of the auxiliary cylinder}     \label{comfort_cyl_sub_glue} 
In all cases above, we have an isometry $\zeta_k^j \hookrightarrow \widehat{\G_k}$. Now we define $A_j$ glueing together the $\widehat {\mathcal G_k}$'s according to the glueing pattern of the associated $\mathcal \zeta_k^j \subset \G_k$: 

\[A_j := \bigcup \left\{ \widehat {\mathcal G_k} ~|~  \widehat {\mathcal G_k} \mbox{ is the triangulated ideal polygon tailored to } \G_k \supset \zeta_k^j \mbox{ as above } \right\}/\sim\,, \]
where $\widehat {\mathcal G_h}$ and $\widehat {\mathcal G_{h'}}$ are glued together if and only if $\zeta_h^j$ and $\zeta_{h'}^j$ are adjacent in $C_j$ or, equivalently, if and only if their associated geometric pieces $\G_h$ and $\G_{h'}$ are adjacent in $B$. More precisely, consider two consecutive geometric pieces $\G_h, \G_{h+1}$, and denote by $e_h \subset \G_h$ and $e_{h+1} \subset \G_{h+1}$ the two edges which are glued together in $B$. We denote by $\widehat{e}_h \subset \widehat{\G_h}$ and  $\widehat{e}_{h+1} \subset \widehat{\G_{h+1}}$ the corresponding special edges of $\widehat{\G_{h}}, \widehat{\G_{h+1}}$, and by $\widehat{P_h} \in \widehat{e_h}$ and $\widehat{P_{h+1}} \in \widehat{e_{h+1}}$ their special points.
There are two cases: 
\begin{itemize}
\item if $e_h, e_{h+1}$ are half-infinite edges, glue together $\widehat{e}_{h}, \widehat{e}_{h+1}$ with an isometry that makes the special points $\widehat{P_h}$ and $\widehat{P_{h+1}}$ coincide (see Figure \ref{A1}).
\item (this can happen only in the Case (3)) if $e_h, e_{h+1}$ are bi-infinite edges, the glueing procedure of $\widehat{\G_h}$ and $\widehat{\G_{h+1}}$ is the one described in Case (3) above. The edges $\widehat{e}_{h}, \widehat{e}_{h+1}$ are glued together with an isometry that keeps the special points $\widehat{P_h}$ and $\widehat{P_{h+1}}$ at a distance equal to the shear between the quadrilateral pieces  $\G_h, \G_{h+1}$.
\end{itemize}

In this way, for every crown $C_j \subset C$, we constructed a space $A_j$ that we call the \emph{auxiliary cylinder} of $C_j$. We remark that our construction relies only on the choice of an orientation on $X$. It satisfies the following properties:

\begin{lemma}   \label{lemma:A_j}
For every $C_j$, the surface $A_j$ constructed above is a complete hyperbolic surface whose interior is homeomorphic to an annulus. Moreover, there exists an isometry $f_j: C_j \hookrightarrow A_j$. 
\end{lemma}
\begin{proof}
By construction $A_j$ is an annulus and $\partial A_j = \{ c_j', c_j'' \} $ where $c_j', c_j''$ are cycles of bi-infinite geodesics, with the leaves of $c_j'$ in bijection with leaves of $c_j \in \partial C_j$. Moreover, by 
construction, there exists $\epsilon >0$ and an isometry $f_\epsilon: N_\epsilon(c_j) \hookrightarrow N_\epsilon(c_j')$. Denote by $\hat \gamma_j$ the core geodesic of the annulus and $\widehat C_j := \mathrm{ConvHull}(c_j', \hat \gamma_j) \subset A_j$. We have that $\widehat C_j$ is a complete hyperbolic surface. Moreover $f_\epsilon$ extends to an isometry $f_j : C_j \hookrightarrow \widehat C_j$. 
\end{proof}

\begin{figure}[ht!]
\begin{center}
\psfrag{f}{$f_j$}
\psfrag{A}{$A_j$}
\psfrag{R'}{\hspace{-0.2cm} $\widehat R_j~$~}
\psfrag{C'}{$\widehat {C}_j$}
\psfrag{H}{$\widehat \gamma_j$}
\psfrag{K}[l][l]{\hspace{-0.2cm} \vspace{0.2cm} \tiny $\widehat{\mathcal G}_{h+1}~$~~}
\psfrag{T}{\tiny $\widehat{\mathcal G}_h$}
\psfrag{Gamma}{$\gamma_i$}
\psfrag{G_j}{\tiny $\zeta_h^j$}
\psfrag{G_k}[l][Bl]{\hspace{-0.2cm} \tiny $\zeta_{h+1}^j~~$~~}
\psfrag{C}{$C_j$}
\psfrag{P}[B][B]{\small ~$~P_h$}
\includegraphics[width=11cm]{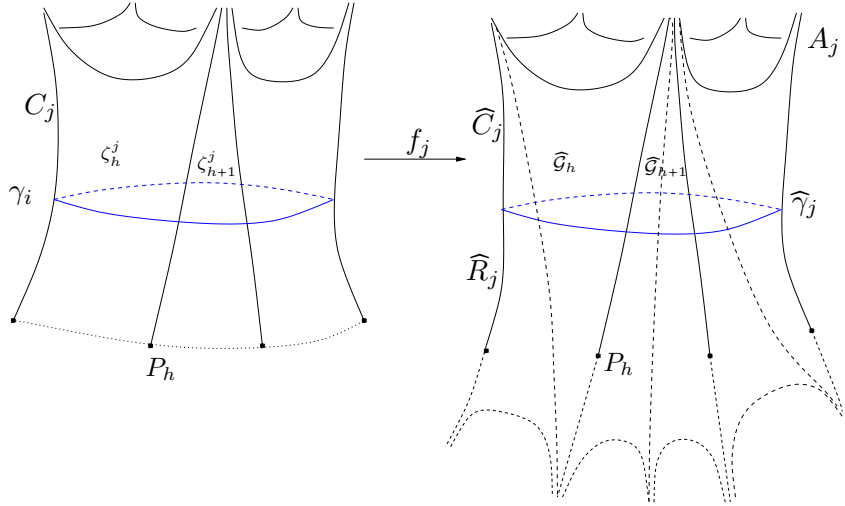}
\caption{ $A_j = \widehat C_j \cup \widehat R_j $ with $\widehat C_j \cap \widehat R_j = \widehat \gamma_j$,  $f_j : C_j \to \widehat C_j \subset A_j$  and $P_h \sim P_{h+1}$}\label{A1}
\end{center}
\end{figure}

We will denote by $\delta_{A_j}$ the (finite) maximal lamination of $A_j$ given by the union of all the edges of the ideal triangles in the triangulations of the $\widehat{\G_k}$'s. This lamination has a natural partition, which will be useful later: we will denote by $\delta_{A_j}'$ the union of all the special edges of the $\widehat{\mathcal G_k}$'s by $\delta_{A_j}''$ is the union of all the extra diagonals we added to the $\widehat{\mathcal G_k}$'s.

\subsection{The triangulated surface} \label{glueing_annulus}

We define the \emph{auxiliary multi-cylinder} $A$ as the disjoint union of all the auxiliary cylinders $A_j$. There is an isometry $f: C \hookrightarrow A$ defined as  $f(z) := f_j(z)$  if $z \in C_j$. This multi-cylinder carries a finite maximal lamination $\delta_A$ defined as the union of all the $\delta_{A_j}$. This is again partitioned in two parts: $\delta_{A}' := \bigcup_j \delta_{A_j}'$ consisting of the special edges, $\delta_{A}'' := \bigcup_j \delta_{A_j}''$ made of the extra diagonals.

\begin{notation}
The following notation will be used here and in the rest of the paper:
\begin{itemize}
\item $\widehat{C_j}:= f_j(C_j)$, $\widehat{C}:= f(C) = \bigsqcup \widehat{C_j}$;
\item $\hat{\gamma}_j := f_j(\gamma_j)$, $\widehat{\Gamma} := f(\Gamma) = \bigsqcup \widehat{\gamma_j}$;
\item $\widehat{R_j} := \overline{A_j \setminus \widehat{C_j}} = \overline{A_j \setminus f_j(C_j)}$, $\widehat{R} := \overline{A \setminus \widehat{C}} = \overline{A \setminus f(C)} = \bigsqcup \widehat{R_j}$.
\end{itemize}
\end{notation}

Thus in this notation we have: $A = \widehat{C} \cup \widehat{R}$ with $\widehat{R} \cap \widehat{C}= \widehat{\Gamma}$ and $f : C \to \widehat{C}$ is an isometry (see Figure \ref{A1}). As above, $\widehat{R}$ is a complete hyperbolic surface. 
\begin{definition}[The triangulated surface] \label{def:X_A}
We define a surface $X_A$, called the \emph{triangulated surface} as follows (see also Figure \ref{S_A}): 
\[X_A:= \faktor{X_C \bigsqcup A}{ \sim }, ~ 
\mbox{ where } z\sim f(z) \mbox{ for every } z \in C\,.\]
Let $\pi: X_C \bigsqcup A \to X_A$ be the quotient map associated. We denote by $g:= \pi_| : A \to X_A$ the restriction of $\pi$ to $A$, and by $\lambda_A$ the lamination 
\[\lambda_A := \pi(\lambda_{X_C}) \cup \pi(\delta_{A}) \subset X_A\,.\]
We denote by $\mu_A$ the closure of $\pi(\delta_A)$ in $X_A$, a sublamination of $\lambda_A$. We denote by $\nu_A$ the lamination $\mu_A \setminus \pi(\delta_A)$, a sublamination of $\lambda_A$. Notice that
\[\nu_A \subset \mu_A \subset \lambda_A\,.\] 
\end{definition}

\begin{proposition}\label{prop:S_A} 
The quotient map $\pi: X_C \bigsqcup A \to X_A$ induces on $X_A$ a structure of (possibly disconnected) complete hyperbolic surface of finite volume with non-compact boundary. The following diagrams are commutative and all arrows are 1-1 local isometries:
$$\xymatrix{ 
& X_C \ar[rd]^{\pi_|} & \\ 
C \ar[ur]^{\iota} \ar[dr]_{f} & & X_A \\ 
& A \ar[ur]_{g:=\pi_|} & 
}  \mbox{ ~~~~~~~~~~~~  }
\xymatrix{ 
& X_C \ar[rd]^{\pi_|} & \\ 
\Gamma \ar[ur]^{\iota_|} \ar[dr]_{f_|} & & X_A \\ 
& A \ar[ur]_{g:=\pi_|} & 
} $$
where $\iota: C \hookrightarrow X_C$ is the canonical inclusion; $g:=\pi_|: A \to X_A$ and $\pi_|: X_C \to X_A$ are the restrictions of $\pi$. Moreover, $\lambda_A:= \pi(\lambda_{X_C}) \cup \pi(\delta_A)$ is a maximal lamination on $X_A$, and $\pi_|: X_C \to  \overline{X_A \setminus g(\widehat{R})}$ is an isometry.
\end{proposition}
\begin{proof}
Notice that, by our constructions, $X_A$ can be equivalently defined as follows: 
$$X_A = \faktor{X_C \bigsqcup \widehat{R}}{ \sim }, ~ 
\mbox{ where } z\sim f(z) \mbox{ for every } z \in \Gamma .$$
So $X_A$ is obtained glueing together two complete hyperbolic surfaces ($X_C$ and $\widehat {R}$) along finitely many compact connected components of their boundary ($\Gamma \subset \partial X_C$ and $\widehat{\Gamma} \subset \partial \widehat{R}$) via a prescribed isometry ($f_|: \Gamma \to \widehat{\Gamma}$). Therefore, $X_A$ is also a (possibly disconnected) complete hyperbolic surface, and the two restrictions $\pi_{|A}$, $\pi_{|X_C}$ of $\pi$ are both 1-1 local isometries. The following holds by our definitions of $B$, $X_C$ and $\lambda_{X_C}$:  
$$X_C \setminus \lambda_{X_C} = \mathring{C} \cup \{ \mathring{\G} ~|~ \G \mbox{ is a triangular geometric piece in } X \setminus \lambda \}~.$$ 
By construction $A \setminus \delta_A$ is a union of ideal triangles. Since $\pi$ identifies $C$ with $\widehat{C} \subset A$, we thus have that $X_A \setminus \lambda_A$ is a union of ideal triangles as well, that is, $\lambda_A$ is maximal.
\end{proof}

\subsection{Stretching the auxiliary cylinders}    \label{stretch_A^t}    \label{Thstretch_A} 

We now want to stretch the triangulated surface $X_A$. We will start by stretching the auxiliary cylinders. In Section \ref{comfort_cyl_sub} we defined the auxiliary cylinder $A_j$ for every crown $C_j$. In Section \ref{sec:stretchB} we defined the stretched boundary block $B^t$ and we introduced the crown $C_j^t$ in $B^t$. We now want to define the stretched auxiliary cylinder $A_j^t$ for every parameter $t\geq 0$.
\begin{definition}[Stretched auxiliary cylinder]
We define the \emph{stretched auxiliary cylinder} $A_j^t$ as the auxiliary cylinder associated to the crown $C_j^t$, i.e. we apply the definition of Section  \ref{comfort_cyl_sub} to the crown $C_j^t$. By construction we also get an isometry $f_j^t:C_j^t \hookrightarrow A_j^t$ and a maximal finite lamination $\delta_{A_j}$ on $A_j^t$. By Lemma \ref{lemma:A_j} $A_j^t$ is a complete hyperbolic surface whose interior is homeomorphic to an annulus. 
\end{definition}

\begin{notation} We denote by $c_j := \bigcup_i b_{i}^j \subset \partial A_j$ the cycle of bi-infinite leaves that corresponds to the cycle with the same name in $B$. We say:
\begin{itemize} 
\item $b_i^j$ is the bi-infinite leaf in $\partial \widehat{\Q_i} \cap \partial A_j$ with $\Q_i:= \Q_{s_i} \subset B$ a quadrilateral piece. The $b_i^j$'s are enumerated so that any two consecutive $b_i^j$,$b_{i+1}^j$ form a spike $a_i^j \subset A_j$; 
\item the leaves of $\delta_{A_j}$ entering the spike $a_{i}^j$ are denoted by $e_{ij}^1, \ldots, e_{ij}^{n_i}$ (Figure \ref{notation}). Note that $\delta_{A_j} = \bigcup_{a_i^j} \{e_{ij}^1, \ldots, e_{ij}^{n_i} \}$.     
\end{itemize} 
\end{notation}

\begin{figure}[ht!]
\begin{center}
\begin{subfigure}[t]{0.4\textwidth}
\psfrag{a_i}[b]{ $~a_i^j$}
\psfrag{b_i}{$b_i^j$}
\psfrag{b_i+1}[Bl][l]{$b_{i+1}^j~$~}
\psfrag{e_1^1}[t][t]{$e_{ij}^1$}
\psfrag{e_1^i}[t][t]{$e_{ij}^k$}
\psfrag{e_1^n}[t][t]{$e_{ij}^{n_i}$}
\includegraphics[width=4cm]{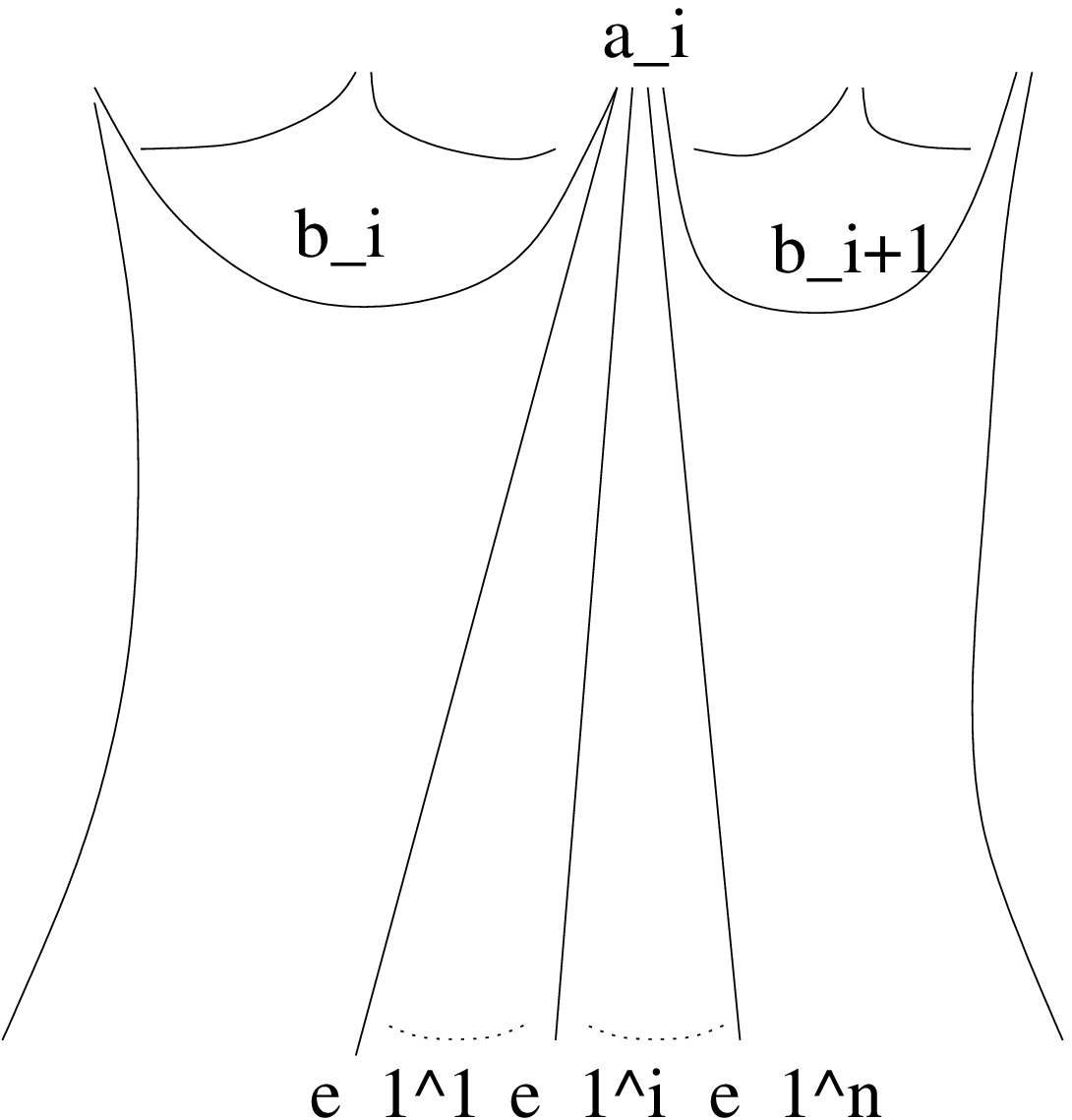}
\caption{The spike $a_i^j$ of ${A_j}$}\label{notation}
\end{subfigure}
\begin{subfigure}[t]{0.4\textwidth}
\psfrag{v_j}{$v_i^j$}
\psfrag{a_j}{$a_i^j$}
\psfrag{b_i}[c][l]{$b_i^j$}
\psfrag{b_i+1}{$b_{i+1}^j$}
\psfrag{e_i^j}{$e_{ij}^1$}
\psfrag{e_i^k}{$e_{ij}^{n_i}$}
\includegraphics[width=3cm]{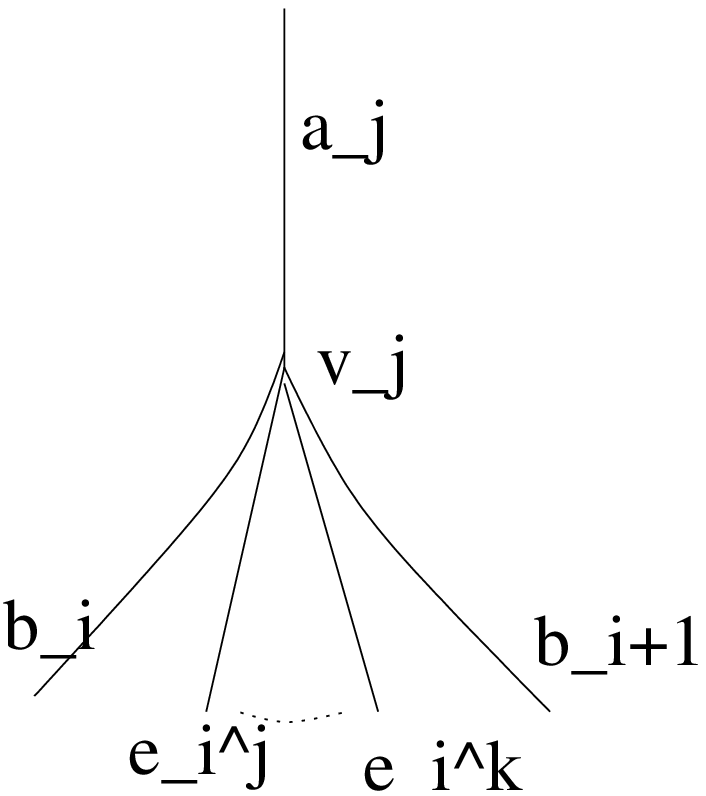}
\caption{Its corresponding subtrack $\tau_{ij} \subset \tau$}\label{traintrack}
\end{subfigure}
\end{center}
\caption{Notation on ${A_j}$}
\end{figure}

The lamination $\delta_{A_j}$ on $A_j$ is actually an ideal triangulation. We associate to each edge the shear between the two adjacent triangles (see Definition \ref{def:shear}). 
\begin{definition}[Shear coordinates for $A_j$ and $A_j^t$] 
We denote the shear coordinates of the hyperbolic structure $A_j$ by $s^0(e_{ij}^k)$, and the shear coordinates of $A_j^t$ by $s^t(e_{ij}^k)$. 
\end{definition}
If $e_{ij}^k \in \delta_{A_j}$ is not a special edge then $s^t(e_{ij}^k) = e^t \cdot s^0(e_{ij}^k)$ by construction of $A_j^t$. Otherwise, in general 
$s^t(e_{ij}^k) \neq e^t \cdot s^0(e_{ij}^k)$ (see Figure \ref{compare_fig} and Proposition \ref{cocycle_difference}).    
\subsubsection{Stretch difference formula}
The auxiliary cylinder $A_j$ is triangulated by $\delta_{A_j}$, so it can be stretched using Thurston's technique 
\cite{Thurston}. We denote by  $(A_j)_{\mathrm{Th}}^t$ the Thurston's stretch of $A_j$. The shear coordinates of  $(A_j)_{\mathrm{Th}}^t$ are $e^t \cdot s^0(e_{ij}^k)$ by construction. In Proposition \ref{cocycle_difference} we quantify the difference between the shear coordinates of $A_j^t$ and $(A_j)_{\mathrm{Th}}^t$.

\begin{figure}[htbp]
\begin{subfigure}[t]{0.4\textwidth}
\begin{center}
\psfrag{C}[b]{$O$}
\psfrag{C'}[l]{$O'$}
\includegraphics[width=4cm]{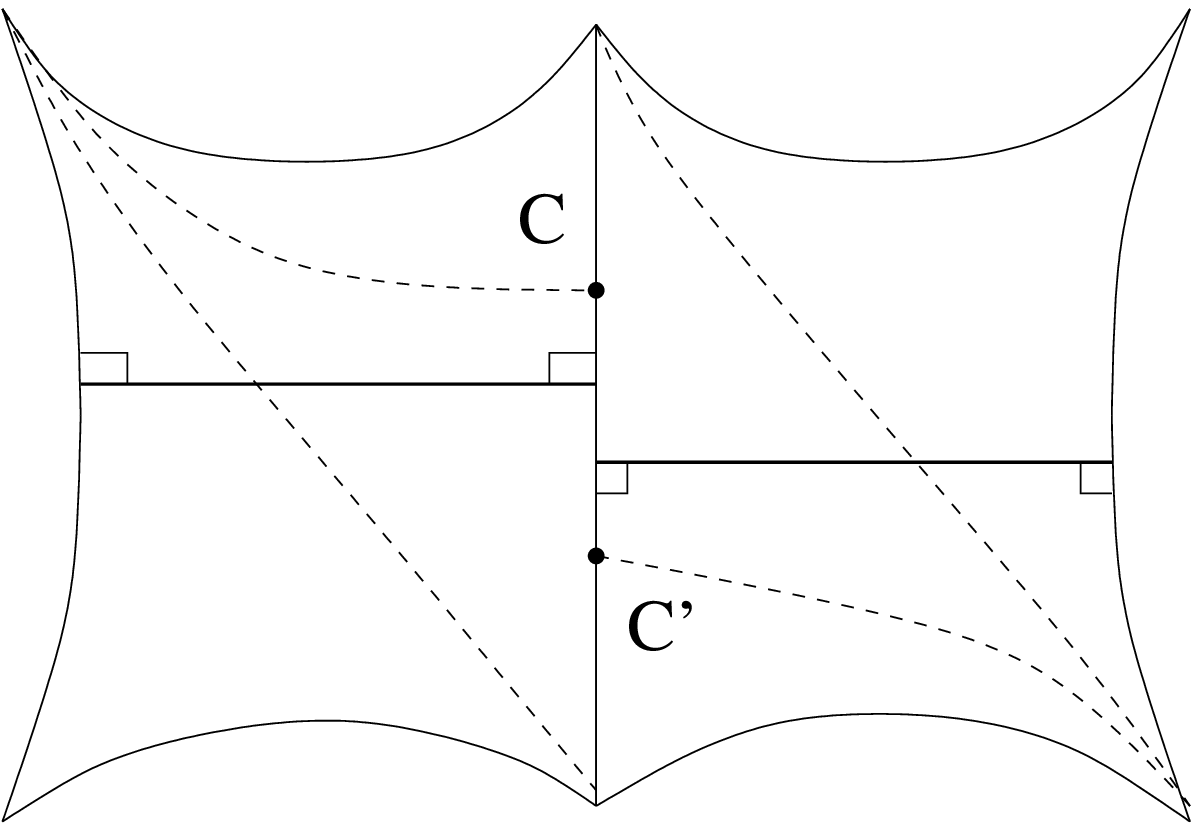}
\end{center}
\caption{$(A_j)^t_{\mathrm{Th}}$}
\end{subfigure}
\begin{subfigure}[t]{0.4\textwidth}
\begin{center}
\psfrag{C}[b]{\small $O$}
\psfrag{C'}[l]{\small $O'$}
\includegraphics[width=4cm]{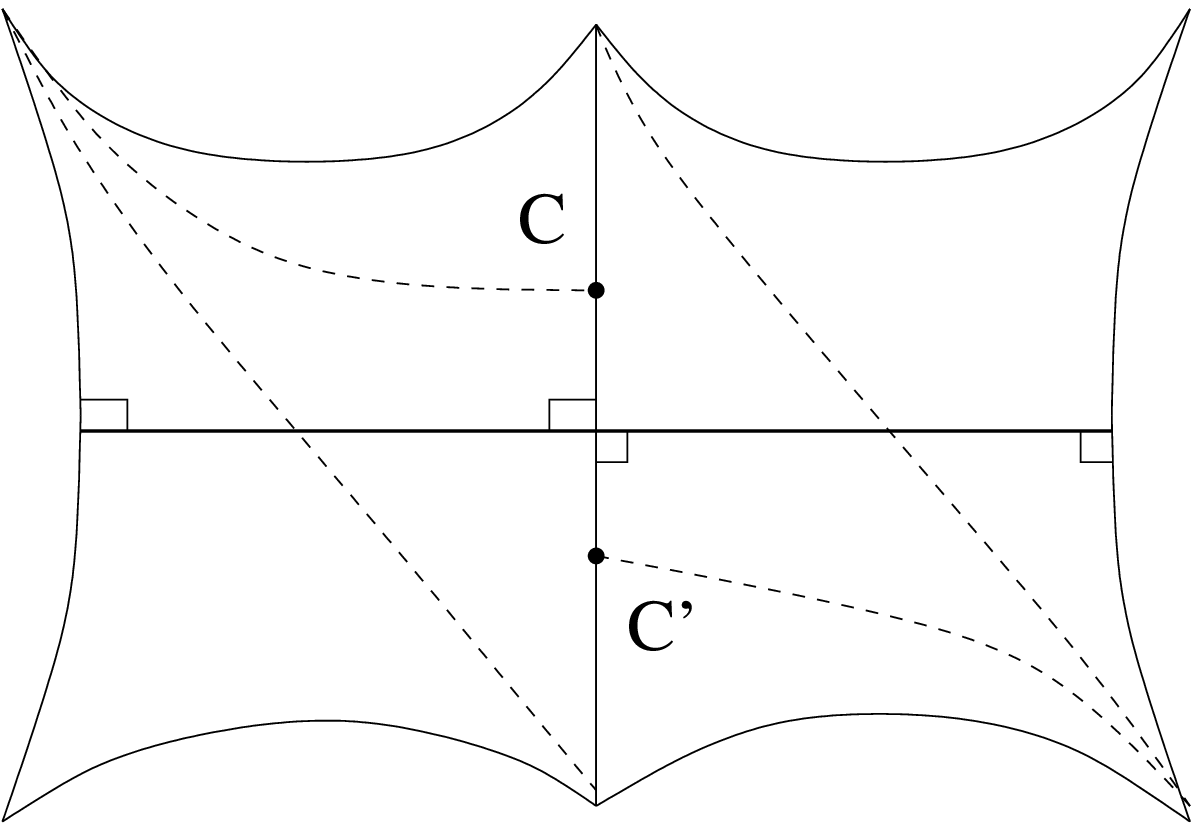}
\end{center}
\caption{$A_j^t$}
\end{subfigure}
\caption{Comparing $A_j^t$ with $(A_j)_{\mathrm{Th}}^t$}\label{compare_fig}
\end{figure}

\begin{notation} \label{not:minus notation}
In this subsection and in the proof of Lemma \ref{lemma:alpha}, if $A,B$ are points on the bi-infinite edge of a quadrilateral piece, we will denote by $A - B$ their signed distance, with the sign given by the orientation of the surface (Figure \ref{fig:delta}).   
\end{notation}

\begin{definition}[Displacement function]\label{def:displacement}
Let $\mathcal Q_{s}$ be the quadrilateral piece of shear $s \in \mathbb R$ and $\mathcal T_s \subset  \mathcal Q_{s}^d$ be the ideal triangle adjacent to the bi-infinite edge $b \subset \Q_s$. Let $O_{\mathcal Q_s}$ be the center of $\Q_s$ and $O_{\mathcal T_s}$ the center of $\mathcal T_s$ on the edge $b$ (see Figure \ref{Q}). We define the \emph{displacement function} $\delta: \mathbb R \to \mathbb R^+$ as follows: 
$$\delta(s) :=  O_{\mathcal T_s} \minus O_{\mathcal Q_s}.$$
(Note that the function $ \delta: \mathbb R \to \mathbb R^+ $ is continuous and bijective.)
\end{definition}

\begin{figure}[htbp]
\begin{center}
\begin{subfigure}[t]{0.4\textwidth}
\psfrag{Q}{$\mathcal Q_s^d$}
\psfrag{C}[b]{\tiny $O_{\mathcal Q_s}$}
\psfrag{C'}[t][l]{\tiny $O_{\mathcal T_s}$}
\psfrag{T}{$\mathcal T_s$}
\includegraphics[width=3.7cm]{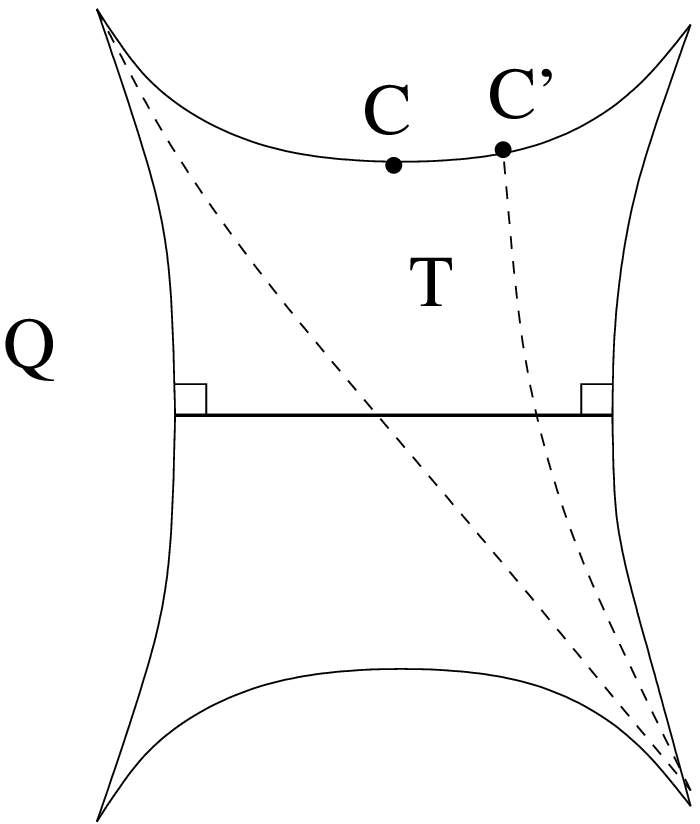}
\caption{ $\delta(s) :=  O_{\mathcal T_s} \minus O_{\mathcal Q_s}$}\label{Q}
\end{subfigure}
\begin{subfigure}[t]{0.5\textwidth}
\psfrag{Q}{$ $}
\psfrag{C}{$A$}
\psfrag{C'}{$B$}
\includegraphics[width=5cm]{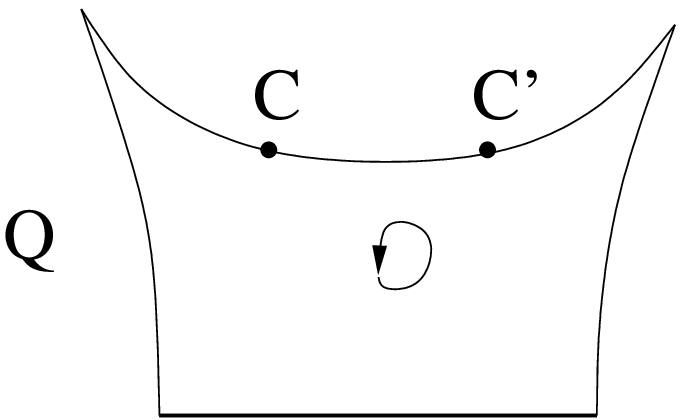}\label{Q3}
\caption{Notation: $A \minus B >0$}
\end{subfigure} 
\end{center}
\caption{The displacement function $\delta$}  \label{fig:delta}
\end{figure}

\begin{definition}[Horocyclic map]
For every spike $a_i^j \in A_j^t$ consider the \emph{horocyclic map} $$\eta^t: b_i^j \to b_{i+1}^j $$ where $\eta^t(P) \in b_{i+1}^j$ is the endpoint of the (unique) horocycle around the spike $a_i^j$ through $P \in b_i^j $ (Figure \ref{horocycle}). 
\begin{figure}[htbp]
\begin{center}
\psfrag{b_i}{$b_i^j$}
\psfrag{b_i+1}{$b_{i+1}^j$}
\psfrag{C}[bl][B]{\tiny $O_{\mathcal Q_i^t}~$~~}
\psfrag{P}{\small $P$}
\psfrag{C'}{\tiny $O_{\mathcal Q_{i+1}^t}$}
\psfrag{h(P)}{$\eta^t(P)$}
\psfrag{a_i}{$a_i^j$}
\includegraphics[width=7cm]{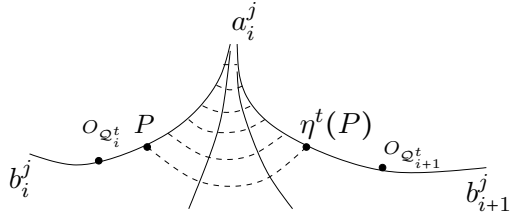}
\end{center}
\caption{The horocycle map $\eta^t: b_i^j \to b_{i+1}^j$}\label{horocycle}
\end{figure}
\end{definition}
By definition of the shear coordinates on $A_j^t$ we have:
\begin{align}\label{horocycle_eq}
\sum_{k=1}^{n_i} s^t(e_{ij}^k) = \eta^t(O_{\mathcal T_i}) \minus O_{\mathcal T_{i+1}^t} ~.
\end{align}

\begin{lemma}\label{lem:horocycle}
For every $t\geq 0$ we have: 
$ \eta^t (O_{\mathcal Q_i^t}) \minus O_{\mathcal Q_{i+1}^t} =  e^t \cdot [\eta^0(O_{\mathcal Q_i^0}) \minus O_{\mathcal Q_{i+1}^0}].$
\end{lemma}
\begin{proof}
See Figure \ref{horocycle}. Let $h^t$ be the horocycle around the spike $a_i^j$ passing through $O_{\mathcal Q_i^t}$. We denote by $e_{ij}^{k_0}, e_{ij}^{k_1}, \dots, e_{ij}^{k_p}$ the special edges that $h^t$ meets in order. For every $e_{ij}^{k_a}$, denote by ${\widehat{P_{ij}^{k_a}}}^t$ the special point on that edge (see definition in Section \ref{comfort_cyl_sub}).  

Claim (1): $d(h^t \cap e_{ij}^{k_0},{\widehat{P_{ij}^{k_0}}}^t) = e^t d( h^0 \cap e_{ij}^{k_0}, {\widehat{P_{ij}^{k_0}}}^0)\,.$ \\ 
To see this, notice that by Lemma \ref{shear_quads}, $d(h^t \cap e_{ij}^{k_0}, {\widehat{P_{ij}^{k_0}}}^t)$ is equal to the parameter of the quadrilateral $\mathcal Q_i^t$, which is $e^t$ times the parameter of the quadrilateral $\mathcal Q_i^0$, which again by Lemma \ref{shear_quads} is equal to $d(h^0 \cap e_{ij}^{k_0}, {\widehat{P_{ij}^{k_0}}}^0)$.

Claim (2):  $d( h^t \cap e_{ij}^{k_a}, {\widehat{P_{ij}^{k_a}}}^t) = e^t d(h^0 \cap e_{ij}^{k_a}, {\widehat{P_{ij}^{k_a}}}^0),$ for every special edge $e_{ij}^{k_a}$. \\ 
By induction on $a$, this is true for the previous edge $e_{ij}^{k_{a-1}}$. For the edge $e_{ij}^{k_a}$ it then follows from Lemma \ref{shear_penta} or Lemma \ref{shear_quads}. 

Claim (3): $d(h^t \cap b^j_{i+1}, O_{\mathcal Q_{i+1}^t}) = e^t d(h^0 \cap b^j_{i+1}, O_{\mathcal Q_{i+1}^0})\,,$ 
which is the statement. \\ We use Step (2) applied to the case $a=p$ (the last special edge), then we conclude as in Step (1), applying Lemma \ref{shear_quads} to the quadrilaterals $\mathcal Q_{i+1}^t$ and $\mathcal Q_{i+1}^0$.   
\end{proof}

We are now ready to prove the stretch difference formula.

\begin{proposition}[Stretch difference formula]\label{cocycle_difference}
The following holds: 
 $$ \sum_{k=1}^{n_i} s^t(e_{ij}^k) - \sum_{k=1}^{n_i} e^t \cdot s^0(e_{ij}^k) = - \delta(e^t \cdot s_i) + e^t \cdot \delta(s_i) - \delta(e^t \cdot s_{i+1}) + e^t \cdot \delta(s_{i+1})\,.$$
\end{proposition}
\begin{proof} 
First compute $\sum_{k=1}^{n_i} s^t(e_{ij}^k)$ using (\ref{horocycle_eq}): 
\begin{equation}
\begin{aligned}
 \sum_{k=1}^{n_i} s^t(e_{ij}^k)  & =  \eta^t(O_{\mathcal T_i}) \minus O_{\mathcal T_{i+1}^t} =\\
                             & =  [\eta^t(O_{\mathcal T_i^t}) \minus \eta^t(O_{\mathcal Q_i^t}) ]  + [\eta^t(O_{\mathcal Q_i^t}) \minus O_{\mathcal Q_{i+1}^t}] + [O_{\mathcal Q_{i+1}^t} \minus O_{\mathcal T_{i+1}^t}]\,.  \label{blablabla}
\end{aligned}
\end{equation}

Let us now compute separately each summand in the second member:
\begin{align}
& \eta^t(O_{\mathcal T_i^t}) \minus \eta^t(O_{\mathcal Q_i^t}) = O_{\mathcal Q_{i}^t} \minus O_{\mathcal T_{i}^t}  =  - \delta(e^t \cdot s_i);  \tag*{by Definition \ref{def:displacement},} \\ 
& \eta^t(O_{\mathcal Q_i^t}) \minus O_{\mathcal Q_{i+1}^t} =  e^t \cdot [\eta^0(O_{\mathcal Q_i^0}) \minus O_{\mathcal Q_{i+1}^0}];  \tag*{by Lemma \ref{lem:horocycle},} \\ 
& O_{\mathcal Q_{i+1}^t} \minus O_{\mathcal T_{i+1}^t} = - \delta(e^t \cdot s_{i+1}). \tag*{by Definition \ref{def:displacement},} 
\end{align}
Substituting the equations above in (\ref{blablabla}),  
we find: 
\begin{align}
\sum_{k=1}^{n_i} s^t(e_{ij}^k) = - \delta(e^t \cdot s_i) - \delta(e^t \cdot s_{i+1}) +  e^t \cdot [\eta^0(O_{\mathcal Q_i^0}) \minus O_{\mathcal Q_{i+1}^0}] ~. \label{comparison}
\end{align}
Now we compute $\sum_{k=1}^{n_i} e^t \cdot s^0(e_{ij}^k)$ using (\ref{comparison}) evaluated in $t=0$:  
\begin{align}\label{10}
\sum_{k=1}^{n_i}e^t \cdot s^0(e_{ij}^k) &=  e^t \cdot \sum_{k=1}^{n_i} s^0(e_{ij}^k) = - e^t\cdot \delta(s_i)  - e^t \cdot \delta(s_{i+1}) + e^t \cdot [\eta^0(O_{\mathcal Q_i}) \minus O_{\mathcal Q_{i+1}^0}]~.
\end{align}
Combining (\ref{comparison}) and (\ref{10})
 , we get the following and we are done: 
\begin{equation*}
\sum_{k=1}^{n_i} s^t(e_{ij}^k) - \sum_{k=1}^{n_i} e^t \cdot s^0(e_{ij}^k) = - \delta(e^t \cdot s_i) + e^t \cdot \delta(s_i) - \delta(e^t \cdot s_{i+1}) + e^t \cdot \delta(s_{i+1})~. \qedhere
\end{equation*}
\end{proof}

\subsection{Stretching the triangulated surface}  \label{sec:S_A}

In this subsection we will stretch the triangulated surface $X_A$ using Bonahon's theory of cataclysms (see Section \ref{Bonahon}). For every $t \geq 0$ we will define a complete hyperbolic surface $((X_A)^t, \lambda_A, g^t)$ with a 1-1 local isometry $g^t: A^t \hookrightarrow (X_A)^t$, where $A^t$ is the hyperbolic surface obtained stretching the auxiliary multi-cylinder along $\delta_A$. 

\begin{definition}[Stretched auxiliary multi-cylinder]
We define the \emph{stretched auxiliary multi-cylinder} $A^t$ and its maximal lamination $\delta_A$ by 
$$A^t:= \bigsqcup_j A_j^t~ \mbox{and } \delta_A := \bigsqcup_j \delta_{A_j} $$ 
It comes with an isometry $f^t:= \bigsqcup_j f_j^t : C_j^t \to A_j^t$. 
\end{definition}
\begin{notation} We will use the following notation: 
\begin{itemize}
\item $\widehat{C_j^t}:= f^t(C_j^t)$; 
\item $\widehat{C^t}:= f^t(C^t)$; 
\item $\widehat{\Gamma^t} := f^t(\Gamma^t)$; 
\item $\widehat{R^t} := \overline{A^t \setminus \widehat{C^t}} = \overline{A^t \setminus f^t(C^t)}$.
\end{itemize}
\end{notation}
 
\subsubsection{Construction of $((X_A)^t, \lambda_A, g^t)$} By Proposition \ref{prop:S_A} we have that $(X_A, \lambda_A)$ is a complete hyperbolic surface of finite volume with non-compact boundary and $\lambda_A$ is a maximal lamination whose complement contains only triangles. We will first consider its double, i.e. the surface $(X_A^d, \lambda_A^d)$ defined as follows.
\begin{definition}[Double of $X_A$] We define:  
$$X_A^d = X_A \sqcup X_A'/\sim\,, $$ 
where $X_A'$ is an isometric copy of $X_A$ with the opposite orientation, and $\sim$ identifies the boundary of $X_A$ and $X_A'$ with the identity map. The lamination $\lambda_A^d$ is defined as 
$$\lambda_A^d = \lambda_A \cup \lambda_A',$$
where $\lambda_A'$ is the copy of $\lambda_A$ on $S_A'$. If $A_1, \dots, A_m$ are the auxiliary cylinders in $X_A$, then $A_{m+1}, \dots, A_{2m}$ are the auxiliary cylinders in $X_A'$ ($A_{i+m}$ is the mirror copy of $A_i$). 
\end{definition}
The following fact is immediate.
\begin{proposition}
The surface $X_A^d$ is a finite hyperbolic surface without boundary and $\lambda_A^d$ is a maximal lamination on $X_A^d$.
\end{proposition}

For every $t\geq 0$ we will now define a new hyperbolic structure $(X_A^d)^t$ by defining a suitable cocycle for $\lambda_A^d$ (see Bonahon's Theorem \ref{Bonahon3}). Let $\rho^0$ be the shearing cocycle for the lamination $\lambda_A^d$ associated to the hyperbolic structure $X_A^d$. Note that for every $t \geq 0$ the cocycle $e^t \cdot \rho^0$ is the cocycle of the hyperbolic structure  $(X_A^d)_{Th}^t$ obtained via the Thurston stretch of $X_A^d$. We will define our hyperbolic structure $(X_A^d)^t$ on $X_A^d$ by adding a term to the cocycle $e^t \cdot \rho^0$.    
\begin{definition}[Cocycle $\rho^t$ on $\lambda_A^d$]
Choose a train track $\tau$ snugly carrying $\lambda_A^d$ such that $\tau$ contains one subtrack $\tau_{ij}$ as in Figure \ref{traintrack} for every spike $a_i^j$ in one of the $A_j$, for $j \in \{1,\dots,2m\}$. (Here we label every edge of $\tau_{ij}$ by the unique edge of $\delta_{A_j} \cup \partial A_j$  it carries, and the switch $v_i^j$ corresponds to the spike $a_i^j$.)  

We define an assignment of real weights $\epsilon^t$ on the edges of $\tau$. Define $\epsilon^t(e) := 0$ for every $e \in \tau$ such that $e \not \in \bigcup_{ij} \tau_{ij}$. For $e\in \tau_{ij}$ the assignment $\epsilon^t(e)$ is the following:
\begin{align} 
\epsilon^t(a_i^j) &: = 0, \\ 
\epsilon^t(e_{ij}^k) & :=  - e^t s^0(e_{ij}^k) + s^t(e_{ij}^k) \mbox { for } k=1, \ldots , n_i, \\ 
\epsilon^t(b_i^j) & := \delta(e^t \cdot s_i) + e^t \cdot \delta(s_i), \\ 
\epsilon^t(b_{i+1}^j) & := \delta(e^t \cdot s_{i+1}) + e^t \cdot \delta(s_{i+1}), 
\end{align}
where the functions $s^t$ and $\delta$ were defined in Section \ref{stretch_A^t}. Now we define $(\rho^t(e))_{e\in \tau}$ as 
\[\rho^t(e) := e^t \cdot \rho^0(e) + \epsilon^t(e)\,.\]
 \end{definition}
 
In the next subsection we will prove the following:
\begin{proposition}  \label{shearing_nu} 
For every $t \geq 0$, the assignment of real weights $\rho^t$ on the edges of $\tau$ defines the shearing cocycle for the lamination $\lambda_A^d$ of a hyperbolic structure on $X_A^d$, which we will denote by $(X_A^d)^t$. 
\end{proposition}

The lamination $\lambda_A^d$ and the cocycle $\rho^t$ are both symmetric for the involution of $X_A^d$, hence the hyperbolic structure $(X_A^d)^t$ is also symmetric. 
\begin{definition}[Triangulated stretched surface]
The \emph{triangulated stretched surface} $(X_A)^t$ is the restriction of the hyperbolic structure $(X_A^d)^t$ to the surface $X_A$. (Note that $X_A^t$ is a complete hyperbolic surface of finite volume.)
\end{definition}

\begin{proposition}
There is a 1-1 local isometry $g^t: A^t \hookrightarrow (X_A)^t$. When $t=0$ we have $((X_A)^0, \lambda_A, g^0) = (X_A, \lambda_A, g)$.
\end{proposition}
\begin{proof}
By definition, for every $e_{ij}^k \in \delta_{A}$ we have: 
$$\rho^t(e_{ij}^k) = e^t \cdot \rho^0(e_{ij}^k) + \epsilon^t(e_{ij}^k) = e^t \cdot s^0(e_{ij}^k) + (- e^t s^0(e_{ij}^k) + s^t(e_{ij}^k)) = s^t(e_{ij}^k)\,,$$ 
which are the shear coordinates of $A^t$ for the ideal triangulation $\delta_A$.
\end{proof}

We also want to construct a stretch map for the stretched triangulated surface $(X_A)^t$. This will be given by the composition of Thurston's stretch map and a shear map. Consider the two hyperbolic surfaces $(X_A^d)_{Th}^t$ and $(X_A^d)^t$. We will denote by $\kappa^t$ the \emph{shear map}  between them with respect to the lamination $\lambda_A^d$:
$$\kappa^t:(X_A^d)_{Th}^t \setminus \lambda_A^d \rightarrow  (X_A^d)^t \setminus \lambda_A^d.$$
Every triangle in the complement of $\lambda_A^d$ in $(X_A^d)_{Th}^t$ is mapped isometrically to the corresponding triangle in the complement of $\lambda_A^d$ in $(X_A^d)^t$ (see Bonahon \cite[Sec. 4]{Bonahon_pleated}). Thurston \cite{Thurston} called  this map a \emph{cataclysm}.

Consider the sublamination $\mu_A \subset \lambda_A$ defined in Definition \ref{def:X_A}. This is the closure of the image of the lamination $\delta_A$. Its double, $\mu_A^d$ is a lamination on $X_A^d$.  

\begin{lemma}  \label{lemma:cataclysm}
The map $\kappa^t$ extends continuously to an isometry 
$$\bar{\kappa}^t:(X_A^d)_{Th}^t \setminus \mu_A^d \rightarrow  (X_A^d)^t \setminus \mu_A^d~.$$
\end{lemma} 
\begin{proof}
Notice that $(X_A^d)_{Th}^t \setminus \mu_A^d$ is an open subset of $(X_A^d)_{Th}^t$. Given a point $x \in (X_A^d)_{Th}^t \setminus \mu_A^d$, we can choose a ball $B$ centered at $x$ and contained in $(X_A^d)_{Th}^t \setminus \mu_A^d$. Notice that the cocycles associated to the two hyperbolic structures $(X_A^d)_{Th}^t$ and $(X_A^d)^t$ differ by the cocycle $\epsilon^t$, which is supported in $\mu_A$. Hence, when we restrict our attention to the ball $B$, the two cocycles agree. To prove the proposition in the ball $B$, we can then proceed as in the proof of \cite[Lemma 11]{Bonahon_pleated}. The proof given there uses the horocyclic foliation of a triangulated surface, see also Section \ref{subsub:horocyclic} where we will generalize that notion for surfaces with boundary.    
\end{proof}

We are now ready to construct the stretch map for $(X_A)^t\setminus \mu_A$. 

\begin{proposition}[Existence of a stretch map for $X_A \setminus \mu_A$]\label{prop:stretch_map_X_A}
For every $t\geq0$, there exists a continuous map $\psi^t: X_A \setminus \mu_A \rightarrow (X_A)^t \setminus \mu_A$ homotopic to the identity with the following properties: 
\begin{enumerate}
\item $\psi^t$ stretches the arc-length of the leaves of $\lambda_A \setminus \mu_A$ by $e^t$; 
\item on every triangular geometric piece $\mathcal T$ in $X_A \setminus \lambda_A$, the map $\psi^t$ restricts to $\psi^t|_{\mathcal{T}} = \phi^t: \mathcal{T} \to \mathcal{T}^t$ as in Lemma \ref{lemma:triangle}; 
\item $\psi^t$ is locally Lipschitz with local Lipschitz constant equal to $e^t$.  
\end{enumerate}
\end{proposition}
\begin{proof}
We will denote Thurston's stretch map by 
$$\tau^t : X_A^d \rightarrow (X_A^d)_{Th}^t.  $$
This map was introduced in \cite{Thurston}. On every triangle in $X_A^d \setminus \lambda_A^d$, this map agrees with the map $\phi^t$ from Lemma \ref{lemma:triangle}. The map $\tau^t$ is continuous, stretches the arc-length of the leaves of $\lambda_A$ by $e^t$ and $\Lip(\tau^t) = e^t$. We define the map $\psi^t$ as follows:
$$\psi^t = \bar{\kappa}^t \circ \tau^t : X_A \setminus \mu_A \rightarrow (X_A)^t \setminus \mu_A.$$
It satisfies the stated properties because of the properties of $\tau^t$ and Lemma \ref{lemma:cataclysm}. 
\end{proof}

\subsubsection{Proof of Proposition \ref{shearing_nu}}

\begin{lemma}
The assignments of real weights $(\epsilon^t(e))_{e\in \tau}$ on the edges of $\tau$ defines a transverse cocycle for the lamination $\lambda_A^d$.
\end{lemma}
\begin{proof}
By Theorem \ref{Bonahon3} we need to check that the switch relations hold at every switch $v \in \tau$. 
First assume $v= v_i^j \in \tau_{ij}$ as in Figure \ref{traintrack}. We will check the switch relation: 
\begin{align}\label{15}
\epsilon^t(a_i^j) = \epsilon^t(b_i^j) + \epsilon^t(b_{i+1}^j) + \sum_{k=1}^{n_i} \epsilon^t(e_{ij}^k)~.
\end{align}
This equation is satisfied because it is equivalent to Lemma \ref{cocycle_difference}: 
\begin{align}\label{20}
 \epsilon^t(b_{i}^j) +\epsilon^t(b_{i+1}^j) + \sum_{k=1}^{n_i} \epsilon^t(e_{ij}^k)  = 0 ~.
\end{align}
If $v$ is a switch of $\tau$ but not a switch of $\tau_{ij}$ then $\epsilon^t(e) =0$ for every edge $e$ of $\tau$ concurring in the switch. Therefore, the switch condition at $v$ is satisfied, and we conclude.
\end{proof}

\begin{lemma}\label{epsilon}
The assignments of real weights $(\rho^t(e))_{e\in \tau}$ on the edges of $\tau$ defines a transverse cocycle for the lamination $\lambda_A^d$.
\end{lemma}
\begin{proof}
By Lemma \ref{epsilon}, we have $\epsilon^t \in H(\lambda_A^d, \mathbb R)$, so $\rho^t$ is a linear combination of transverse cocycles for $\lambda_A^d$.  As $H(\lambda_A^d, \mathbb R)$ is a vector space by Theorem \ref{Bonahon3}, we have $\rho^t \in H(\lambda_A^d, \mathbb R)$.
\end{proof}

\begin{lemma}\label{lemma_boh}
For every measure $\mu$ on $\lambda_A^d$, we have $\omega(\epsilon^t, \mu) =0$
\end{lemma}
\begin{proof}
We compute $\omega(\epsilon^t, \mu)$ using Lemma \ref{lemma:sympl form}.  
After a finite sequence of splittings, $\tau$ can be made generic. In particular, splitting each subtrack $\tau_{ij}$ we get to a generic subtrack $\tau_{ij}'$ as in Figure \ref{traintracksplit}. Note that if $v$ is a switch and $v \not \in \tau_{ij}'$ then $\epsilon^t(e_v^r) = \epsilon^t(e_v^l) = 0$ because $\epsilon^t(e) = 0$ for every edge $e \not \in \bigcup \tau_{ij}$.
Therefore,  we just need to look at all the switches $v \in \tau_{ij}'$, and we have: 
$$ \omega(\epsilon^t, \mu) = \sum_{j} \sum_{i} \sum_{w \in \tau_{ij}' } [\epsilon^t(e_w^r) \mu(e_w^l) - \epsilon^t(e_w^l) \mu(e^r_w)] . $$

\begin{figure}[htbp]
\begin{center}
\psfrag{v_j}{\tiny $w_0$}
\psfrag{v}{\tiny $w_{n_i}$}
\psfrag{w}{\tiny $w_1$}
\psfrag{a_j}{$a_i^j$}
\psfrag{b_i}{$b_i^j$}
\psfrag{f1}{$f_1$}
\psfrag{f2}{$f_2$}
\psfrag{b_i+1}{$b_{i+1}^j$}
\psfrag{e_i^j}[c]{$e_{ij}^1$}
\psfrag{e_i^k}[Tr][Br]{~  $~e_{ij~~}^{n_i~} ~$}
\includegraphics[width=3.6cm]{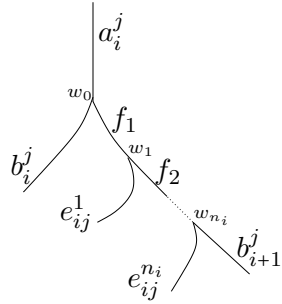}
\end{center}
\caption{Splitting $\tau_{ij} $ to make it generic}\label{traintracksplit}
\end{figure}

Note that for every measure $\mu$ on $\lambda_A^d$, we must have $\mu(e_{ij}^k) = 0$ for all $k=1, \ldots, n_i$. Indeed, all of the $e_{ij}^k$'s enter the same spike $a_i^j \in A_j \subset X_A$, and either all the $ e_{ij}^k$'s enters the same cusp or they accumulate on the same sublamination of $\lambda_A^d$. In either case for each of them $\mu(e_{ij}^k) = 0$. It follows that at every switch $w \neq w_0 \in \tau_{ij}'$ we have: 
\begin{align}\label{17}
\epsilon^t(e_w^r) \mu(e_w^l) &= 0. 
\end{align} 

Moreover, by the switch relations for $\epsilon^t$, we have: 
\begin{align}
\epsilon^t(f_1) &= \epsilon^t(a_i^j) - \epsilon^t(b_i^j) =  - \epsilon^t(b_i^j)  \label{18}\\ 
\epsilon(b_i^j) + \epsilon^t(e_{ij}^1) + \ldots + \epsilon^t(e_{ij}^{n_i}) & =  \epsilon^t(a_i^j) - \epsilon^t(b_{i+1}^j) = - \epsilon^t(b_{i+1}^j) \label{19}
\end{align}
Using (\ref{17}), (\ref{18}) and (\ref{19})  
in the computation of $\omega(\epsilon^t, \mu)$, we have: 
\begin{align*}
\omega(\epsilon^t, \mu) & = \sum_{j} \sum_{i}  \sum_{w \in \tau_{ij}' } [\mu(e_w^l) \epsilon^t(e_w^r) - \epsilon^t(e_w^l) \mu(e^r_w)] \\ 
& =  \sum_{j} \sum_{i} \left[ \mu(b_i^j) \epsilon^t(f_1) - \mu(b_{i+1}^j) [\epsilon^t(b_i^j) + \epsilon^t(e_{ij}^1) + \ldots + \epsilon^t(e_{ij}^{n_i})] \right] \tag*{by (\ref{17}) } \\ 
& =  \sum_{j} \sum_{i} \left[ - \mu(b_i^j) \epsilon^t(b_i^j) + \mu(b_{i+1}^j) \epsilon^t(b_{i+1}^j) \right]  \tag*{by (\ref{18}) and (\ref{19})}\\ 
&=\sum_j 0 = 0 ~. 
\end{align*}
Indeed, for every $j$ we have $$\sum_{i} \left[ - \mu(b_i^j) \epsilon^t(b_i^j) + \mu(b_{i+1}^j) \epsilon^t(b_{i+1}^j) \right]  = 0 ~$$ because the $b_i^j$'s form the cycle $c_j = \bigcup_{i} b_{i}^j \subset \partial A_j^\Sigma$. 
\end{proof}

\begin{proof}[Proof of Proposition \ref{shearing_nu}]
Note that $\rho^0$ is the shearing cocycle associated with a hyperbolic structure, hence by Theorem \ref{Bonahon4} it satisfies 
\[\omega(\rho^0, \alpha) > 0\] 
for every transverse measure $\alpha$ on $\lambda_A^d$. By the bi-linearity of $\omega$ and Lemma \ref{lemma_boh}, for every transverse measure $\alpha$ on $\lambda_A^d$ we have: 
$$\omega(e^t \cdot \rho^0 + \epsilon^t, \alpha) =  e^t \cdot \omega(\rho^0, \alpha) + \omega(\epsilon^t, \alpha) =  e^t \cdot \omega(\rho^0, \alpha) > 0 .$$
The statement then follows from Theorem \ref{Bonahon4}.
\end{proof}

%% file: 09-Generalized_Stretch_Lines.tex
\section{Generalized Stretch Lines}     \label{sec:gsl}

In this section we finally prove Theorem \ref{theorem:S}.

\subsection{Generalized stretch lines}   \label{subsec:gsl}

Let $X\in \T(S)$ and $\lambda$ a maximal lamination on $X$. In this subsection we will define the generalized stretch line starting from $X$ and directed by $\lambda$: for every $t \geq 0$ we will define an element $X_\lambda^t \in \T(S)$. 
 
In Section \ref{sec:S_A}, we defined the stretched triangulated surface $(X_A)^t$, with a 1-1 local isometry of the stretched auxiliary multicylinder:
$$g_t:A^t \rightarrow (X_A)^t.$$
Recall that 
$$A^t = \widehat{C^t} \cup \widehat{R^t},$$
where $\widehat{C^t}$ is the union of $m$ cylinders $\widehat{C_j^t}$ which are isometric to cylinders in the stretched boundary block $B^t$ (See Figure \ref{S_A}.)   

We define:  
$$(X_{C})^t := \overline{({X_A})^t \setminus g^t(\widehat{R^t})} \subset ({X_A})^t .$$ 
This is a hyperbolic structure on a surface with boundary homeomorphic to $X_C$. It contains a copy of $C^t$, we will denote the 1-1 local isometry by $h^t := g^t \circ f^t: C^t \hookrightarrow (X_C)^t$.

We can now define the hyperbolic structure $X_\lambda^t$, for every $t\geq 0$:
$$X_\lambda^t :=  B^t \sqcup (X_{C})^t / \sim,$$
where $\sim$ identifies a point $z \in C^t$ with the point $h^t(z) \in (X_C)^t$.

\begin{proposition}\label{S^t} 
Let $\pi: B^t \cup (X_C)^t  \to X_\lambda^t$ be the projection map. The following diagram is commutative with all arrows 1-1 local isometries, therefore $X_\lambda^t$ is a hyperbolic structure on $S$. 
$$\xymatrix{
   & B^t \ar[dr]^{\pi_|} & \\ 
C^t \ar[ur]^{\iota} \ar[dr]_{h^t} & & X_\lambda^t  \\ 
& (X_C)^t \ar[ur]_{\pi_|} &  
}$$
\end{proposition}

Notice that the lamination $\lambda$ on $X_\lambda^t$ is the union
$$\lambda = \pi(\lambda_B) \cup \pi(\lambda_{X_C}) $$

Consider the set $\pi(\partial^{nc} B^t)$, the image of the union of the non-compact boundary components of $B^t$. This set is a union of finitely many geodesics, but it is in general not closed. Its closure is a sublamination of $\lambda$ which we will call $\mu_X$:
$$\mu_X = \overline{\pi(\partial^{nc} B^t)}.$$
We also denote by $\nu_X$ the lamination 
$$\nu_X = \mu_X \setminus \pi(\partial^{nc} B^t).$$
Notice that
$$\nu_X \subset \mu_X \subset \lambda.$$
These sublaminations are closely related to  
$\mu_A, \nu_A$ from Definition  \ref{def:X_A}. 

\begin{definition}
For every $X \in \T(S)$ and $\lambda$ maximal lamination, the line
$$ s_{X,\lambda}: \R_{\geq 0} \ni t \rightarrow X_\lambda^t \in \T(S) $$
is a \emph{generalized stretch line} starting from $X$ and directed by $\lambda$. 
\end{definition}

\subsection{Generalized stretch maps}  \label{subsec:stretch maps}
We will now see that the generalized stretch lines verify Theorem \ref{theorem:S}. We will need to define the \emph{generalized stretch map} $\Phi^t: X \rightarrow X_\lambda^t$. 

\subsubsection{The generalized stretch map on an open dense subset}

We will first define a map in an open dense subset, and later we will extend it everywhere. Consider the map 
$$\alpha^t: X \setminus \nu_X \rightarrow X_\lambda^t \setminus \nu_X $$
defined as follows:
\begin{align} 
\alpha^t(z):= 
\begin{cases}
\beta^t(z) &\text{ if } z \in B, \\ 
\psi^t(z) &\text{ if } z \not \in B,
\end{cases}
\end{align}
where $\beta^t:B \rightarrow B^t$ is the map constructed in Proposition \ref{prop:B}, and $\psi^t: X_A \setminus \mu_A \rightarrow (X_A)^t \setminus \mu_A$ is the map given by Proposition \ref{prop:stretch_map_X_A}. 

\begin{lemma}    \label{lemma:alpha}
The map $\alpha^t$ is well defined on $X \setminus \nu_X$ and continuous. 
\end{lemma}
\begin{proof}
Notice that $\psi^t$ is defined in $X_A \setminus \mu_A$. Every point of $X$ coming from $\mu_A$ is either in $B$ or in $\nu_X$, hence $\alpha^t$ is well defined on $X \setminus \nu_X$. We need to check the continuity of the map at the points of $\partial^{nc} B$. Every connected component $b$ of $\partial^{nc} B$ is the bi-infinite edge of a quadrilateral piece $\Q$, with shear parameter $s$. The geodesic $b$ is part of the lamination $\lambda_A$ of $X_A$, hence the map $\psi^t$ is not defined on $b$. We can extend it to $b$ in two ways: we denote by $\underline{\psi}^t$ the extension from $B$, and by $\overline{\psi^t}$ the extension from $X \setminus B$. Both extensions are mapping $b$ to $b^t$, the copy of $b$ in $X_\lambda^t$: 
$$\underline{\psi}^t,\overline{\psi^t}: b \rightarrow b^t $$ 
\begin{align*} 
\underline{\psi}^t(z) & := \lim_{m \to + \infty} \psi^t(w_m), \ \ \ \ \mbox{ where } \{w_m \}_{m \in \mathbb N} \subset \mathcal Q \mbox{ such that } w_m \to z  \mbox{ as } m \to + \infty \\ 
\overline{\psi^t}(z) & := \lim_{m \to + \infty} \psi^t(w_m), \ \ \ \ \mbox{ where } \{w_m \}_{m \in \mathbb N} \subset X \setminus B \mbox{ such that } w_m \to z  \mbox{ as } m \to + \infty 
\end{align*}
By the definition of $\psi^t$ using shear maps and by the definition of the cocycle $\epsilon^t$, we have: 
\begin{align}
\underline{\psi^t}(z) \minus \overline{\psi^t}(z) &= \delta(e^t \cdot s) - e^t \delta(s) \,, \label{cataclysm_shear}
\end{align}
using Notation \ref{not:minus notation}.
Now we claim that, for every $z \in b$, we have: 
\begin{equation} \label{beta_and_psi}
\beta^t(z) \minus \underline{\psi^t}(z) = - \delta (e^t \cdot s) + e^t \delta(s) \,.
\end{equation} 

In order to see this, notice that 
\begin{enumerate}
\item the map $\underline{\psi}^t: b \to b^t $ fixes $O_{\mathcal T^0}$ and stretches the arc-length of $b$ by a factor $e^t$:
\begin{align*}
\forall P \in b_i ~~:~~ [\underline{\psi}^t(P) \minus O_{\mathcal T^0}] = e^t \cdot (P \minus O_{\mathcal T^0}) ~~\mbox{ with } \underline{\psi}^t(O_{\mathcal{T}^0}) = O_{\mathcal{T}^t} 
\end{align*}
\item the map $\beta^t: b \to b^t$ fixes $O_{\mathcal Q}$ and stretches the arc-length of $b$ by a factor $e^t$: 
\begin{align*}
\forall P \in b ~~:~~ [\beta^t(P) \minus O_{\mathcal Q^t}] = e^t \cdot (P \minus O_{\mathcal Q^0}) ~~\mbox{ with } \beta^t(O_{\mathcal{Q}^0}) = O_{\mathcal{Q}^t} 
\end{align*}
Putting these formulas together we have: 
\begin{align*}
[\beta^t(P) \minus \underline{\psi^t}(P)] &= [O_{{\mathcal Q}^t} \minus O_{{\mathcal T}^t}] + e^t \cdot [O_{{\mathcal Q}^0} \minus O_{{\mathcal T}^0}] =  -\delta(e^t \cdot s ) + e^t \cdot \delta(s)\,,  
\end{align*}
\end{enumerate}
which proves the claim. Now (\ref{cataclysm_shear}) and (\ref{beta_and_psi})  
imply that $\beta^t(z) \minus \overline{\psi^t}(z) = 0 $, hence $\beta^t(z) = \overline{\psi^t}(z)$. This shows the continuity of the map $\alpha^t$ on $X \setminus \nu_X$.
\end{proof}

\subsubsection{The horocyclic foliation on $X$} \label{subsub:horocyclic}
We will now define a partial foliation on $X$, called the horocyclic foliation and denoted by $\mathcal{K}$. Denote by $\{\G_i\}$ the finite set of geometric pieces of $X \setminus \lambda$. For every $\G_i$, we defined a horocyclic foliation $\mathcal{K}_i$ in Definition \ref{def:horfol3}, \ref{def:horfol4} and \ref{def:horfol5}. If $\G_i$ is an hexagonal piece, $\mathcal{K}_i$ is empty.  
For every $\G_i$, we denote by $K_i$ the support of $\mathcal{K}_i$. The support of $\mathcal{K}$ will be the set
$$K = \overline{\bigcup_{i} K_i}\,. $$

On $K \cap B$, we define the partial foliation by glueing the partial foliations $\mathcal{K}_i$ on the pieces $\G_i$ that are in $B$. To define the partial foliation on $K \setminus B$, we notice that $K \setminus B \subset X_A$. The double $(X_A)^d$ is a finite hyperbolic surface without boundary, hence we can apply Thurston's theory \cite{Thurston}, and consider Thurston's horocyclic foliation on $(X_A)^d$. The set $K \setminus B$ is contained in the support of Thurston's foliation, and we define our foliation on $K \setminus B$ as the restriction of Thurston's foliation. This defines a partial foliation on $X$ whose support is $K$. We will call it the \emph{horocyclic foliation} on $X$, and denote it by $\mathcal{K}$. By definition, for every $\G_i$, the restriction of $\mathcal{K}$ to $\G_i$ coincides with $\mathcal{K}_i$. 

We will now describe how $\mathcal{K}$ looks like locally in the neighborhood of every point. If a point lies in $X \setminus \lambda$, then it is in the interior of a piece $\G_i$. In this case we know that the horocyclic foliation around this point looks like one of the explicit models given in Definition \ref{def:horfol3}, \ref{def:horfol4} or \ref{def:horfol5}. If the point is on $\lambda$, then it lies on a geodesic $\ell \subset \lambda$. For every side of $\ell$, there can be a geometric piece bounded by $\ell$ on that side or not. If there is a geometric piece, then again $\mathcal{K}$ looks like one of the explicit models on that side of $\ell$. If there is no geometric piece on that side of $\ell$, the situation is even simpler, as we now describe. For $z\in \ell$, a small ball centered at $z$ is divided by $\ell$ in two parts, which we call \emph{half-balls}, one on every side of $\ell$. 

\begin{lemma} \label{lemma:foliazione}
Let $z \in \ell \subset \lambda$. If on one side $\ell$ does not bound a geometric piece, then there exists a small half-ball $U$ centered at $z$ on that side of $\ell$ such that $U$ is completely foliated by $\mathcal{K}$ with leaves that hit orthogonally $\ell$ and all the leaves of $U \cap \lambda$.  
\end{lemma}

\begin{proof}
Assume that the radius of $U$ is small so that $U$ does not intersect $\partial X$ nor any leaf of $\lambda$ of finite length. In particular, $U$ does not intersect any hexagonal piece.

Now, let us work in the universal covering $\widetilde{X} \subset \mathbb{H}^2$. We denote by $\widetilde{U}$ a lift of $U$. Every connected component of $\widetilde{X} \setminus \widetilde{\lambda}$ is a geometric piece that is a copy of one of the $\G_i$.   

If we assume that the radius of $U$ is also smaller than $\tfrac{1}{2}\log(3)$ (the radius of the inscribed circle to an ideal triangle), then $\widetilde{U}$ intersects at most two edges of every geometric piece, these two edges meet at an ideal vertex of the piece. There is at most one piece, say $\widetilde{\G}_0$, such that $\widetilde{U}$ intersects only one edge of $\widetilde{\G}_0$, see Figure \ref{fig:lemma foliazione}. 

\begin{figure}[htbp]
\begin{center}
\psfrag{z}{$z$}
\psfrag{U}{$U$}
\psfrag{G}{$\mathcal G_0$}
\psfrag{Q}{$\mathcal \ell$}
\includegraphics[width=4cm]{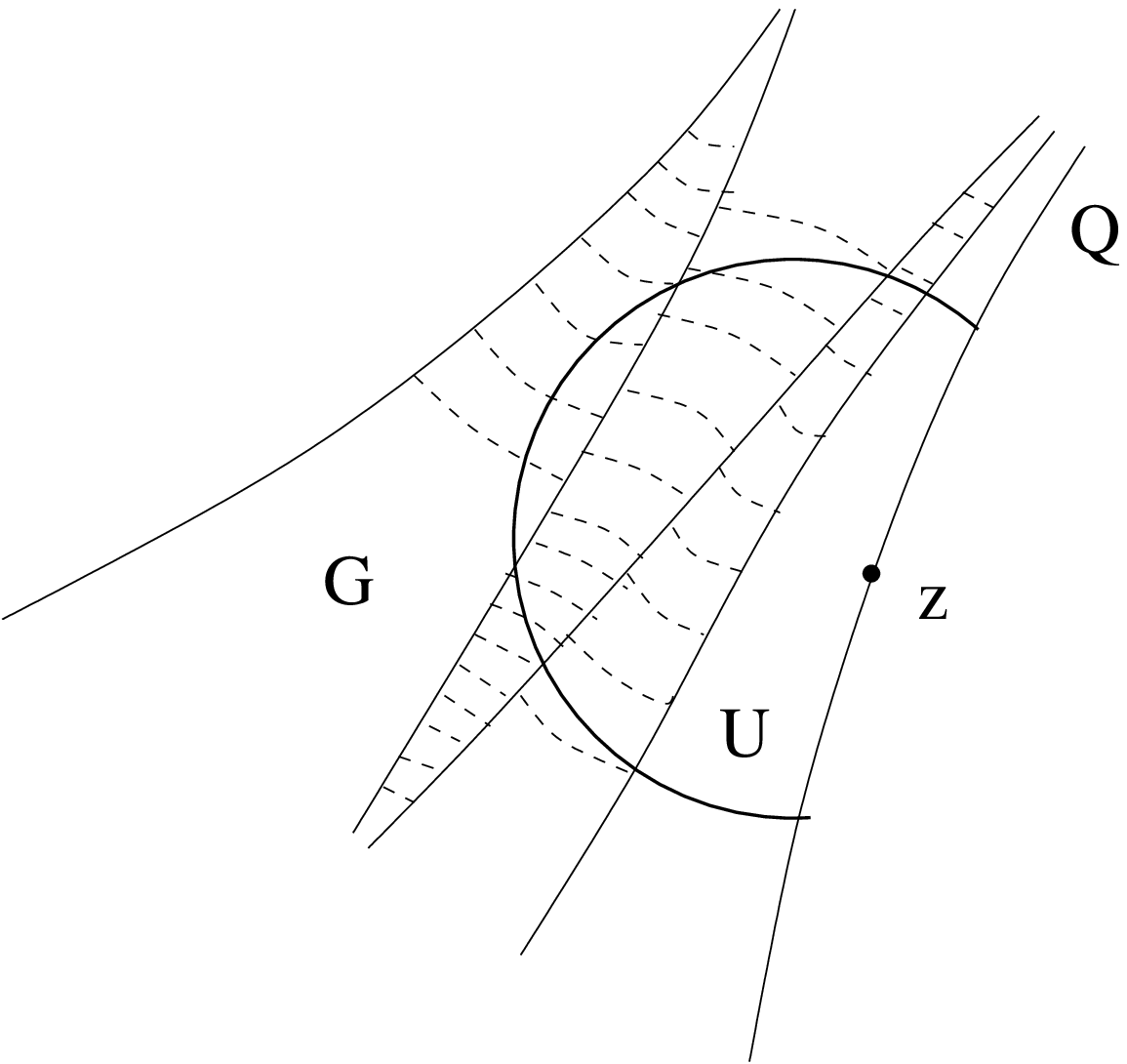}
\caption{Lemma \ref{lemma:foliazione} }  \label{fig:lemma foliazione}
\end{center}
\end{figure}

There are only finitely many $\G_i$, with finitely many values of the parameters. The lengths of the horocycles of the $\mathcal{K}_i$ passing through the points $O_\TT^j$ for a triangular piece, $O_C, O_D$ for a quadrangular piece, $O_W$ for a pentagonal piece have a minimal value $D$. 

We can also assume that the radius of $U$ is smaller than $D/2$, this is the situation represented in Figure \ref{fig:lemma foliazione}. Now if $\widetilde{U}$ meets two edges of a geometric piece, then the intersection of $\widetilde{U}$ with the piece is completely contained in the support of the horocyclic foliation. This might still be false for $\widetilde{\G}_0$, but up to reducing again the radius of $U$, we can make sure that $\widetilde{U}$ does not intersect this piece at all. With this choice of the radius, we have that $U\setminus \lambda$ is completely foliated by $\mathcal{K}$, with a foliation made by pieces of horocycles. This foliation extends nicely to all the leaves of $\lambda$ contained in $U$ and it is perpendicular to them, this is proved in \cite[Section 2]{Bonahon_pleated}.
\end{proof}

The definition of the horocyclic foliation works for every finite hyperbolic surface equipped with a maximal lamination, in particular for $X_\lambda^t$, $(X_A^d)^t$, and $(X_A^d)_{Th}^t$.

\begin{lemma}   \label{lemma:leaves}
Let $z \in \ell$, for a geodesic $\ell \subset \nu_X$. Let $U$ be a half-ball centered at $z$ as in Lemma \ref{lemma:foliazione}. If $x,y \in U\setminus \nu_X$ are in the same leaf for the restriction of $\mathcal{K}$ to $U$, then their images $\alpha^t(x),\alpha^t(y)$ are in the same leaf for the horocyclic foliation on $X_\lambda^t$. 
\end{lemma}     
\begin{proof}
If $U$ satisfies the thesis of Lemma \ref{lemma:foliazione}, it is contained in $X_C$.
Let $f \subset U$ be the leaf for the restriction of $\mathcal{K}$ to $U$ that contains $x,y$. Every component of $f\setminus \lambda$ is contained in a geometric piece, hence it is mapped by $\alpha^t$ to a leaf of $\mathcal{K}$.

Let's first assume that $x,y$ are not in $B$. Recall that $\alpha^t$ on them is defined as $\psi^t=\bar{\kappa}^t \circ \tau^t$. Consider the arc $f_{x,y}$ of $f$ between $x$ and $y$. When $f_{x,y}$ is considered as a subset of $X_A^d$, it lies in a leaf of the horocyclic foliation of $X_A^d$. Thurston's stretch map $\tau^t:X_A^d \rightarrow (X_A^d)^t_{Th}$ sends $f_{x,y}$ in a leaf of the horocyclic foliation of $(X_A^d)^t_{Th}$, hence $\tau^t(x)$ and $\tau^t(y)$ are in the same leaf in $(X_A^d)^t_{Th}$.
We have to check that $\bar{\kappa^t}$ also sends $\tau^t(x)$ and $\tau^t(y)$ to the same leaf in $(X_A^d)^t$. In order to do this, notice that the horocyclic distance between $\bar{\kappa^t}(\tau^t(x))$ and $\bar{\kappa^t}(\tau^t(y))$ is equal to the measure of the arc $\tau^t(f_{x,y})$ for $\epsilon^t$ (see the definition of cocycle associated to a hyperbolic metric in \cite[Section 2]{Bonahon_pleated}). Since $x,y$ do not lie in $B$, the measure of this arc for $\epsilon^t$ is zero by definition of $\epsilon^t$. 

Let us now prove the statement when $x,y$ are in $B$. Notice that every component of $f\setminus \lambda$ is mapped by $\alpha^t$ to a leaf of $\mathcal{K}$. By the continuity of $\alpha^t$ on $X\setminus \nu_X$ (Lemma \ref{lemma:alpha}), every connected component of $f\setminus \nu_X$ is mapped by $\alpha^t$ to a leaf of $\mathcal{K}$. 
\end{proof}

\begin{proposition}   \label{prop:extension}
The map $\alpha^t : X\setminus \nu_X \to X_\lambda^t\setminus \nu_X$ extends to a continuous map $$\Phi^t:X \to X_\lambda^t~.$$
\end{proposition}
\begin{proof}
For every point $z \in \nu_X$, consider a small half-ball $U$ centered in $z$ as in Lemma \ref{lemma:leaves}. If $f$ is the leaf of $\mathcal{K}$ through $z$, then $\alpha^t(f)$ lies on a leaf of the horocyclic foliation on $X_\lambda^t$. We define $\Phi^t(z)$ as the point of $\nu_X \subset X_\lambda^t$ lying on the leaf containing $\alpha^t(f)$. The map $\Phi^t$ maps $U$ homeomorphically to a half-neighborhood of $\Phi^t(z)$ in $X_\lambda^t$.  
\end{proof}

We are ready to prove our main theorem: 

\thurstonstretch*
\begin{proof}
Properties (1)-(5) follow from the construction of $\Phi^t$. For Property (6), from Property (2) we have:
\[d_A(X, X_\lambda^{t}) \leq d_{L\partial}(X, X_\lambda^{t}) \leq t\,.\] 
The arc-length on $\lambda$ is multiplied by $e^{t}$ by Property (4), and since $\lambda$ contains a measurable sublamination we have, by Theorem \ref{thm:sup on laminations}, that $d_A(X, X_\lambda^{t}) \geq t$. This implies 
\begin{equation*}
d_A(X, X_\lambda^{t}) = d_{L\partial}(X, X_\lambda^{t}) = t\,. \qedhere 
\end{equation*}
\end{proof}

%% file: 10-Geometry.tex
\section{The geometry of the arc distance}    \label{sec:geometryT}

In this section we will prove the corollaries of Theorem \ref{theorem:S} stated in the introduction.

\subsection{Stretch lines are geodesics}
 
We will now prove that, if a lamination $\lambda$ contains a measurable sublamination, then a generalized stretch line 
is a geodesic in $\T(S)$ for both the arc distance $d_A$ and the Lipschitz distance $d_{L\partial}$.

\thurstonstretchline*
\begin{proof}
This follows from Theorem \ref{theorem:S}, once we notice that
${\left(X_\lambda^{t_1}\right)}_\lambda^{t_2 - t_1} = X_\lambda^{t_2}\,. $
\end{proof}

\subsection{The Teichm\"uller space is geodesic}  \label{subsec:geodesics}
We will now prove that every pair of points $X,Y \in \T(S)$ is  
connected by a path that is geodesic for both distances $d_A$ and $d_{L\partial}$. The path will be a finite concatenation of generalized stretch segments. Notice that, in a Riemannian manifold, a concatenation of geodesic segments coming from distinct geodesic cannot be a geodesic. These distances are indeed not induced by a Riemannian metric, we will see later that they are instead induced by a Finsler metric. In a Finsler manifold, a geodesic segment might admit several geodesic extensions. 

The first ingredient for the proof is the notion of ratio-maximizing measured lamination. We will first recall Thurston's definition for closed or punctured surfaces and we then extend these notions to surfaces with boundary.
 \begin{definition}[Ratio-maximizing lamination for closed or punctured surfaces \cite{Thurston}]
Let $S$ be a closed or punctured surface. Fix $X,Y \in \T(S)$. A geodesic lamination $\mu$ is a \emph{ratio-maximizing} for $X, Y$ if there exists a homeomorphism $f$ from a neighborhood of $\mu$ in $X$ to a neighborhood of $\mu$ in $Y$ such that
\begin{enumerate}
\item $f$ is $R$-Lipschitz, where $R:= \exp(d_A(X,Y))$.
\item $f$ is homotopic to the identity.
\item $f$ maps the support of $\mu$ in $X$ to the support of $\mu$ in $Y$ stretching the arc-length of $\mu$ affinely by a factor $R$.
\end{enumerate} 
\end{definition}
Thurston \cite{Thurston} proves that for every pair of points $X,Y \in \T(S)$ there exists a unique \emph{largest ratio-maximizing lamination} $\mu(X,Y)$.
\begin{definition}[Ratio-maximizing lamination for surfaces with boundary]
Let $S$ be a surface with boundary and fix $X,Y \in \T(S)$. A geodesic lamination $\mu$ is a \emph{ratio-maximizing lamination} for $X$ and $Y$ if $\mu^d$ is ratio-maximizing for $X^d,Y^d$. Moreover, consider the unique largest ratio-maximizing lamination $\mu(X^d,Y^d)$ in $S^d$. By uniqueness, $\mu(X^d,Y^d)$ is symmetric and restricts to a lamination on $S$ that we denote by $\mu(X,Y)$ and call the \emph{largest ratio-maximizing lamination} for $X, Y \in \T(S)$.  
\end{definition}

\begin{proposition}  \label{prop:ratio of ratio maximizing}
Let $X,Y \in \T(S)$ and let $\mu$ be a measured lamination. Then, the support of $\mu$ is ratio-maximizing for $X,Y$ if and only if $\mu$ realizes the maximum in the formula for $d_A(X,Y)$ given in Theorem \ref{thm:sup on laminations}.
\end{proposition}
\begin{proof}
If $S$ is closed or punctured, $d_A = d_{Th}$ and the result was proven by Thurston \cite{Thurston}. If $S$ has boundary we can conclude by a doubling argument: the support of $\mu^d$ is ratio-maximizing for $X^d, Y^d$ hence it realizes the maximum for $d_{Th}$ and by Proposition \ref{prop:doubling}, $\mu$ realizes the maximum for $d_A$. 
\end{proof}

\begin{lemma}   \label{lemma:ratio-maximizing}
Let $\lambda$ be a maximal lamination, and $X\in \T(S)$. Then for all $t\geq 0,$ 
$$\mu(X,X_\lambda^t) = \lambda,.$$
\end{lemma}  
\begin{proof}
We need to construct a suitable homeomorphism $\phi$ from a neighborhood $M_X$ of $\lambda$ in $X$ to a neighborhood $M_{X_\lambda^t}$ of $\lambda$ in $X_\lambda^t$. We initially describe $\phi$ in every geometric piece. For the triangular pieces, $\phi$ will agree with Thurston's stretch map $\phi^t$. For the other pieces, $\phi$ agrees with our stretch maps $\phi^t$ on the support of the horocyclic foliation $\mathcal{K}$ and on the edges that are leaves of $\lambda$. We can always extend it to a small neighborhood of these edges by a homeomorphism. In the rest of $X$, $\phi$ is defined only on the support of $\mathcal{K}$, and there it agrees with our stretch map $\Phi^t$. 
We know that $\Phi^t$ is a homeomorphism on the support of $\mathcal{K}$ (proof of Proposition \ref{prop:extension}). Using Lemma \ref{lemma:foliazione}, we see that the union of the support of $\mathcal{K}$ with the triangular pieces and with a small neighborhood of the edges of the geometric pieces part of $\lambda$ is a neighborhood of $\lambda$.   
\end{proof}

\begin{lemma}
The  lamination $\mu(X,Y)$ contains a measurable sublamination.
\end{lemma}
\begin{proof}
This is because the maximum in the formula for $d_A(X,Y)$ given in Theorem \ref{thm:sup on laminations} is always achieved by some measurable lamination, whose support is ratio-maximizing by Proposition \ref{prop:ratio of ratio maximizing}, hence contained in $\mu(X,Y)$. 
\end{proof}

The following lemma is a simple adaptation of a result of Thurston.
 
\begin{lemma}\label{lemma_0}
Let $X,Y \in \T(S)$. If $X_i$ and $Y_i$ are sequences of hyperbolic structures converging to $X$ and $Y$, then $\mu(X,Y)$ contains every lamination in the limit set of $\mu(X_i,Y_i)$ in the Hausdorff topology. 
\end{lemma}
\begin{proof}
If $S$ has no boundary, see Thurston \cite[Theorem 8.4]{Thurston}. If $S$ has boundary, it follows from Thurston's result via a doubling argument.  
\end{proof}

Given $X,Y \in \T(S)$, we will now construct a geodesic segment joining them. Our proof goes along similar lines as the proof of Thurston's \cite[Theorem 8.5]{Thurston}. In our case, we need to be more careful because our generalized stretch maps are not known to be homeomorphisms everywhere, in contrast to Thurston's stretch maps. Our maps are known to be homeomorphisms only when restricted to the subset $K$ from Section \ref{subsub:horocyclic}, we already used this fact in Lemma  \ref{lemma:ratio-maximizing}.  
 
The idea of the proof is the following: if $\mu(X,Y)$ is a maximal lamination, we can simply consider the generalized stretch line starting at $X$ with respect to $\mu(X,Y)$, and prove
that it passes through $Y$. If $\mu(X,Y)$ is not maximal, we will first complete it to a maximal lamination $\lambda \supset \mu(X,Y)$, and consider the generalized stretch line starting at $X$ with respect to $\lambda$. This will usually not pass through $Y$, hence we need to stop following this geodesic at some point, and start following another one.

\begin{lemma}\label{lemma_1} 
Let $\lambda$ be a maximal lamination containing $\mu(X,Y)$.
There exists $\epsilon$ such that for every $0< t < \epsilon$ we have: 
\begin{enumerate}
\item $\mu(X_\lambda^t,Y) = \mu(X, Y)$; 
\item $d_A(X_\lambda^t, Y) = d_A(X,Y) - t.$
\end{enumerate}
\end{lemma}
\begin{proof}
Let $\mu$ denote $\mu(X,Y)$. By definition of $\mu$, there exist neighborhoods $N_X, N_Y$ of $\mu$ in $X$ and $Y$ respectively and a Lipschitz homeomorphism $f:N_X \to N_Y$ with $\mathrm{Lip}(f)=e^{d_{A}(X,Y)}$ mapping $\mu$ to itself and stretching its arc length affinely  by $e^{d_A(X,Y)}$. 

By Lemma \ref{lemma:ratio-maximizing}, there exists two neighborhoods $M_X, M_{X_\lambda^t}$ of $\mu$ in $X$, 
$X_\lambda^t$ respectively and a homeomorphism $\phi: M_X \to M_{X_\lambda^t}$ with $\mathrm{Lip}(\phi)=e^t$ and $\phi$ maps $\mu$ to itself by affinely stretching it by $e^t$. The composition 
\[f'= f \circ \phi^{-1}: M_{X_\lambda^t} \to N_Y\] 
has $\mathrm{Lip}(f') = e^{d_{A}(X,Y) - t}$ and maps $\mu$ to itself affinely stretching by $e^{d_A(X,Y) - t }$.

By Lemma \ref{lemma_0} there exists $\epsilon$ such that if $0 < t < \epsilon$ then $\mu(X_\lambda^t, Y ) \subset N_X$. Since $\mathrm{Lip}(f') = e^{d_A(X,Y) - t}$, we have 
\[d_A(X_\lambda^t, Y) \leq d_A(X,Y) - t\,.\]  
On the other end, by the triangle inequality we have  
$d_A(X_\lambda^t, Y) \geq d_A(X,Y) - t$. 
We thus have 
\[d_A(X_\lambda^t, Y) = d_A(X,Y) - t\,.\] 
This implies that $\mu$ is ratio-maximizing for $X_\lambda^t$ and $Y$. If we choose $N_X$ to be small enough, all other laminations in this neighborhood must intersect $\mu$. We saw that $\mu(X_\lambda^t, Y ) \subset N_X$  and this implies 
\begin{equation*}
\mu(X_\lambda^t, Y)= \mu\,. \qedhere
\end{equation*}
\end{proof}

Let $\overline{t} := \overline{t}_{X,Y,\lambda}$ be the supremum of the $\epsilon$'s as in Lemma \ref{lemma_1}. If $\overline{t} = d_A(X,Y)$, Lemma \ref{lemma_1} gives a geodesic segment joining $X$ and $Y$. Otherwise, we will need the following: 

\begin{lemma}\label{lemma_2}
If $\overline{t} < d_A(X,Y)$ then 
$\mu(X, Y) \subsetneq \mu(X_\lambda^{\overline{t}}, Y).$
\end{lemma}
\begin{proof}
Let $\{t_n\} $ be a sequence of positive numbers such that  $t_n \nearrow \overline{t}$. By Lemma \ref{lemma_1} $\mu(X_\lambda^{t_n}, Y) = \mu(X, Y)$ for every $n$. Now, Lemma \ref{lemma_0} says that $\mu(X_\lambda^{t_n}, Y) \subseteq \mu(X_\lambda^{\overline{t}},Y)$. 

By contradiction assume that $\mu(X, Y) = \mu(X_\lambda^{\overline{t}},Y)$. Applying Lemma \ref{lemma_1} on $X_\lambda^{\overline{t}}$ and $Y$, we find values bigger than $\overline{t}$ satisfying the same properties. 
\end{proof}

\geodesicspace*
\begin{proof}
We define inductively a sequence of hyperbolic structures $X_0, X_1, \dots, X_k$ in the following way. We set $X_0 = X$. Now assume that $X_i$, $i\geq 0$ has been defined. Choose a maximal lamination  $\lambda_i$  that contains $\mu(X_i,Y)$. Consider the generalized stretch line $(X_i)_{\lambda_i}^t$, and compute $\overline{t}_i:=\overline{t}_{X_i,Y,\lambda_i}$ as defined after Lemma \ref{lemma_1}. If  $\overline{t}_i < d_A(X_i,Y)$, we set $X_{i+1} = (X_i)_{\lambda_i}^{\overline{t}_i}$. If $\overline{t}_i \geq d_A(X_i,Y)$, this implies that $Y$ lies on the generalized stretch line $(X_i)_{\lambda_i}^t$, in this case we set $k=i$ and we stop. 

\begin{figure}[htbp] 
\begin{center}
\psfrag{A}{\hspace{-0.3cm} \small $(X_i)_{\lambda_i}^{\overline{t}_i}=X_{i+1}~$~ }
\psfrag{B}{$ $}
\psfrag{C}{\small $(X_{i+1})_{\lambda_{i+i}}^{t}$}
\psfrag{D}{\small $(X_{i+1})_{\lambda_{i+i}}^{\overline{t}_{i+1}}= X_{i+2}$}
\includegraphics[width=7cm]{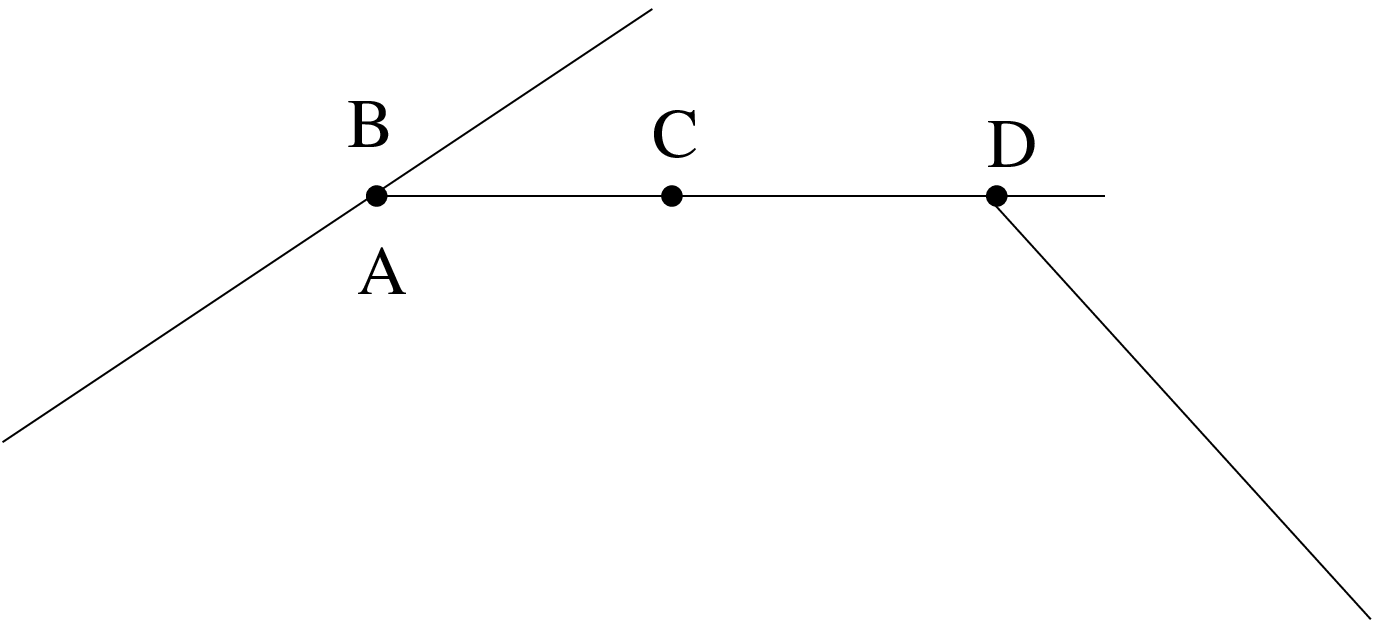}
\caption{Lemma \ref{lemma_2}}\label{thurston:stretch}
\end{center}
\end{figure}

This defines the sequence of the $X_i$'s. Notice that $\mu(X_i,Y) \subsetneq \mu(X_{i+1},Y)$, so we have a strictly increasing chain of geodesic laminations. This implies $k \leq 2|\chi(S)|$. 
We have found a finite sequence  of concatenated geodesic segments
$$ t \rightarrow (X_i)_{\lambda_i}^{t}, \text{ for } 0 \leq t \leq \overline{t}_i$$ 
such that $X$ lies in the first one, and $Y$ in the last one. 
By Lemma \ref{lemma_1}, we have $d_A(X,X_i) + d_A(X_i,Y) = d_A(X,Y)$, hence this concatenation of segments is geodesic.  
\end{proof}

\distancescoincide*
\begin{proof}
Consider the geodesic segment joining $X$ and $Y$ in the proof of Theorem \ref{thm:geodesic space}: it passes through the points $X = X_0, X_1, \dots, X_k = Y$, where $X_{i+1} = (X_i)_{\lambda_i}^{\overline{t}_i}$. Since $X_i$ and $X_{i+1}$ are on the same stretch line, by Theorem \ref{theorem:S} we have a map $\Phi_{i}^{\overline{t}_i}:X_i \rightarrow X_{i+1}$ with $\Lip(\Phi_i^{\overline{t}_i}) = e^{\overline{t}_i}$ and $\Phi_i^{\overline{t}_i}(\partial X_i) \subset \partial X_{i+1}$. Consider the composition:
$$\phi = \Phi_{k-1}^{\overline{t}_{k-1}} \circ \dots \circ \Phi_0^{\overline{t}_0} : X \rightarrow Y \,,$$
which satisfies $\phi(\partial X) \subset \partial Y$. The Lipschitz constant of a composition is bounded by the product of the constants:
$$\Lip(\phi) \leq \prod_i e^{\overline{t}_i} = e^{\sum_i \overline{t}_i} = e^{d_A(X,Y)}\,.$$ 
We know that
$d_A(X,Y) \leq d_{L\partial}(X,Y) \leq \log(\Lip(\phi)) \leq d_A(X,Y)\,. $
Hence $\log(\Lip(\phi)) = d_A(X,Y)$ and $d_A(X,Y) = d_{L\partial}(X,Y)$. 
\end{proof}

\subsection{Geodesics in the Teichm\"uller space of the double}

A \emph{geodesic embedding} between two metric spaces $f: (\Omega, d) \to (\Omega', d')$ is an isometric embedding such that for every pair of points $P,Q \in f(\Omega)$ there exists a geodesic for $d'$ that joins them and it is contained in $f(\Omega)$. The following is a consequence of Theorem \ref{thm:geodesic space} and Proposition \ref{prop:doubling}. 
  
\geodesicembedding*

\newgeodesics*
\begin{proof}
This line is a geodesic by Proposition \ref{prop:doubling}. Notice that it stretches the length of the lamination $\lambda^d$ by a factor $e^t$. If it were a stretch line in the sense of Thurston, it would be directed by a maximal lamination that contains $\lambda^d$. Since $\lambda$ contains a leaf orthogonal to $\partial X$, every extension of $\lambda^d$ to a maximal lamination in $X^d$ is not symmetric, and the corresponding stretch line does not lie in the submanifold of symmetric hyperbolic structures (see Th\'eret \cite{TheretThese}).   
\end{proof}

In this way we find infinitely many examples of new geodesics for the Teichm\"uller spaces of surfaces  without boundary that are not stretch lines in the sense of Thurston. 

\subsection{The Finsler metric}

A \emph{Finsler metric} on a smooth manifold $M$ is a continuous function
$$F:TM \ni (x,v) \rightarrow F_x(v) \in \R_{\geq 0} $$
which is a (possibly asymmetric) norm on the tangent space $T_x M$ at every point $x \in M$. In a Finsler manifold the length of a smooth 
curve $\gamma : [a,b] \to M$ is given by the formula
$$\ell(\gamma) := \int_a^b  F_{\gamma(t)}(\dot{\gamma}(t)) dt,$$
and the (possibly asymmetric) distance induced by a Finsler metric is defined as
$$d_F(x,y) = \inf_{\gamma} \ell(\gamma), $$
where the infimum is taken over all the smooth curves joining $x$ and $y$.

\finslerspace*
\begin{proof}
Consider the map $\T(S) \hookrightarrow \T(S^d)$ as in Corollary \ref{cor:geodesic}. By \cite{Thurston}, the space $(\T(S^d),d_{Th})$ is a Finsler manifold. The space $\T(S)$ can be identified with a  
submanifold of $\T(S^d)$, and naturally inherits the Finsler metric by restriction. Now let's prove that the distance induced is the same as the distance $d_A$. Let $X, Y$ be two points in $\T(S)$, we proved 
that $d_A(X,Y)$ is the same as the length of a geodesic segment joining them. By Corollary \ref{cor:geodesic}, the length of any geodesic segment in $\T(S)$ is the same as in $\T(S^d)$. This in turn equals the length of the curve with respect to the Finsler norm, because the Finsler norm induces the distance $d_{Th}$ on $\T(S^d)$. 
\end{proof}